\documentclass[a4ppaper,11pt]{amsart}

\usepackage{enumerate}
\usepackage[latin1]{inputenc}
\usepackage{amsmath}
\usepackage{amsthm}
\usepackage{latexsym}
\usepackage{amssymb}
\usepackage{amsmath}
\usepackage{comment}
\usepackage[all]{xy}
\usepackage{calc}
\usepackage{multicol}
\usepackage{abstract, lipsum}
\usepackage{etoolbox}
\usepackage{titlesec}
\usepackage{secdot}

\makeatletter
\patchcmd{\@settitle}{\uppercasenonmath\@title}{}{}{}
\patchcmd{\@setauthors}{\MakeUppercase}{}{}{}
\patchcmd{\section}{\scshape}{}{}{}
\makeatother

\usepackage{titlesec}
\titleformat{\section}
  {\large\bfseries}{\thesection}{1em}{}

\usepackage[textwidth=138mm,
  textheight=215mm,
  left=30mm,
  right=30mm,
  top=25.4mm,
  bottom=25.4mm,
  headheight=2.17cm,
  headsep=4mm,
  footskip=12mm,
  heightrounded,]{geometry}

\usepackage[colorlinks=true,urlcolor=blue,
citecolor=black,linkcolor=black,linktocpage,pdfpagelabels,
bookmarksnumbered,bookmarksopen]{hyperref}
\usepackage[hyperpageref]{backref}

\newtheorem{thm}{Theorem}[section]
\newtheorem{prop}{Proposition}[section]
\newtheorem{defi}{Definition}[section]
\newtheorem{lem}{Lemma}[section]

\newtheorem{rem}{Remark}[section]
\newtheorem*{notation}{Notation}
\theoremstyle{definition}
\newtheorem*{result}{Statement of main results}
\newtheorem*{structure}{Structure of the paper}
\newtheorem*{ack}{Acknowledgements}

\newcommand{\R}{\mathbb{R}}
\numberwithin{equation}{section}
\newcommand{\N}{\mathbb{N}}

\newcommand{\eps}{\varepsilon}

\makeatletter
\newcommand{\leqnomode}{\tagsleft@true}
\newcommand{\reqnomode}{\tagsleft@false}
\makeatother

\begin{document}

\title[Normalized solutions to the CSS system]{ \Large Normalized solutions to the Chern-Simons-Schr\"odinger system}\thanks{*Corresponding author}

\author{
    {{\normalsize{Tianxiang Gou}}$^{a}$, {\normalsize{Zhitao Zhang}}$^{a, b,*}$}\bigskip \\
    $^{a}$ {\it HLM, CEMS, Academy of Mathematics and Systems Science, Chinese Academy of Science, \\
    Beijing 100190, China} \smallskip \\
    $^{b}$ {\it School of Mathematical Sciences, University of Chinese Academy of Sciences,\\
    Beijing 100049, China} \smallskip \\
    {\it E-mail addresses}: tianxiang.gou@amss.ac.cn (T. Gou), zzt@math.ac.cn (Z. Zhang)}

\maketitle
\begin{abstract}

\vspace{0.2cm}

In this paper, we study normalized solutions to the Chern-Simons-Schr\"odinger system, which is a gauge-covariant nonlinear Schr\"odinger system with a long-range electromagnetic field, arising in nonrelativistic quantum mechanics theory. The solutions correspond to critical points of the underlying energy functional subject to the $L^2$-norm constraint. Our research covers several aspects. Firstly, in the mass subcritical case, we establish the compactness of any minimizing sequence to the associated global minimization problem. As a by-product of the compactness of any minimizing sequence, orbital stability of the set of minimizers to the minimization problem is achieved. In addition, we discuss the radial symmetry and uniqueness of minimizer to the minimization problem. Secondly, in the mass critical case, we investigate the existence and nonexistence of normalized solutions. Finally, in the mass supercritical case, we prove the existence of ground state and infinitely many radially symmetric solutions. Moreover, orbital instability of ground states is explored. \medskip

{\it Keywords:} Chern-Simons-Schr\"odinger system; Normalized solutions; Standing waves; Orbital instability; Variational methods. \medskip

{\it MSC:}  35J20, 35J61,35B06, 35B35, 35B38, 35B40.

\end{abstract}

\section{Introduction}

\reqnomode

In this paper, we are concerned with the following time-dependent Chern-Simons-Schr\"odinger system in two spatial dimension,
\begin{align} \label{sys1}
\left\{
\begin{aligned} 
i D_t \varphi + \left(D_1 D_1 +D_2 D_2\right)\varphi &=-\lambda |\varphi|^{p-2}\varphi,\\ 
\partial_t A_1-\partial_1 A_0&=-\mbox{Im}\left(\bar{\varphi} \,D_2 \varphi\right), \\ 
\partial_t A_2-\partial_2 A_0&=\mbox{Im}\left(\bar{\varphi} \,D_1 \varphi\right), \\  
\partial_1 A_2-\partial_2 A_1&=- 1/2 \,|\varphi|^2,
\end{aligned}
\right.
\end{align}
where $i$ denotes the imaginary unit, $\varphi: \R \times \R^2 \to \mathbb{C}$ is complex scalar field, $\lambda >0$ is a coupling constant representing the strength of interaction potential, $D_t, D_{j}$ are covariant derivative operators defined by
$$
D_{t}:=\partial_{t} + i A_{0}, \quad D_{j}:=\partial_{j} + i A_{j},
$$
and $A_{0}, A_j : \R \times \R^2 \to \R$ are gauge fields for $j=1, 2$. The system \eqref{sys1} as a gauge-covariant nonlinear Schr\"odinger system is a nonrelativistic quantum model. It is used to describe the dynamics of a large number of particles in the plane interacting both directly and via a self-generated field. The system \eqref{sys1} plays an important role in the search of the high temperature superconductivity, fractional quantum Hall effect and Aharovnov-Bohm scattering, see \cite{EHI, EHI1, JaPi, JaPi1, JaPi2} for more physical information.

Observe that the system \eqref{sys1} is invariant under the following gauge transformation,
$$
\varphi \to \varphi e^{i \chi}, \quad A_0 \to A_0-\partial_t \chi, \quad A_{j} \to A_{j}-\partial_{j} \chi,
$$
where $\chi: \R \times \R^2 \to \R$ is a smooth function and $j=1, 2$. This means that if $(\varphi, A_0, A_1, A_2)$ is a solution to the system \eqref{sys1}, so is
$$
(\varphi e^{i \chi}, A_0-\partial_{t} \chi, A_1-\partial_{1} \chi, A_2-\partial_{2} \chi).
$$
For this reason, in order that the evolution of the system \eqref{sys1} is well-defined, it is required to fix the gauges. To do this, we shall impose the Coulomb gauge condition
\begin{align} \label{gauge}
\partial_1 A_1 + \partial_2 A_2 =0.
\end{align}
With the Coulomb gauge condition \eqref{gauge}, the system \eqref{sys1} leads to
\begin{align} \label{defA}
\left\{
\begin{aligned}
-\Delta A_1&=-1/2 \,\partial_2\left(|\varphi|^2 \right), \\
-\Delta A_2&=1/2 \, \partial_1\left(|\varphi|^2 \right), \\
-\Delta A_0&=\mbox{Im}\left(\partial_2 \bar{\varphi} \,\partial_1 \varphi- \partial_1 \bar{\varphi} \, \partial_2\varphi\right) -\partial_1\left(A_2|\varphi|^2 \right) +\partial_2\left(A_1|\varphi|^2 \right).
\end{aligned}
\right.
\end{align}
In Coulomb gauge, we see that $A_1, A_2, A_0$ are determined in terms of $\varphi$ by solving the elliptic system \eqref{defA}. This then implies that the problem under consideration admits a proper variational structure and variational methods come into play in our survey. Let us also mention another choice to fix the gauges, which consists in introducing the heat gauge condition
$$
A_0=\partial_1A_1 + \partial_2 A_2.
$$
In heat gauge, $A_1, A_2, A_0$ are determined in terms of $\varphi$ by solving heat equations, see \cite{LiSmTa}. This shows that it seems impossible to investigate \eqref{sys1} from variational perspectives. The research of the system \eqref{sys1} under the heat gauge condition is beyond the scope of this paper, we refer the readers to \cite{Dem} for more relevant information.

The intention of this paper is to consider standing waves to the system \eqref{sys1}, which are solutions of the form
$$
\varphi(t, x)= e^{i\alpha t} u(x),
$$
where the frequency $\alpha \in \R$ and $u$ is complex-valued. In light of \eqref{defA}, this ansatz then gives rise to the following stationary system satisfied by $u$,
\begin{align} \label{system1}
\left\{
\begin{aligned}
-\left(D_1D_1 +D_2D_2\right) u + A_0 u + \alpha u &=\lambda |u|^{p-2}u, \\
\partial_1 A_0&= \mbox{Im} \left( \overline{u} \,D_2 u \right),\\
\partial_2 A_0&=-\mbox{Im}\left(\overline{u} \,D_1 u\right), \\
\partial_1 A_2-\partial_2 A_1&=-1/2\, |u|^2.
\end{aligned}
\right.
\end{align}
In view of the definitions of $D_1$ and $D_2$, it follows from the system \eqref{defA} that
\begin{align} \label{AA} \tag{A}
\left\{
\begin{aligned}
A_1&:=A_1(u)= -\frac 12\,G_2 * |u|^2, \,\,\, A_2:=A_2(u)=\frac 12\,G_1 * |u|^2, \\
A_0&:=A_0(u)=-G_1* \left(\mbox{Im} \left( \overline{u} \,D_2 u \right)\right)+G_2* \left(\mbox{Im}\left(\overline{u} \,D_1 u\right)\right),
\end{aligned}
\right.
\end{align}
where the symbol $*$ represents the convolution of two functions in $\R^2$, and the convolution kernels $G_j$ are defined by
$$
G_j(x):=-\frac{1}{2 \pi} \frac{x_j}{|x|^2} \,\,\, \mbox{for} \, \, \, j=1,2.
$$

In order to explore solutions to the system \eqref{system1}, we would like to mention two substantially distinct options in terms of the frequency $\alpha$. The first one is to fix the frequency $\alpha \in \R$. In this situation, any solution to the system \eqref{system1} corresponds to a critical point of the energy function $J$ on $H^1(\R^2)$, where
$$
J(u):=\frac 12 \int_{\R^2} |D_1 u|^2 + |D_2 u|^2\, dx + \frac{\alpha}{2} \int_{\R^2}|u|^2 \, dx - \frac {\lambda}{p} \int_{\R^2}|u|^p \, dx.
$$
As far as we know, there already exist many works towards this direction. Assuming $\alpha >0$ and imposing the following gauge fields,
\begin{align} \label{gc} \tag{${A}_{r}$}
A_0=k(|x|), \quad A_1=\frac{x_2}{|x|^2} h(|x|), \quad  A_2=-\frac{x_1}{|x|^2} h(|x|),
\end{align}
where $k, h$ are real-valued, and $h(0)=0$, the authors in \cite{BHS} dealt with the existence and nonexistence of solutions to the system \eqref{system1} for any $p>2$. Successively, the author in \cite{huh} established the existence of infinitely many solutions to the system \eqref{system1} for any $p>6$, and the authors in \cite{PoRu} considered the existence and nonexistence of positive solutions for any $2<p<4$. In \cite{CdPS}, the system \eqref{system1} equipped with a general nonlinearity was investigated. For more related research, the readers can refer to \cite{BHS1, HuSe1, JPR, PoRu1, WaTa, WaTa1, WaTa2} and references therein.

Alternatively, it is of great interest to study solutions to the system \eqref{system1} having prescribed $L^2$-norm, namely, for any given $c>0$, to consider solutions to the system \eqref{system1} under the $L^2$-norm constraint
\begin{align} \label{mass}
S(c):=\{u \in H^1(\R^2): \int_{\R^2}|u|^2 \, dx =c\}.
\end{align}
Physically, such solutions are so-called normalized solutions to the system \eqref{system1}, which formally correspond to critical points of the energy functional $E$ restricted on $S(c)$, where
\begin{align*}
E(u):=\frac 12 \int_{\R^2} |D_1 u|^2 + |D_2 u|^2\, dx - \frac {\lambda}{p} \int_{\R^2}|u|^p \, dx.
\end{align*}
It is worth pointing out that, in this situation, the frequency $\alpha \in \R$ is an unknown part, which is determined as the Lagrange multiplier associated to the constraint $S(c)$.

From a physical point of view, it is quite meaningful to explore normalized solutions to the system \eqref{system1}. This is not only because $L^2$-norm of solution to the Cauchy problem of the system \eqref{sys1} is conserved along time, that is, for any $t \in \R$,
$$
\int_{\R^2} |\varphi(x, t)|^2 \, dx =\int_{\R^2} |\varphi(x, 0)|^2 \, dx,
$$
see Lemma \ref{wellposed}, but also because it can provide a good insight of the dynamical properties (orbital stability and instability) of solutions to the system \eqref{system1}. Consequently, the aim of the present paper is to consider normalized solutions to the system \eqref{system1}.

For further clarification, we agree that mass subcritical case, mass critical case and mass supercritical case mean that $2<p<4$, $p=4$ and $p>4$, respectively.

Regarding the study of normalized solutions to the system \eqref{system1}, to the best of our knowledge, there are quite few results. So far, we are only aware of two articles \cite{LiLuo,Yuan} in this direstion, where the authors principally concerned the existence of real-valued normalized solutions to the system \eqref{system1} in the mass supercritical case. Let us emphasise that both of the studies were carried out in the radially symmetric context and under the gauge fields \eqref{gc}. It seems that normalized solutions to the system \eqref{system1} are far from being well understood. Hence, in the present paper, we shall more completely and deeply investigate normalized solutions to the system \eqref{system1} in the mass subcritical case, in the mass critical case as well as in the mass supercritical case. Our research, which is performed in more general setting, covers several interesting aspects.

We now highlight a few features of our survey. Firstly, we consider complex-valued solutions to \eqref{system1}-\eqref{mass} under the more natural gauge fields \eqref{AA}, which is quite interesting in physics and lays a foundation to stduy dynamical behaviors of solutions to the Cauchy problem of the system \eqref{sys1}. In fact, when one is interested in real-valued solutions to \eqref{system1}-\eqref{mass}, then the underlying energy functional $E$ is reduced to the following version,
$$
E(u)=\frac 12 \int_{\R^2}|\nabla u|^2 + \left(A_1^2+A_2^2\right) |u|^2 \,dx - \frac {\lambda}{p} \int_{\R^2}|u|^p \, dx.
$$
In this context, the problem under consideration is somewhat simplified. Secondly, we explore the existence of ground states to \eqref{system1}-\eqref{mass} in $H^1(\R^2)$, instead of in the radially symmetric functions subspace $H_{rad}^1(\R^2)$. Note that the embedding $H^1(\R^2)\hookrightarrow L^t(\R^2)$ for any $t\geq 2$ is only continuous, which makes more hard to check the compactness of the associated sequences in our situation. To overcome the difficulty of lack of compactness, some fresh arguments are proposed in this paper. Thirdly, we also focus on the study of the radial symmetry, uniqueness and dynamical behaviors of solutions to \eqref{system1}-\eqref{mass}, which is new as far as we know.

\begin{rem}
Let us mention that some strategies we use to discuss normalized solutions to the system \eqref{system1} are robust, which may be applicable to consider normalized solutions to other Schr\"odinger-type problems.
\end{rem}

In contrast with the search of solutions to the system \eqref{system1} when the frequency $\alpha \in \R$ is a prior given, the search of normalized solutions becomes more complex and involved. In the normalized setting, when $p>4$, the boundedness of any Palais-Smale sequence in $H^1(\R^2)$ is not guaranteed. However, for the unconstrained issues, when $p>4$, the boundedness of any Palais-Smale sequence is easy to achieve. Furthermore, since the parameter $\alpha$ as the associated Lagrange multiplier is an unknown part, then we have to control the sign of $\alpha$ in order to derive the compactness of the associated Palais-Smale sequences, which is pivotal but delicate to handle. For the same reason, the approach of Nehari manifold, which plays a crucial role in treating unconstrained issues, is unavailable, thus we need to rely on the corresponding Pohozaev manifold to investigate normalized solutions in the mass supercritical case. As we shall see, the verification of the assertion that the Pohozaev manifold is a natural constraint is more complicated in our settting.

Before stating our main results, let us provide the readers with some references regarding the study of normalized solutions to Schr\"odinger-type equations and systems. In the mass subcritical, following the pioneer works \cite{Li1, Li2}, there exist numerous papers considering the compactness of the associated minimizing sequences, for instance \cite{AlBh, CCWei, CDSS, CaLi, TGou, GoJe, NW1, NW2, NW3, Shibata} and references therein. In the mass critical case, many researchers mainly concerned concentration of normalized solutions, for instance \cite{GLWZ1, GLWZ2, GS, GZZ1, GZZ2}. In the mass supercritical case, starting with the early work \cite{Jeanjean}, the study of normalized solutions has received much attention recently. For example, see \cite{AcWe, BaVa, BaJean, BeJeLu, BCGJ, CiJe, JeLuWa, So1,So2} for normalized solutions to equations in the whole space $\R^N$, see \cite{BaJe,BaJeSo, BaSo1, BaSo} for normalized solutions to systems set in $\R^N$, and see \cite{NoTaVe, NoTaVe1, NoTaVe2, PiVe} for normalized solutions to equations and systems confined to bounded domains in $\R^N$.

\begin{result}
Above all, for any $t>0$ and $u \in S(c)$, we introduce a scaling of $u$ as $u_t(x):=tu(tx)$ for $x \in \R^N$. Let us remind that such a scaling is rather useful, which will be frequently used throughout the paper. By simple calculations, see Lemma \ref{escaling}, then $u_t \in S(c)$ and
\begin{align} \label{scaling}
E(u_t)= \frac {t^2}{2} \int_{\R^2}|D_1 u|^2 + |D_2 u|^2 \, dx - \frac{\lambda t^{p-2}}{p} \int_{\R^2}|u|^p \, dx.
\end{align}

Firstly, we shall study solutions to \eqref{system1}-\eqref{mass} in the mass subcritical case. In this case, as a consequence of the Gagliardo-Nirenberg type inequality \eqref{MGN}, we have that the energy functional $E$ restricted on the constraint $S(c)$ is bounded from below. This then leads to the following global minimization problem,
\begin{align} \label{gmin}
m(c):=\inf_{u \in S(c)}E(u).
\end{align}
Since $2<p<4$, from \eqref{scaling}, then $E(u_t)<0$ for any $t>0$ small enough, this infers that $m(c)<0$ for any $c>0$. Moreover, it is clear that every minimizer to \eqref{gmin} is a solution to \eqref{system1}-\eqref{mass}. In this case, we first have the following compactness result.

\begin{thm}\label{compactness}
Assume $2 <p<4$, then there exists a constant $c_0>0$ such that, for any $0 < c <c_0$, every minimizing sequence to \eqref{gmin} is compact in $H^1(\R^2)$ up to translations. In particular, for any $0<c<c_0$, \eqref{system1}-\eqref{mass} admits a solution.
\end{thm}

In order to achieve the compactness of every minimizing sequence to \eqref{gmin} in $H^1(\R^2)$, we shall apply the well-known Lions concentration compactness principle \cite{Li1, Li2}, it suffices to rule out vanishing and dichotomy. Note that $2<p<4$ and $m(c)<0$ for any $c>0$, hence vanishing can be easily excluded via the Lions concentration compactness Lemma \cite[Lemma I.1]{Li2}. So as to exclude dichotomy, we need to establish the following strict subadditivity inequality,
\begin{align} \label{strictsub1}
m(c) < m(c_1)+m(c_2)
\end{align}
for any $0<c_1, c_2<c$, and $c_1+c_2=c$. However, in our scenario, scaling techniques are no longer enough to verify \eqref{strictsub1}, because of the presence of the nonlocal terms $\int_{\R^2} \left(A_1^2+A_2^2\right) |u|^2 \,dx$ and $\mbox{Im} \int_{\R^2}\left( A_1 \partial_1 u + A_2 \partial_2 u\right)\overline{u}\,dx$. In such a situation, we shall adopt some ideas developed in \cite{CDSS} to prove that \eqref{strictsub1} is valid for any $c>0$ small. As we shall see, the proof does not rely on perturbative arguments and it is highly nontrivial.

\begin{rem}
Since the possibility of dichotomy is delicate to deal with, then, for any $c>0$ large, the compactness of every minimizing sequence to \eqref{gmin} in $H^1(\R^2)$ up to translations remains open.  In addition, for any $c>0$ large, the existence of minimizers to \eqref{gmin} is unknown so far.
\end{rem}

As a direct application of Theorem \ref{compactness}, by using the elements in \cite{CaLi}, we are able to derive orbital stability of the set of minimizers to \eqref{gmin}. For any $0<c<c_0$, let us define the set
$$
G(c):=\{u \in S(c): E(u)=m(c)\}.
$$
By Theorem \ref{compactness}, we know that $G(c) \neq \emptyset$. More precisely, we have the following result.

\begin{thm} \label{stable}
Assume $2 <p<4$, then, for any $0<c<c_0$, the set $G(c)$ is orbitally stable, namely, for any $\eps>0$, there exists a constant $\delta>0$ so that if $\varphi_0 \in H^1(\R^2)$ satisfies
\begin{equation*}
\inf_{u\in G(c)} \|\varphi_0 - u\|\leq \delta,
\end{equation*}
then
\begin{equation*}
\sup_{t \in [0, \infty)} \inf_{u\in G(c)}
\|\varphi(t)-u\| \leq \eps,
\end{equation*}
where $\varphi \in C([0, \infty); H^1(\R^2))$ is the solution to the Cauchy problem of the system \eqref{sys1} with $\varphi(0)=\varphi_0$, and $\|\cdot\|$ denotes the standard norm in $H^1(\R^2)$.
\end{thm}

\begin{rem}
When $2<p<4$, the global well-posedness to the Cauchy problem of the system \eqref{sys1} in $H^1(\R^2)$ is addressed in Theorem \ref{globalexis}.
\end{rem}

Additionally, we are interested in the radial symmetry and uniqueness of minimizer to \eqref{gmin}. For this subject, we obtain the following statement.

\begin{thm} \label{symmetry}
Assume $2<p<4$, then there exists a constant $\tilde{c}>0$ with $0<\tilde{c} \leq c_0$ such that, for any $0<c < \tilde{c}$, every minimizer to \eqref{gmin} is radially symmetric up to translations, and minimizer to \eqref{gmin} is unique up to translations, where the constant $c_0$ is given by Theorem \ref{compactness}.
\end{thm}

Inspired by \cite{GPV}, the proof of this theorem relies in an essential way on the implicit function theorem.

\begin{rem}
Motivated by \cite{BeGh}, for any $c>0$ large, we conjecture that a symmetry breaking phenomenon to \eqref{gmin} may occur provided that \eqref{gmin} admits a minimizer. To our knowledge, this is a tough problem, even if in the framework of real-valued functions Sobolev space.
\end{rem}

Secondly, we shall consider solutions to \eqref{system1}-\eqref{mass} in the mass critical case. In this case, the regime of the parameter $\lambda$ plays an important role in the existence of solutions to \eqref{system1}-\eqref{mass}. Our main result reads as follows.

\begin{thm}\label{nonexistence}
Assume $p=4$.
\begin{itemize}
\item [(i)] If $\lambda<1$, then $m(c)=0$ for any $c >0$, and $m(c)$ cannot be attained for any $c >0$. Furthermore, \eqref{system1}-\eqref{mass} has no solution for any $c >0$.
\item [(ii)] If $\lambda =1$, then $m(c)=0$ for any $c >0$, and $m(c)$ is attained only for $c=8 \pi$ and every minimizer $u$ has the form
\begin{align} \label{u}
u(x)=\frac{4 \sqrt{2} \mu}{4 + \mu^2|x-x_0|^2},
\end{align}
where $\mu>0$ and $ x_0 \in \R^2$. Furthermore, \eqref{system1}-\eqref{mass} admits a solution if and only if $c=8\pi$.
\item [(iii)] If $\lambda >1$, then there exists a constant $c^* >0$ such that $m(c)=0$ for any $0<c \leq c^*$, and $m(c)$ cannot be attained for any $0<c<c^*$, it is attained for $c=c^*$, where the constant $c^*$ is defined by
\begin{align} \label{mins}
c^*:= \inf_{u \in \mathcal{P}} \int_{\R^2} |u|^2 \, dx,  \,\,\, \mbox{and}\,\,\, \mathcal{P}:=\{u \in H^1(\R^2)\backslash\{0\} : E(u)=0\}.
\end{align}
Furthermore, \eqref{system1}-\eqref{mass} has no solution for any $0<c < c^*$, and it has a solution for $c=c^*$.
\end{itemize}
\end{thm}

In this case, the authors in \cite{LiLuo} only considered the existence of minimizers to \eqref{gmin} for $\lambda=1$ by restricting in the real-valued and radially symmetric Sobolev space. As we have seen, in our paper, we systematically and deeply investigate solutions to \eqref{system1}-\eqref{mass} in the more general Sobolev space.

When $p=4$, $\lambda=1$, if $E(u)=0$, one shall find that $u$ enjoys the following first order equation called self-dual equation,
\begin{align} \label{self}
D_1 u+i\, D_2 u=0.
\end{align}
The equation \eqref{self} can be transformed into the Loiuville equation, an integrable equation whose solutions are explicitly given.

When $p=4$, $\lambda>1$, in order to solve the minimization problem \eqref{mins}, we shall use a slight variant of the classical Lions concentration compactness principle from \cite{Li1, Li2}.

\begin{rem}
Note that Theorem \ref{nonexistence}, when $p=4$, $\lambda>1$, the problem has not been completely handled yet. In this case, it turns out that there is a constant $c_1 \geq c^*$ such that $m(c)=-\infty$ for any $c>c_1$, see Remark \ref{massinfty}, however, it is unclear if $c_1=c^*$, we conjecture this is the case. In addition, it is an open question to prove that \eqref{system1}-\eqref{mass} admits no solution for any $c >c^*$.
\end{rem}

Thirdly, we shall study solutions to \eqref{system1}-\eqref{mass} in the mass supercritical case. In this case, the energy functional $E$ restricted on $S(c)$ becomes unbounded from below for any $c>0$. Indeed, since $p>4$, one then easily gets from \eqref{scaling} that $E(u_t) \to -\infty$ as $t \to \infty$, which infers that $E$ restricted on $S(c)$ is unbounded from below. For this reason, it is unlikely to catch a solution to \eqref{system1}-\eqref{mass} by developing a global minimization problem, which causes our research more complex. In this situation, in order to seek for solutions to \eqref{system1}-\eqref{mass}, we shall introduce the following minimization problem,
\begin{align} \label{min}
\gamma(c):=\min_{u \in \mathcal{M}(c)} E(u),
\end{align}
where the constraint $\mathcal{M}(c)$ is defined by
\begin{align*}
\mathcal{M}(c):=\{u \in S(c): Q(u)=0\},
\end{align*}
and the functional $Q$ is defined by
\begin{align}  \label{defq}
Q(u):= \int_{\R^2} |D_1 u|^2 + |D_2 u|^2\, dx -\frac { \lambda(p-2)}{p} \int_{\R^2}|u|^p \, dx.
\end{align}
Indeed, the identity $Q(u)=0$ is the Pohozaev identity to the system \eqref{system1}, see Lemma \ref{Ph}. The constraint $\mathcal{M}(c)$ is the so-called Pohozaev manifold related to \eqref{system1}-\eqref{mass}. From Lemma \ref{ps}, one can find that $\mathcal{M}(c)$ is a natural constraint, which suggests that any minimizer to \eqref{min} is a ground state to \eqref{system1}-\eqref{mass} in the sense that it is a solution to \eqref{system1}-\eqref{mass} and minimizes the energy functional $E$ among all solutions with the same $L^2$-norm. \medskip

In this case, we first establish the existence of ground state to \eqref{system1}-\eqref{mass}.

\begin{thm} \label{existence}
Assume $p>4$, then there exists a constant $\hat{c}>0$ such that, for any $0<c< \hat{c}$, \eqref{system1}-\eqref{mass} possesses a ground state $u_c \in S(c)$ with $E(u_c)=\gamma(c)$.
\end{thm}

Let us briefly sketch the proof of Theorem \ref{existence}. To begin with, according to Lemma \ref{ps}, we can obtain a Palais-Smale sequence $\{u_n\} \subset \mathcal{M}(c)$ for the energy functional $E$ restricted on $S(c)$ at the level $\gamma(c)$. Later, by Lemma \ref{pre}, there exist a constant $\alpha_c \in \R$ and a nontrivial $u_c \in H^1(\R^2)$ as the weak limit of $\{u_n\}$ in $H^1(\R^2)$ satisfying the equation
$$
-\left(D_1D_1 + D_2D_2\right) u_c +  A_0 u_c +\alpha_c u_c = \lambda |u_c|^{p-2} u_c.
$$
Finally, since the function $c \mapsto \gamma(c)$ is nonincreasing on $(0, \infty)$, see Lemma \ref{nonincreasing}, we then conclude from Lemma \ref{lpconv} that $\gamma(c)=E(u_c)$ and $u_n \to u_c$ in $L^p(\R^2)$ as $n \to \infty$. This together with the fact that $\alpha_c >0$ for any $0<c< \hat{c}$, see Lemma \ref{alpha}, shows that $\|u_n-u_c\|=o_n(1)$ and $u_c \in S(c)$ is a ground state to \eqref{system1}-\eqref{mass}. Thus the proof is completed.

\begin{rem}
As we have already noticed, in Theorem \ref{existence}, it is required that the mass $c>0$ is small, which is only to ensure that the Lagrange multiplier $\alpha_c>0$, see Lemma \ref{c}. This is sufficient to achieve the compactness of the associated Palais-Smale sequence. As matter of fact, $\alpha_c \geq 0$ is always the case for any $c >0$, see Proposition \ref{nonnegative}. Thereby an interesting but challengeable issue is to prove the existence of solution to \eqref{system1}-\eqref{mass} for any $c>0$ large. To do this, one needs to study the limit equation
\begin{align} \label{limitequ}
-\left(D_1D_1 + D_2D_2\right) u +  A_0 u= \lambda |u|^{p-2} u.
\end{align}
If one can show that the equation \eqref{limitequ} has no solution in $H^1(\R^2)$, which in turn results that $\alpha_c > 0$ for any $c>0$ large, then the issue is successfully solved. Indeed, we expect that the equation \eqref{limitequ} has no solutions in $H^1(\R^2)$, however, the verification of this assertion is far from easy.
\end{rem}

By means of genus theory due to M.A. Krasnosel'skii, see \cite{AmMa}, we obtain the existence of infinitely many radially symmetric solutions to \eqref{system1}-\eqref{mass}. Here we define $\mathcal{M}_{rad}(c):=\mathcal{M}(c) \cap H^1_{rad}(\R^2)$, and $S_{rad}(c):=S(c) \cap H^1_{rad}(\R^2)$.

\begin{thm} \label{infsolutions}
Assume $p>4$, then, for any $0<c< \hat{c}$, \eqref{system1}-\eqref{mass} admits infinitely many radially symmetric solutions, where the constant $\hat{c}$ is given in Theorem \ref{existence}.
\end{thm}

When $p>4$, a similar result has been established  in \cite{LiLuo, Yuan}, where the authors discussed the existence of normalized solutions in the real-valued Sobolev space. Our Theorem \ref{infsolutions} is established in the complex-valued Sobolev space. As we shall see, in our situation, the proof of Theorem \ref{infsolutions} becomes complex, due to the presence of an additional term $\mbox{Im} \int_{\R^2}\left( A_1 \partial_1 u + A_2 \partial_2 u\right)\overline{u}\,dx$ in the energy functional.

The key argument to establish this theorem lies in verifying the compactness of the Palais-Smale sequence $\{u_n\} \subset \mathcal{M}_{rad}(c)$ for the energy functional $E$ restricted on $S_{rad}(c)$. The existence of such a Palais-Smale sequence $\{u_n\} \subset \mathcal{M}_{rad}(c)$ is insured by Lemma \ref{ps1}. At this point, one can follow the outline to prove Theorem \ref{existence} to complete the proof. We only need to point out that, in the present setting, the convergence of the sequence $\{u_n\}$ in $L^p(\R^2)$ is given for free, because the embedding $H_{rad}(\R^2) \hookrightarrow L^p(\R^2)$ is compact.

\begin{rem}
When $p>4$, it is also interesting to analyze the profile of the function $c \mapsto \gamma(c)$ for any $c>0$. The monotonicity of the function has already been revealed by Lemmas \ref{nonincreasing}-\ref{c}, and Proposition \ref{nonnegative}. One can check that the function $c \mapsto \gamma(c)$ is continuous for any $c>0$, and $\lim_{c \to 0^+} \gamma(c)=\infty$. In order to consider the properties of the function $c \mapsto \gamma(c)$ as $c$ goes to infinity, one needs to explore the existence of solution to the equation \eqref{limitequ} in some suitable Sobolev space. We would like to leave the research of this theme to the interested readers.
\end{rem}

Lastly, we shall study dynamical behaviors of ground states to \eqref{system1}-\eqref{mass}. To this end, we need to investigate the Cauchy problem of the system \eqref{sys1}.  For the case $p=4$, it seems first proved in \cite{BBS} that the Cauchy problem of the system \eqref{sys1} is locally well-posed in $H^2(\R^2)$. In the same paper, they also established global existence of solutions in $H^1(\R^2)$ for initial data having sufficiently small $L^2$ norm. However, the uniqueness of solutions in $H^1(\R^2)$ was unknown. Subsequently, unconditional uniqueness in $L^{\infty}_t H^1$ of solutions was obtained in \cite{huh2}. Recently, in \cite{Lim}, the author further demonstrated the local well-posedness of solutions with large initial data in $H^s(\R^2)$ for any $s \geq 1$. In this case, we also refer the readers to \cite{LiSm, LiSmTa} for the global well-posedness of solutions and scatter of solutions in $L^2(\R^2)$. For the case $p>2$, the local well-posedness of solutions in $H^1(\R^2)$ is given in forthcoming Lemma \ref{wellposed}, by which we are able to deduce the global well-posedness of solutions in $H^1(\R^2)$. Namely,

\begin{thm} \label{globalexis}
Let $\varphi_0 \in H^1(\R^2)$, either $2<p<4$, or $p=4$ and $\|\varphi_0\|_2$ is small enough, or else $p>4$ and $ \varphi_0 \in \mathcal{O}_c:=\{u \in S(c): E(u)<\gamma(c), Q(u)>0\}$, then the solution $\varphi \in C([0, T); H^1(\R^2))$ to the Cauchy problem of the system \eqref{sys1} with $\varphi(0)=\varphi_0$ exists globally in time.
\end{thm}

\begin{rem}
For the sake of simplicity, in Theorem \ref{globalexis}, when $p=4$, we only show that if the $L^2$-norm of initial datum $\varphi_0$ is small enough, then the solution $\varphi$ exists globally in time. In fact, when $p=4$, $\lambda \geq 1$, we can find a sharp threshold for the global existence of solutions to the Cauchy problem of the system \eqref{sys1}. In other words, we have that if initial datum $\varphi_0\in H^1(\R^2)$ with $\|\varphi_0\|_2 <c^*$, then the solution $\varphi$ to the problem exists globally in time, while if initial datum $\varphi_0\in H^1(\R^2)$ with $\|\varphi_0\|_2 \geq c^*$, then the finite time blowup of the solution $\varphi$ to the problem may occur, where the constant $c^*$ is defined by \eqref{mins}. It is noticed that $c^*=8 \pi$ for $\lambda=1$. The proof of this argument shall be presented in future publication.
\end{rem}

When $2<p<4$, Theorem \ref{stable} indicates that the set of minimizers to \eqref{gmin} is orbitally stable. While $p > 4$, we shall prove that ground states to \eqref{system1}-\eqref{mass} are strongly unstable in the following sense.

\begin{defi}
We say that $u \in H^1(\R^2)$ is strongly unstable, if, for any $\eps >0$, there is $\varphi_0 \in H^1(\R^2)$ with $\|\varphi_0-u\| \leq \eps$ such that the solution $\varphi \in C([0, T); H^1(\R^2))$ to the Cauchy problem of the system \eqref{sys1} with $\varphi(0)=\varphi_0$ blows up in finite time.
\end{defi}

When $p >4$, we have the following result.

\begin{thm} \label{unstable}
Assume $p>4$. Let $u_c \in S(c)$ be a ground state to \eqref{system1}-\eqref{mass} obtained in Theorem \ref{existence}. If either $u_c$ belongs to $\Sigma:=\left\{ u \in H^1(\R^2): |xu| \in L^2(\R^2)\right\}$ or $u_c$ is radially symmetric and $p \leq 6$, then it is strongly unstable.
\end{thm}

If $u_c \in \Sigma$, the instability can be established by analyzing the evolution of the virial quantity $I(t)$ given by
$$
I(t):=\int_{\R^2}|x|^2|\varphi|^2 \, dx,
$$
where $\varphi \in C([0, T); H^1(\R^2)$ is a solution to the Cauchy problem of the system \eqref{sys1}. For the evolution of $I(t)$, see Lemma \ref{virial1}. We would like to mention that the quantity $I(t)$ was introduced in \cite{BBS} to consider the finite time blowup of solutions to the problem in the mass critical case.

To make matter worse, it is severe difficult to guarantee that any ground state to \eqref{system1}-\eqref{mass} belongs to $\Sigma$, because the electromagnetic interaction is long-range in the sense that the gauge fields $A_1, A_2$ decay less quickly than $|x|^{-1}$ as $|x|$ goes to infinity. In such a situation, if $u_c$ is radially symmetric, the instability can be proved by analyzing the evolution of a localized virial quantity $V_{\chi_R}[\varphi(t)]$ given by
$$
V_{\chi_R}[\varphi(t)]:=\mbox{Im} \int_{\R^2} \overline{\varphi} \left(D_1 \varphi \, \partial_1 \chi_R + D_2 \varphi \, \partial_2 \chi_R \right) \, dx,
$$
where $\chi_R$ defined by \eqref{defchi} is a radially symmetric cut-off function. For the evolution of $V_{\chi_R}[\varphi(t)]$, see Lemma \ref{vlem}.

It is worth mentioning that our paper seems the first occasion to employ the localized quantity $V_{\chi_R}[\varphi(t)]$ to investigate the finite time blow of solutions to the Cauchy problem of the system \eqref{sys1}.

\begin{rem}
Let us point out that Lemma \ref{vlem} is flexible. It is not only useful to discuss the instability of ground state to \eqref{system1}-\eqref{mass}, it is also applicable to study the finite time blowup of radially symmetric solutions to the Cauchy problem of the system \eqref{sys1} for any $p \geq 4$. To our knowledge, the finite time blowup of solutions to the problem was considered only for $p=4$.
\end{rem}

\begin{rem}
In Theorem \ref{unstable}, when $u_c$ is radially symmetric, the assumption that $p \leq 6$ is technical. If $u_c$ is radially symmetric, we conjecture that $u_c$ is strongly unstable for any $p >4$. In addition, it is an important open question to show that any ground state (without radial symmetry restriction) to \eqref{system1}-\eqref{mass} is strongly unstable.
\end{rem}

When $p=4$, we address the following remark with respect to the instability of ground states to \eqref{system1}-\eqref{mass}.

\begin{rem} \label{blowupcritical}
When $p=4$, $\lambda=1$, then any solution to \eqref{system1}-\eqref{mass} is strongly unstable. Indeed, in this case, one can closely follow the approach in \cite{huh1} to achieve the instability by constructing a finite time blowup solution to the Cauchy problem of the system \eqref{sys1}. We shall leave the proof to the interested readers. While $p=4$, $\lambda>1$, in order to demonstrate the instability of ground states to \eqref{system1}-\eqref{mass}, we need a refinement of Lemma \ref{vlem}. This shall be also treated in future publication.
\end{rem}

\end{result}

\begin{structure}
The rest of this paper is laid out as follows. In Section \ref{preliminaries}, we present some preliminary results. In Section \ref{masssub}, we consider solutions to \eqref{system1}-\eqref{mass} in the mass subcritical case, and we establish Theorem \ref{compactness} and Theorem \ref{symmetry}. In Section \ref{masscri}, we study solutions to \eqref{system1}-\eqref{mass} in the mass critical case, and we prove Theorem \ref{nonexistence}. In Section \ref{masssup}, we investigate solutions to \eqref{system1}-\eqref{mass} in the mass supercritical case, and we demonstrate Theorems \ref{existence}-\ref{infsolutions}. Section \ref{dynamical} is devoted to the study of the dynamical behaviors of ground states to \eqref{system1}-\eqref{mass}, which contains the proofs of Theorem \ref{stable} and Theorems \ref{globalexis}-\ref{unstable}. Finally, in Appendix, we show the proof of Lemma \ref{virial} which serves to discuss the evolution of the virial type quantities.
\end{structure}

\begin{notation}
Throughout the paper, for any $1 < t \leq \infty$, we denote by $L^t(\R^2)$ the Lebesgue space consisting of complex-valued functions, equipped with the norm  $\|\cdot\|_t$. We denote by $H^1(\R^2)$ the usual Sobolev space consisting of complex-valued functions, equipped with the standard norm $\|\cdot\|$. In addition, $H^1_{rad}(\R^2)$ stands for the subspace of $H^1(\R^2)$, which consists of the radially symmetric functions in $H^1(\R^2)$. The symbol $\overline{u}$ represents the conjugate function of $u$. We use the letter $C$ for a generic positive constant, whose value may change from line to line. The convergence of sequence in associated spaces is understood in the sense of subsequence. We use the notation $o_n(1)$ for any quantity which tends to zero as $n \to \infty$. For convenience of notations, we shall omit the dependence of some quantities on the constant $\lambda>0$.
\end{notation}

\section{Preliminary Results} \label{preliminaries}

In this section, we shall present some preliminary results used to establish our main results. To begin with, we give few observations. Note that the Coulomb gauge condition \eqref{gauge}, by straightforward manipulations, then
$$
\left(D_1D_1+D_2D_2\right) u = \Delta u -\left(A_1^2 + A_2^2 \right) u+2\,i \left(A_1 \partial_1 u + A_1 \partial_2 u\right).
$$
This suggests that the first equation in the system \eqref{system1} is equivalent to
$$
-\Delta u +\left(A_1^2 + A_2^2 \right) u-2\,i \left(A_1 \partial_1 u + A_1 \partial_2 u\right) + A_0 u + \alpha u =\lambda |u|^{p-2}u.
$$
Besides, there holds that
\begin{align} \label{idd}
|D_1 u|^2 + |D_2 u|^2=|\nabla u|^2+\left(A_1^2 + A_2^2 \right) |u|^2 + 2\,\mbox{Im} \left(A_1 \partial_1 u + A_2\partial_2 u\right)\overline{u}.
\end{align}
Thus we may rewrite the energy functional $E(u)$ as
\begin{align*}
E(u)=\frac 12 \int_{\R^2} |\nabla u|^2 + \left(A_1^2+A_2^2\right)|u|^2 + 2\, \mbox{Im} \left(A_1 \partial_{1} u + A_2 \partial_{2} u \right)\,\overline{u} \, dx-\frac {\lambda}{p} \int_{\R^2}|u|^p \,dx.
\end{align*}
By using the equations satisfied by $A_{\mu}$ in the system \eqref{system1} for $\mu=0, 1, 2$, then
\begin{align}\label{Aid}
\begin{split}
\int_{\R^2}A_0|u|^2&=-2\int_{\R^2}A_0\left(\partial_1 A_2-\partial_2 A_1\right) \,dx =2\int_{\R^2}A_2 \partial_1 A_0-A_1\partial_2 A_0\, dx \\
&=2\int_{\R^2} \left(A_1^2 + A_2^2\right) |u|^2 \, dx + 2\,\mbox{Im} \int_{\R^2}\left(A_1 \partial_1 u + A_2 \partial_2 u\right) \overline{u} \, dx.
\end{split}
\end{align}
 
\begin{lem} \label{escaling}
For any $t >0$ and $u \in S(c)$, there holds that $u_t \in S(c)$ and
\begin{align} \label{eut}
E(u_t)= \frac {t^2}{2} \int_{\R^2}|D_1 u|^2 + |D_2 u|^2 \, dx - \frac{\lambda t^{p-2}}{p} \int_{\R^2}|u|^p \, dx.
\end{align}
\begin{proof}
Since $u_t(x):=tu(tx)$ for $x \in \R^2$, then 
$$
\int_{\R^2} |u_t|^2 \, dx =\int_{\R^2} t^2 |u(tx)|^2 \, dx=\int_{\R^2} |u|^2 \, dx.
$$
It follows that $u_t \in S(c)$ if $u \in S(c)$. Note that
$$
\int_{\R^2} |\nabla u_t|^2 \, dx = \int_{\R^2} t^4 |\nabla u(tx)|^2 \, dx =t^2 \int_{\R^2} |\nabla u|^2 \, dx.
$$
In view of the definition of $A_1$, then
\begin{align*}
\int_{\R^2} A_1^2(u_t) |u_t|^2 \, dx&=\frac{1}{16 \pi^2}\int_{\R^2} \left(\frac{x_2-y_2}{|x-y|^2} |u_t(y)|^2\, dy\right)^2 |u_t(x)|^2 \, dx \\
&=\frac{t^2}{16 \pi^2}\int_{\R^2} \left(\frac{x_2-y_2}{|x-y|^2} |u(y)|^2\, dy\right)^2 |u(x)|^2 \, dx\\
&=t^2\int_{\R^2} A_1^2 |u|^2 \, dx
\end{align*}
and
\begin{align*}
\int_{\R^2} A_1(u_t) \partial_1 u_t \,\overline{u_t} \, dx&=\frac{1}{4\pi}\int_{\R^2} \frac{x_2-y_2}{|x-y|^2} |u_t(y)|^2\, dy \, \partial_1 u_t(x)  \overline{u_t}(x)  \, dx \\
&=\frac{t^2}{4\pi}\int_{\R^2} \frac{x_2-y_2}{|x-y|^2} |u(y)|^2\, dy \, \partial_1 u(x)  \overline{u}(x)  \, dx \\
&=t^2\int_{\R^2} A_1 \partial_1 u \, \overline{u} \, dx.
\end{align*}
By a similar way, we can obtain that
$$
\int_{\R^2} A_2^2(u_t) |u_t|^2 \, dx = t^2\int_{\R^2} A_2^2 |u|^2 \, dx, \quad \int_{\R^2} A_2(u_t) \partial_2 u_t \,\overline{u_t} \, dx=t^2\int_{\R^2} A_2 \partial_2 u \, \overline{u} \, dx.
$$
Furthermore, we find that
$$
\int_{\R^2} |u_t|^p \,dx =\int_{\R^2} t^p |u(tx)|^p \, dx =t^{p-2} \int_{\R^2} |u|^p \, dx.
$$
Taking into account \eqref{idd}, we then get \eqref{eut}. Thus the proof is completed.
\end{proof}
\end{lem}

As a consequence of the Hardy-Littlewood-Sobolev inequality, see \cite[Chapter 4]{Li}, we know the following estimates to the gauge fields $A_{\mu}$ for $\mu=0, 1, 2$.

\begin{lem} \cite[Propositions 4.2-4.3]{huh} \label{Aineq}
Assume $1 <s <2$, $\frac 1s - \frac 1q =\frac 12$, then
\begin{align*}
\|A_j\|_q \leq C \|u\|^2_{2s} \, \, \, \text{for} \,\, j=1,2, \quad \|A_0\|_q \leq C \|u\|^2_{2s} \|u\|_4^2.
\end{align*}
\end{lem}

With Lemma \ref{Aineq} in hand, we are able to immediately deduce that
\begin{align} \label{bdd}
\|A_j u\|_2 \leq \|A_j\|_q \|u\|_{s'} \leq C\|u\|_{2s}^2\|u\|_{s'},
\end{align}
where $q=\frac{2s}{2-s}$, and $s'=\frac{s}{s-1}$. In light of the Young inequality with $\eps>0$, we have that
\begin{align} \label{bdd1}
\left|\mbox{Im} \int_{\R^2}\left(A_1 \partial_1 u + A_2 \partial_2 u\right) \overline{u} \,dx\right|
\leq \eps \int_{\R^2}|\nabla u|^2 \, dx + C_{\eps} \int_{\R^2}\left(A_1^2 + A_2^2\right) |u|^2 \, dx.
\end{align}
Since $H^1(\R^2)$ is continuously embedded into $L^t(\R^2)$ for any $t \geq 2$, then \eqref{bdd}-\eqref{bdd1} readily indicates that the energy functional $E$ is well-defined in $H^1(\R^2)$.  From \cite[Lemma 3.2]{huh}, we can deduce that the energy functional $E$ is of class $C^1$, and for any $v \in H^1(\R^2)$,
\begin{align*}
E'(u)v&= \int_{\R^2}\nabla u \cdot\nabla \overline{v} + \left(A_1^2 + A_2^2 \right) u \, \overline{v}+ A_0 u \,\overline{v} \, dx  \\
&\quad -2i\int_{\R^2}\left(A_1 \partial_1 u + A_2\partial_2 u\right) \overline{v}\, dx
-\lambda  \int_{\R^2} |u|^{p-2}u \, \overline{v} \, dx.
\end{align*}
This shows that any critical point of the energy functional $E$ restricted on $S(c)$ corresponds to a solution to \eqref{system1}-\eqref{mass}.

For any $2 \leq t < \infty$, let us recall the well-known Gagliardo-Nirenberg inequality in $H^1(\R^2)$,
\begin{align} \label{GN}
\|u\|_t\leq C \|\nabla u\|_2^{\alpha} \|u\|_2^{1- \alpha} \quad \mbox{for} \, \, \alpha=1-\frac 2t.
\end{align}
Let us also recall the following diamagnetic inequality in $\R^n$.

\begin{lem} \cite[Chapter 7]{Li}\label{diaineq}
Let $A: \R^n \to \R^n$ be in $L_{loc}^2(\R^n)$ and let $f$ be in $H^1_A:=\left\{f \in L^2(\R^n): \|\left(\nabla +i A\right) f \|_2 < \infty \right\}$, then $|f| \in H^1(\R^n)$, and the diamagnetic inequality
$$
|\nabla |f|(x)| \leq |(\nabla + i A) f(x)|
$$
holds pointwise for almost every $x \in \R^n$.
\end{lem}

As an easy result of the Gagliardo-Nirenberg inequality \eqref{GN} and Lemma \ref{diaineq}, then
\begin{align} \label{MGN}
\|u\|_t \leq C \|\nabla |u|\|_2^{\frac{t-2}{t}} \|u\|_2^{\frac 2 t} \leq C \left(\|D_1 u\|_2^{\frac{t-2}{t}} + \|D_2 u\|_2^{\frac{t-2}{t}}\right)\|u\|_2^{\frac 2 t}.
\end{align}
The inequality \eqref{MGN} can be regarded as a Gagliardo-Nirenberg type inequality with respect to the covariant derivative operators $D_1, D_2$.

\begin{lem} \label{BL}
Assume $\{u_n\} \subset H^1(\R^2)$, and $u_n \rightharpoonup u$ in $H^1(\R^2)$ as $n \to \infty$, then
\begin{enumerate}
  \item [(i)] $\int_{\R^2}{A_{j}^2(u_n)} u_n \overline{v}  \,dx =\int_{\R^2}A_{j}^2(u) u \,\overline{v} \, dx +o_n(1)$ for any $v \in H^1(\R^2),$ and $j=1,2$;
  \item [(ii)] $\int_{\R^2}{A_{0}(u_n)} u_n \overline{v}  \,dx =\int_{\R^2}A_{0}(u) u \,\overline{v} \, dx +o_n(1)$ for any $v \in H^1(\R^2)$;
  \item [(iii)] $\int_{\R^2}A_j^2(u_n-u)|u_n-u|^2 \, dx + \int_{\R^2}A_j^2(u)|u|^2 \, dx =\int_{\R^2}A_j^2(u_n)|u_n|^2 \, dx + o_n(1)$ for $j=1,2$.
\end{enumerate}
\end{lem}
\begin{proof}
The assertion $(iii)$ is from \cite[Proposition 2.2]{WaTa}, it remains to prove that the assertions $(i)$-$(ii)$ hold true. In what follows, we only show the proof of the assertion $(i)$, and the assertion $(ii)$ can be proved by an analogous way. Since $u_n \rightharpoonup u$ in $H^1(\R^2)$ as $n \to \infty$, we then know that $u_n(x) \to u(x)$ a.e. $x \in \R^2$ as $n \to \infty$. It follows from \cite[Proposition 2.2]{WaTa} that $A_j(u_n)(x) \to A_j(u)(x)$ a.e. $x \in \R^2$ as $n \to \infty$. Therefore, $A^2_j(u_n)u_n(x) \to A_j^2(u)u(x)$ a.e. $x \in \R^2$ as $n \to \infty$. In addition, by the H\"older inequality and Lemma \ref{Aineq},
\begin{align} \label{Al2}
\begin{split}
\int_{\R^2}|A_j^2(u_n) u_n|^2 \, dx &\leq \left(\int_{\R^2}|A_j^2(u_n)|^q\right)^{\frac 2q} \left(\int_{\R^2}|u_n|^{\frac{2q}{q-2}}\right)^{\frac{q-2}{q}} \\
&\leq C \left(\int_{\R^2}|u_n|^{2s}\right)^{\frac 4s} \left(\int_{\R^2}|u_n|^{\frac{2q}{q-2}}\right)^{\frac{q-2}{q}},
\end{split}
\end{align}
where $s=\frac{2q}{q+1}$ and $2<q<\infty$. Thus \eqref{Al2} indicates that $\{A_j^2(u_n) u_n\}$ is bounded in $L^2(\R^2)$, because $\{u_n\}$ is bounded in $H^1(\R^2)$ and  $H^1(\R^2)$ is continuously embedded into $L^t(\R^2)$ for any $t \geq 2$. From \cite[Corollary 3.8]{Brezis}, then $A_j^2(u_n)u_n \rightharpoonup A_j^2(u)u$ in $L^2(\R^2)$ as $n \to \infty$, hence the assertion $(i)$ follows, and the proof is completed.
\end{proof}

Correspondingly, we have the following statement.

\begin{lem} \label{BL1}
Assume $\{u_n\} \subset H^1(\R^2)$, and $u_n \rightharpoonup u$ in $H^1(\R^2)$ as $n \to \infty$, then
\begin{enumerate}
  \item [(i)] $ \mbox{Im} \int_{\R^2}{A_{j}(u_n)} \partial_j{u_n}  \overline{v}  \,dx =\mbox{Im}\int_{\R^2}A_{j}(u) \partial_j u \, \overline{v} \, dx +o_n(1)$ for any $v \in H^1(\R^2)$, $j=1,2$;
  \item [(ii)] $\mbox{Im} \int_{\R^2}A_j(u_n-u) \partial_j(u_n-u)\left(\overline{u}_n-\overline{u}\right)\, dx+ \mbox{Im} \int_{\R^2}A_j(u) \partial_j u_j  \overline{u} \, dx \\ =\mbox{Im} \int_{\R^2}{A_j(u_n)}\partial_j u_n \overline{u}_n \,dx+o_n(1)$ for $j=1,2$.
\end{enumerate}
\end{lem}
\begin{proof}
Here we only give the proof of the assertion $(i)$, and the assertion $(ii)$ can be deduced by adapting the similar arguments. In view of Lemma \ref{Aineq}, we get that $A_{j}(u) v \in L^2(\R^2)$. Since $u_n \rightharpoonup u$ in $H^1(\R^2)$ as $n \to \infty$, then
\begin{align}\label{convergence}
\mbox{Im} \int_{\R^2}{A_j(u)} \partial_j u_n \,\overline{v}  \,dx =\mbox{Im} \int_{\R^2}{A_j(u)} \partial_j u \,\overline{v}  \,dx+o_n(1).
\end{align}
It follows from Lemma \ref{Aineq} that $\{A_j(u_n)\}$ is bounded in $L^q(\R^2)$ for any $2<q<\infty$. On the other hand, we know that $A_j(u_n)(x) \to A_j(u)(x)$ a.e. $x \in \R^2$ as $n \to \infty$. Therefore, it holds that $A_j(u_n)\rightharpoonup A_j(u)$ in $L^q(\R^2)$ as $n \to \infty$. This together with the fact that $A_{j}(u)\, |v|^2 \in L^{q'}(\R^2)$ leads to
\begin{align} \label{AA1}
\int_{\R^2}{A_j(u_n)}{A_j(u)}|{v}|^2 \,dx =\int_{\R^2} A_j^2(u) \,|{v}|^2 \, dx +o_n(1).
\end{align}
Moreover, by using the weak convergence of $\{u_n\}$ in $H^1(\R^2)$, we can infer that
\begin{align} \label{AA2}
\int_{\R^2}{A_j^2(u_n)}|{v}|^2 \,dx =\int_{\R^2}{A_j^2(u)}|{v}|^2 \,dx +o_n(1).
\end{align}
We now apply \eqref{AA1}-\eqref{AA2} to conclude that
\begin{align} \label{Aj}
\int_{\R^2}\left|A_j(u_n)-A_j(u)\right|^2|v|^2\, dx =o_n(1).
\end{align}
As a consequence of the H\"older inequality and \eqref{Aj}, we then get that
\begin{align} \label{aim}
\mbox{Im}\int_{\R^2}{A_j(u_n)} \partial_j u_n \overline{v}  \,dx =\mbox{Im} \int_{\R^2}{A_j(u)} \partial_j u_n \overline{v}  \,dx+o_n(1).
\end{align}
Hence \eqref{convergence} and \eqref{aim} readily indicate that the assertion $(i)$ holds, and we then finish the proof.
\end{proof}

As a direct application of \cite[Proposition 2.1]{BHS1}(see also \cite{HuSe}), we have the following.

\begin{lem} \label{Ph}
Let $u \in H^1(\R^2)$ be a solution to the system \eqref{system1}, then $Q(u)=0$, where the functional $Q$ is defined by \eqref{defq}.
\end{lem}

\begin{lem} \label{invariant}
Assume $2<p <\infty$, then the energy functional $E$ is invariant under any orthogonal transformation in $\R^2$.
\end{lem}
\begin{proof}
Recall first that any orthogonal transformation in $\R^2$ is either a rotation or a reflection. Let $T$ be an orthogonal transformation in $\R^2$, and define the action of $T$ on $u$ by $(T\circ u)(x):=u(T^{-1}x),$ where $T^{-1}$ represents the inverse transformation of $T$. For simplicity, we shall only deduce that the energy functional $E$ is invariant under any rotation. Suppose now that $T$ is a rotation in $\R^2$ through an angle $\theta$ with fixing the origin, then it can be denoted by
$$
T:= \left(\begin{array}{cc}
    cos \,\theta & -sin \,\theta \\
    sin \,\theta & cos \,\theta
  \end{array} \right),
\quad
T^{-1}:= \left(\begin{array}{cc}
    cos \,\theta & sin \,\theta \\
    -sin \,\theta & cos \,\theta
  \end{array} \right).
$$
Our goal is to prove that $E(T\circ u)=E(u)$. To do this, it is sufficient to assert that
$$
\int_{\R^2}\left(A_1^2(T\circ u) + A_2^2(T\circ u)\right)|T\circ u|^2 \, dx = \int_{\R^2} \left(A_1^2+A_2^2\right)|u|^2 \,dx,
$$
and
\begin{align*}
& \mbox{Im}\int_{\R^2}\left(A_1(T\circ u) \,\partial_{1}\left(T\circ u)\right)+ A_2(T\circ u) \, \partial_{2}\left(T\circ u)\right) \right) \overline{T\circ u} \,dx \\
&= \mbox{Im}\int_{\R^2}\left(A_1 \, \partial_{1} u + A_2 \, \partial_{2} u \right)\,\overline{u} \, dx.
\end{align*}
Indeed, by simple calculations, then
\begin{align*}
&\int_{\R^2}\left(A_1^2(T\circ u) + A_2^2(T\circ u)\right)|T\circ u|^2 \, dx \\
&=\int_{\R^2} \left(\int_{\R^2}\frac{x_2-y_2}{|x-y|^2} \frac{|u(T^{-1}y)|^2}{2} \, dy \right)^2 |u(T^{-1}x)|^2 \, dx \\
&\quad +\int_{\R^2} \left(\int_{\R^2}\frac{x_1-y_1}{|x-y|^2} \frac{|u(T^{-1}y)|^2}{2} \, dy \right)^2 |u(T^{-1}x)|^2 \, dx \\
&=\int_{\R^2} \left(\int_{\R^2}\frac{sin \,\theta \,(x_1-y_1) + cos \,\theta \,(x_2-y_2)}{|x-y|^2} \frac{|u(y)|^2}{2} \, dy \right)^2 |u(x)|^2 \, dx \\
&\quad +\int_{\R^2} \left(\int_{\R^2}\frac{cos \,\theta \,(x_1-y_1) - sin \,\theta \,(x_2-y_2)}{|x-y|^2} \frac{|u(y)|^2}{2} \, dy \right)^2 |u(x)|^2 \, dx \\
&=\int_{\R^2} \left(A_1^2+A_2^2\right)|u|^2 \,dx,
\end{align*}
and
\begin{align*}
&\mbox{Im}\int_{\R^2}\left(A_1(T\circ u) \,\partial_{1}\left(T\circ u)\right)+ A_2(T\circ u) \, \partial_{2}\left(T\circ u)\right) \right) \overline{T\circ u} \,dx \\
&=\mbox{Im} \int_{\R^2} \int_{\R^2}\frac{x_2-y_2}{|x-y|^2} \frac{|u(T^{-1}y)|^2}{2} \, dy \, \partial_{1}\left(u(T^{-1}x)\right) \overline{u(T^{-1}x)}\, dx \\
&\quad -\mbox{Im} \int_{\R^2} \int_{\R^2}\frac{x_1-y_1}{|x-y|^2} \frac{|u(T^{-1}y)|^2}{2} \, dy \, \partial_{2}\left(u(T^{-1}x)\right) \overline{u(T^{-1}x)} \, dx \\
&=\mbox{Im} \int_{\R^2} \int_{\R^2} \frac{sin \,\theta \,(x_1-y_1) + cos \theta \,(x_2-y_2)}{|x-y|^2} \frac{|u(y)|^2}{2} \, dy \left(cos \,\theta \, \partial_{1} u -sin \,\theta \, \partial_{2} u\right) \overline{u} \,dx\\
&\quad -\mbox{Im} \int_{\R^2} \int_{\R^2} \frac{cos \,\theta \,(x_1-y_1) - sin \,\theta \,(x_2-y_2)}{|x-y|^2} \frac{|u(y)|^2}{2} \, dy \left(sin \,\theta \, \partial_{1} u + cos \,\theta  \,\partial_{2}  u\right) \overline{u} \,dx\\
&= \mbox{Im}\int_{\R^2}\left(A_1 \, \partial_{1} u + A_2 \, \partial_{2} u \right)\,\overline{u} \, dx.
\end{align*}
Thus the proof is completed.
\end{proof}

\begin{rem} \label{invar}
By Lemma \ref{invariant} and the principle of symmetric criticality, see \cite[Theorem 1.28]{Wi}, we know that a critical point of the energy functional $E$ in the radially symmetric functions subspace is one of that in the whole Sobolev space.
\end{rem}

\section{The mass subcritical case} \label{masssub}

In this section, we study solutions to \eqref{system1}-\eqref{mass} in the mass subcritical case. The main aim of this section is to consider the minimization problem \eqref{gmin}.

\begin{lem} \label{subadd}
Assume $2<p<4$, then $m(c) \leq m(c_1)+m(c_2)$ for $c=c_1 +c_2$, and $c_1, c_2 \geq 0$.
\end{lem}
\begin{proof}
Note that the energy functional $E$ is invariant under any translation in $\R^2$, then, by the definition of $m(c)$ and the density of $C^{\infty}_0(\R^2)$ in $H^1(\R^2)$, we know that, for any $\eps >0$, there exist two functions $\psi_1, \psi_2 \in C^{\infty}_0(\R^2)$ with $\text{supp} \, \psi_1 \cap \text{supp} \, \psi_2 =\emptyset$ and $\psi_1 \in S(c_1), \psi_2 \in S(c_2)$ such that
\begin{align} \label{sub0}
E(\psi_1) \leq m(c_1) + \frac{\eps}{4}, \quad E(\psi_2) \leq m(c_2) + \frac{\eps}{4}.
\end{align}
Without loss of generality, we assume that
\begin{align} \label{dist}
\text{dist}(\text{supp} \, \psi_1, \text{supp} \, \psi_2) \geq n \,\,\, \mbox{for some} \,\,\, n \in \mathbb{N}^+.
\end{align}
Now define $\psi:=\psi_1 + \psi_2$, since $\psi_1$ and $\psi_2$ have disjoint supports, then $ \psi \in S(c)$ and
\begin{align} \label{sub}
\begin{split}
\int_{\R^2} |\nabla \psi|^2 \, dx &=\int_{\R^2} |\nabla \psi_1|^2 \,dx + \int_{\R^2} |\nabla \psi_2|^2 \, dx, \\
\int_{\R^2} |\psi|^p \, dx &=\int_{\R^2} |\psi_1|^p \,dx + \int_{\R^2} |\psi_2|^p \, dx.
\end{split}
\end{align}
In addition,
\begin{align}\label{subb1}
\begin{split}
\int_{\R^2}A_1^2(\psi)|\psi|^2 \, dx &= \int_{\R^2} \left(\int_{\R^2} \frac{x_2-y_2}{|x-y|^2} \frac{|\psi_1(y) + \psi_2(y)|^2}{2} \, dy\right)^2|\psi_1(x) + \psi_2(x)|^2\, dx \\
&=\int_{\R^2} \left(\int_{\R^2} \frac{x_2-y_2}{|x-y|^2} \frac{|\psi_1(y)|^2 + |\psi_2(y)|^2}{2} \, dy\right)^2\left(|\psi_1(x)|^2 + |\psi_2(x)|^2 \right)\, dx\\
&=\int_{\R^2} A_1^2(\psi_1)|\psi_1|^2 \, dx + \int_{\R^2}A_1^2(\psi_2)|\psi_2|^2 \, dx + \int_{\R^2}A_1^2(\psi_1)|\psi_2|^2 \, dx \\
&\quad + \int_{\R^2} A_1^2(\psi_2)|\psi_1|^2 \, dx + 2\int_{\R^2} A_1(\psi_1)A_1(\psi_2) \left(|\psi_1|^2 + |\psi_2|^2 \right)\, dx,
\end{split}
\end{align}
and
\begin{align} \label{subb2}
\begin{split}
\mbox{Im} \int_{\R^2}A_1(\psi) \,\partial_1 \psi \, \overline{\psi} \,dx
&=\mbox{Im}\int_{\R^2}A_1(\psi_1) \,\partial_1 \psi_1\, \overline{\psi}_1 + A_1(\psi_2) \,\partial_1 \psi_2 \, \overline{\psi}_2 \, dx \\
& \quad + \mbox{Im}\int_{\R^2} A_1(\psi_1) \,\partial_1 \psi_2\, \overline{\psi}_2 + A_1(\psi_2) \,\partial_1 \psi_1 \, \overline{\psi}_1 \, dx.
\end{split}
\end{align}

Let us first treat the terms in the right hand side of \eqref{subb1}. By applying \eqref{dist}, we deduce that
\begin{align*}
\int_{\R^2}A_1^2(\psi_1)|\psi_2|^2 \, dx&=\int_{\R^2}\left(\int_{\R^2} \frac{x_2-y_2}{|x-y|^2} \frac{|\psi_1(y)|^2}{2} \, dy\right)^2|\psi_2(x)|^2\, dx\\
&=\int_{\text{supp} \, \psi_2}\left(\int_{\text{supp} \, \psi_1} \frac{x_2-y_2}{|x-y|^2} \frac{|\psi_1(y)|^2}{2} \, dy\right)^2|\psi_2(x)|^2\, dx  \\
& \leq \frac{1}{4n^2}\|\psi_1\|_2^2 \|\psi_2\|^2_2,
\end{align*}
and
\begin{align*}
\int_{\R^2} A_1(\psi_1)A_1(\psi_2) |\psi_1|^2 \, dx&=\int_{\R^2} \int_{\R^2} \frac{x_2-y_2}{|x-y|^2} \frac{|\psi_2(y)|^2}{2} \, dy \,A_1(\psi_1) |\psi_1|^2 \, dx \\
&=\int_{{\text{supp} \, \psi_1}} \int_{{\text{supp} \, \psi_2}} \frac{x_2-y_2}{|x-y|^2} \frac{|\psi_2(y)|^2}{2} \, dy \,A_1(\psi_1)|\psi_1|^2 \, dx \\
&\leq \frac{1}{2n}\|\psi_2\|_2^2 \int_{\R^2} |A_1(\psi_1)||\psi_1|^2 \, dx\\
& \leq \frac{1}{2n}\|\psi_2\|_2^2 \|A(\psi_1)\|_{q} \|\psi_1\|_{\frac{2q}{q-1}}^2 \\
& \leq  \frac{C}{2n}\|\psi_2\|_2^2 \|\psi_1\|_{2s}^2 \|\psi_1\|_{\frac{2q}{q-1}}^2,
\end{align*}
where we used the H\"older inequality and Lemma \ref{Aineq} with $s=\frac{2q}{q+2}$ and $2<q<\infty$. By an analogous manner, we can obtain that
\begin{align*}
\int_{\R^2}A_1^2(\psi_2)|\psi_1|^2 \, dx & \leq \frac{1}{4n^2}\|\psi_2\|_2^2 \|\psi_1\|^2_2, \\
\int_{\R^2} A_1(\psi_1)A_1(\psi_2) |\psi_2|^2 \, dx & \leq \frac{C}{2n}\|\psi_1\|_2^2 \|\psi_2\|_{2s}^2 \|\psi_2\|_{\frac{2q}{q-1}}^2.
\end{align*}
Consequently, from the estimates above,
\begin{align}\label{estimate1}
\int_{\R^2}A_1^2(\psi)|\psi|^2 \,dx = \int_{\R^2} A_1^2(\psi_1)|\psi_1|^2 \, dx + \int_{\R^2}A_1^2(\psi_2)|\psi_2|^2 \, dx + o_n(1).
\end{align}
We next deal with the terms in the right hand side of \eqref{subb2}. By using \eqref{dist} again, we derive that
\begin{align*}
\left|\mbox{Im}\int_{\R^2} A_1(\psi_1) \,\partial_1 \psi_2\, \overline{\psi}_2 \, dx \right|
&=\left|\mbox{Im} \int_{\R^2}\int_{\R^2} \frac{x_2-y_2}{|x-y|^2} \frac{|\psi_1(y)|^2}{2} \, dy \,  \partial_1 \psi_2 \, \overline{\psi}_2 \,dx \right| \\
&=\left|\mbox{Im} \int_{\text{supp} \, \psi_2}\int_{\text{supp} \, \psi_1} \frac{x_2-y_2}{|x-y|^2} \frac{|\psi_1(y)|^2}{2} \, dy \,  \partial_1 \psi_2 \, \overline{\psi}_2 \, dx \right| \\
& \leq \frac{1}{2n} \|\psi_1\|_2^2 \|\nabla \psi_2\|_2\|\psi_2\|_2,
\end{align*}
and
\begin{align*}
\left|\mbox{Im}\int_{\R^2} A_1(\psi_2) \,\partial_1 \psi_1\, \overline{\psi}_2 \, dx \right| \leq  \frac{1}{2n} \|\psi_2\|_2^2 \|\nabla \psi_1\|_2 \|\psi_1\|_2.
\end{align*}
Thus
\begin{align}\label{estimate2}
\mbox{Im}\int_{\R^2}A_1(\psi) \,\partial_1 \psi \, \overline{\psi} \,dx  = \mbox{Im} \int_{\R^2}A_1(\psi_1) \,\partial_1 \psi_1\, \overline{\psi}_1 + A_1(\psi_2) \,\partial_2 \psi_2 \, \overline{\psi}_2 \,dx +o_n(1).
\end{align}
Similarly, one can show that
\begin{align} \label{estimate3}
\int_{\R^2}A_2^2(\psi)|\psi|^2 \, dx &= \int_{\R^2} A_2^2(\psi_1)|\psi_1|^2 \, dx + \int_{\R^2}A_2^2(\psi_2)|\psi_2|^2 \, dx + o_n(1),
\end{align}
and
\begin{align} \label{estimate4}
\mbox{Im}\int_{\R^2}A_2(\psi) \,\partial_2 \psi \, \overline{\psi} \,dx  &= \mbox{Im}\int_{\R^2}A_2(\psi_1) \,\partial_2 \psi_1\, \overline{\psi}_1 + A_2(\psi_2) \, \partial_2 \psi_2 \, \overline{\psi}_2 \,dx +o_n(1).
\end{align}
Hence, for $n \in \N^+$ large enough, we obtain from \eqref{sub0}-\eqref{sub} and \eqref{estimate1}-\eqref{estimate4} that
$$
m(c) \leq E(\psi) \leq E(\psi_1) + E(\psi_2) + \frac{\eps}{2} \leq m(c_1) + m(c_2) + \eps,
$$
and the proof is completed.
\end{proof}

\begin{rem}
According to the definition of $m(c)$, by a simple scaling technique, it is easy to show that the function $c \mapsto m(c)$ is continuous for any $c \geq 0$.
\end{rem}

We now are ready to establish Theorem \ref{compactness}.

\begin{proof}[Proof of Theorem \ref{compactness}]
As already mentioned, in order to discuss the compactness of any minimizing sequence to \eqref{gmin}, we shall make use of the Lions concentration compactness principle, and our aim is to exclude vanishing and dichotomy of minimizing sequence. Suppose now that $\{u_n\} \subset S(c)$ is a minimizing sequence to \eqref{gmin}. Since $2<p<4$, it then follows from \eqref{scaling} that $m(c) <0$ for any $c>0$. Thus, by using \eqref{MGN}, we deduce that
\begin{align} \label{ddb}
\int_{\R^2}|D_1 u_n|^2 + |D_2 u_n|^2 \, dx \leq C.
\end{align}
In other words,
\begin{align}\label{bdd0}
\begin{split}
&\int_{\R^2} |\nabla u_n|^2 + \left(A_1^2(u_n)+A_2^2(u_n)\right)|u_n|^2  \, dx \\
& \quad + 2 \,\mbox{Im} \int_{\R^2}\left(A_1(u_n)\, \partial_{1} u_n + A_2(u_n) \,\partial_{2} u_n \right)\,\overline{u}_n \, dx \leq C.
\end{split}
\end{align}
Taking into account \eqref{MGN} and \eqref{ddb}, we have that $\|u_n\|_t \leq C$ for any $2 \leq t  < \infty$. In light of \eqref{bdd} and \eqref{bdd1} with $\eps>0$ small enough, it then follows from \eqref{bdd0} that $\|\nabla u_n\|_2 \leq C$, that is, $\{u_n\}$ is bounded in $H^1(\R^2)$. We now apply the Lions concentration compactness Lemma \cite[Lemma I.1]{Li2} to conclude that vanishing does not happen. Otherwise, $\|u_n\|_p=o_n(1)$, this gives that $m(c) \geq 0$, and we then reach a contradiction. We next deduce that dichotomy does not happen, either. Indeed, it suffices to establish the following strict subadditivity inequality
\begin{align} \label{strictsub}
m(c) < m(c_1)+m(c_2)
\end{align}
for any $0<c_1, c_2<c$, and $c_1+c_2=c$. Inspired by \cite{CDSS}, we shall verify that \eqref{strictsub} is valid for any $c>0$ small.

To do that, for any $\theta >0$ and $u \in S(c)$, let us introduce a scaling of $u$ as
\begin{align} \label{sc}
u^{\theta}(x):=\theta^{\frac{1+ 2\beta}{2}}u(\theta^{\beta} x).
\end{align}
It is easy to check that $u^{\theta} \in S(\theta c)$ and
\begin{align*}
E(u^{\theta})&= \frac{\theta^{1 + 2\beta}}{2} \int_{\R^2}|\nabla u|^2 \, dx + \frac{\theta^{3+ 2\beta}}{2} \int_{\R^2}\left(A_1^2(u)+A_2^2(u)\right)|u|^2 \, dx \\
&\quad + \theta^{2+2\beta} \, \mbox{Im} \int_{\R^2}\left(A_1(u)\, \partial_{1} u + A_2(u) \,\partial_{2} u \right)\,\overline{u} \, dx - \lambda \frac{\theta^{\frac{1 + 2 \beta}{2}p-2 \beta }}{p} \int_{\R^2}|u|^p \, dx.
\end{align*}
We now choose $\beta \in \R$ such that $1 + 2 \beta =\frac{1 + 2 \beta}{2} p -2 \beta$, i.e. $\beta= \frac{p-2}{8-2p}$, and $\beta>0$ due to $2<p<4$. Thus it is not difficult to find that
\begin{align} \label{anothermin}
m(c)=c^{1+2 \beta} e^c(1),
\end{align}
where $e^c(\mu)$ is given by the minimization problem
\begin{align} \label{mini}
e^c(\mu):=\inf_{u \in S(\mu)} \mathcal{E}^c(u),
\end{align}
and
\begin{align*}
\mathcal{E}^c(u)&:= \frac 12 \int_{\R^2}|\nabla u|^2\, dx - \frac{\lambda}{p} \int_{\R^2}|u|^p \, dx+ \frac{c^2}{2} \int_{\R^2}\left(A_1^2 + A_2^2 \right)|u|^2 \, dx  \\
& \quad + c \,\mbox{Im} \int_{\R^2}\left(A_1 \partial_{1} u + A_2 \partial_{2} u \right)\,\overline{u} \, dx.
\end{align*}
Via simple scaling arguments, we derive that
\begin{align} \label{scal1}
e^c(\mu)=\mu^{1 + 2\beta} e^{\mu c}(1).
\end{align}
In addition, by using the similar way of proving Lemma \ref{subadd}, we know that
\begin{align} \label{subadd1}
e^c(\mu) \leq  e^c(\eta)+ e^c(\mu-\eta)
\end{align}
holds for any $0 \leq \eta \leq \mu$. Note that $2<p<4$, then $e^0(1)<0$. By using the approach of the proof of \cite[Lemma II. 1]{Li1}, it is standard that the following strict subadditivity holds true,
\begin{align} \label{strict1}
e^0(1)<e^0(\eta) + e^0(1-\eta)
\end{align}
for any $0<\eta <1$. By applying \eqref{anothermin} and \eqref{scal1}, then the strict subadditivity inequality \eqref{strictsub} is equivalent to
\begin{align} \label{strictineq}
e^c(1) < e^c(\xi) + e^c(1-\xi),
\end{align}
where $0<\xi:=\frac{c_1}{c}<1$.

We now verify that \eqref{strictineq} holds for any $c>0$ small. To this end, we shall argue by contradiction that there exist a sequence $\{c_n\} \subset \R^+:=(0, \infty)$ with $c_n =o_n(1)$ and a sequence $\{\xi_n\} \subset \R^+$ with $0 <\xi_n <1$ such that
\begin{align} \label{ide}
e^{c_n}(1)=e^{c_n}(\xi_n) +e^{c_n}(1-\xi_n).
\end{align}
This is because the subadditivity inequality \eqref{subadd1} with $\mu=1$ always holds. Without restriction, we may suppose that $\frac 12 \leq \xi_n < 1$, otherwise one can replace the roles of $\xi_n$ by $1-\xi_n$. Furthermore, we may choose that
\begin{align} \label{defxi}
\xi_n=\inf \left\{\xi \in [\frac 12, 1): e^{c_n}(1)=e^{c_n}(\xi) +e^{c_n}(1-\xi)\right\}.
\end{align}
We first consider the case that $\xi_n \to \xi_0<1$ as $n \to \infty$. Notice that $e^c(1)\to e^0(1)$ as $c \to 0$, then \eqref{ide} implies that
$$
e^{0}(1)=e^{0}(\xi_0) +e^{0}(1-\xi_0),
$$
this contradicts \eqref{strict1}. We next consider the case that $\xi_n \to 1$ as $n \to \infty$. In this case, let us first claim that
\begin{align}\label{ineq1}
e^{c_n}(\xi_n)<e^{c_n}(\eta) +e^{c_n}(\xi_n-\eta)
\end{align}
for any $0<\eta <\xi_n$. To prove this claim, we argue again by contradiction that \eqref{ineq1} were false, thus there would exist a sequence $\{\eta_n\} \subset \R^+$ with $\frac 12 \xi_n \leq \eta_n<\xi_n$ such that
\begin{align} \label{ineq2}
e^{c_n}(\xi_n)=e^{c_n}(\eta_n) +e^{c_n}(\xi_n-\eta_n).
\end{align}
Therefore, from \eqref{subadd1}, \eqref{ide} and \eqref{ineq2}, we obtain that
\begin{align*}
e^{c_n}(\eta_n)+e^{c_n}(1-\eta_n) \geq e^{c_n}(1)&=e^{c_n}(\xi_n) + e^{c_n}(1-\xi_n)\\
&=e^{c_n}(\eta_n) +e^{c_n}(\xi_n-\eta_n) + e^{c_n}(1-\xi_n) \\
& \geq e^{c_n}(\eta_n) + e^{c_n}(1-\eta_n),
\end{align*}
this infers that
\begin{align} \label{ide1}
e^{c_n}(1)=e^{c_n}(\eta_n)+e^{c_n}(1-\eta_n).
\end{align}
On one hand, we know that $ \eta_n \geq \frac 12 \xi_n \geq \frac 14$. On the other hand, by the definition of $\xi_n$, see \eqref{defxi}, the assumption that $\eta_n<\xi_n$, and \eqref{ide1}, we conclude that $\eta_n < \frac 12$. Hence there is a constant $\frac 14 \leq \eta_0 \leq \frac 12$ such that $\eta_n=\eta_0+o_n(1)$, and by using \eqref{ide1}, we then get that
$$
e^0(1)=e^{0}(\eta_0)+e^{0}(1-\eta_0),
$$
this contradicts \eqref{strict1}. Thus the claim follows, namely \eqref{ineq1} holds for any $0<\eta <\xi_n$. Noticing that $e^{c_n}(\xi_n)<0$ and $2<p<4$, from the Lions concentration compactness principle, we then infer that the infimum to \eqref{mini} with $c=c_n, \mu=\xi_n$ is attained, hence there is $w_n \in S(\xi_n)$ so that $\mathcal{E}^{c_n}(w_n)=e^{c_n}(\xi_n)$, and $w_n$ fulfills the following equation
\begin{align} \label{eq1}
\begin{split}
-\Delta w_n &+ \alpha_n w_n + c_n^2\left(A_1^2(w_n) + A_2^2(w_n)\right) w_n +  c_n^2 A_0(w_n) \\
&+ 2 c_n i \, \left(A_1(w_n) \partial_1 (w_n) + A_2(w_n)\partial_2 (w_n)\right) = \lambda |w_n|^{p-2}w_n,
\end{split}
\end{align}
where $\alpha_n \in \R$ is the Lagrange multiplier associated to the constraint $S(\xi_n)$. Recall that $c_n=o_n(1)$, then $\{{\xi_n}^{-\frac 12 }w_n\} \subset S(1)$ is a minimizing sequence to \eqref{mini} with $c=0, \mu=1$. Thus, by using the fact that $e^0(1)<0$ and \eqref{strict1}, we deduce from the Lions concentration compactness principle that $\{{\xi_n}^{-\frac 12 }w_n\}$ is compact in $H^1(\R^2)$ up to translations. Since $\xi_n=1+o_n(1)$, then there is $w \in S(1)$ such that $w_n \to w$ in $H^1(\R^2)$ as $n \to \infty$, and it result from \eqref{eq1} that $w$ solves the following equation
\begin{align} \label{eq2}
-\Delta w + \alpha w = \lambda |w|^{p-2}w,
\end{align}
where $\alpha=\alpha_n + o_n(1)$. By using the Pohozaev identity associated to the equation \eqref{eq2}, we have that $\alpha >0$. At this point, we make use of the assumption \eqref{ide}, then
\begin{align} \label{id}
\frac{e^{c_n}(1)-e^{c_n}(\xi_n)}{1-\xi_n}=\frac{e^{c_n}(1-\xi_n)}{1-\xi_n}.
\end{align}
By \eqref{scal1}, we have that
\begin{align} \label{contr1}
\frac{e^{c_n}(1-\xi_n)}{1-\xi_n}=(1-\xi_n)^{2 \beta}e^{c_n(1-\xi_n)}(1)=o_n(1),
\end{align}
because of $\xi_n=1+o_n(1)$. On the other hand, by the equations \eqref{eq1}-\eqref{eq2} and the fact that $w_n \to w $ in $H^1(\R^2)$ as $n \to \infty$, it is not difficult to see  that
\begin{align*}
e^{c_n}(1)-e^{c_n}(\xi_n) \leq \mathcal{E}^{c_n}(w)-\mathcal{E}^{c_n}(w_n) = -\frac {\alpha}{2} (1-\xi_n) +o_n(1),
\end{align*}
this implies that
\begin{align} \label{contr2}
\frac{e^{c_n}(1)-e^{c_n}(\xi_n)}{1-\xi_n} \leq -\frac {\alpha}{2} +o_n(1).
\end{align}
By combining \eqref{contr1} and \eqref{contr2}, we then reach a contradiction from \eqref{id}, because of $\alpha >0$. So far, we have proved that the strict subadditivity inequality \eqref{strictineq} holds true for any $c>0$ small.

From the discussion above, the Lions concentration compactness principle reveals that there exists $u \in S(c)$ such that $\|u_n -u\|_2=o_n(1)$ up to translations, and $m(c) \leq E(u)$. Thus $\|u_n -u\|_p =o_n(1)$, this yields that $E(u) \leq m(c) $, hence $E(u)=m(c)$. Since $E(u_n)=m(c)+o_n(1)$, then $E(u_n)=E(u)+o_n(1)$, from which we obtain that $\|\nabla u_n - \nabla u\|_2=o_n(1)$. Thus the proof is complete.
\end{proof}

\begin{rem}
According to the definition \eqref{anothermin}, for any $0<c<c_0$, we have that $u \in S(c)$ is a minimizer to \eqref{gmin} if and only if $u^{\frac 1 c}$ is a minimizer to \eqref{mini} with $\mu =1$, where $u^{\frac 1 c}$ is defined by \eqref{sc}.

\end{rem}

The next subsection is devoted to discussing the radial symmetry and uniqueness of minimizer to \eqref{gmin}. To this purpose, let us first fix some notations. We denote by $Q \in H^1(\R^2)$ the unique radially symmetric solution to the equation
$$
-\Delta Q + \alpha_0 Q=\lambda |Q|^{p-1}Q,
$$
where $\alpha_0>0$. In addition, we define a linear operator $L: H^1(\R^2) \to H^{-1}(\R^2)$ by
$$
L(h):=-\Delta h + \alpha_0 h -\lambda(p-1)|Q|^{p-2}h.
$$
It is well-known that the operator $L$ is nondegenerate, see \cite[Lemma 4.2]{NiTa}, thus
$$
T_{Q}:=\text{ker}(L)=\text{span}\left\{\frac{\partial Q}{\partial x_1}, \frac{\partial Q}{\partial x_2}\right\}=\left\{\alpha_1 \frac{\partial Q}{\partial x_1}+ \alpha_2 \frac{\partial Q}{\partial x_2}: \mbox{for all}\,\,\, \alpha_1, \alpha_2 \in \mathbb{C} \right\}.
$$
For any $\tau \in \R^2$, we set $Q_{\tau}(x):=Q(x+ \tau)$, and denote by $T_{Q_{\tau}}^{\bot}$ the orthogonal space of $T_{Q_{\tau}}$ in $H^1(\R^2)$ with respect to $L^2(\R^2)$ scalar product. \medskip

We now show the proof of Theorem \ref{symmetry}.

\begin{proof}[Proof of Theorem \ref{symmetry}]
To establish this theorem, we shall borrow some ingredients developed in \cite{GPV}. The proof is divided into the following steps. \smallskip

{\bf Step 1:} {\it Claim that there exists a constant $\eps_1>0$ so that, for any $u \in B_{\eps_1}(Q):=\{u \in H^1(\R^2):\|u-Q\| \leq \eps_1\},$ there are a unique $\tau(u) \in \R^2$ and a unique $R(u) \in T_{Q_{\tau(u)}}^{\bot}$ such that
$$
u= Q_{\tau(u)} + R(u).
$$
In addition, $\tau \in C^1(B_{\eps_1}(Q), \R^2)$ and $R \in C^1(B_{\eps_1}(Q), H^1(\R^2))$ with $\tau(Q)=0, R(Q)=0$. Here we define an operator $P: B_{\eps_1}(Q) \to H^1(\R^2)$ by $P(u):=Q_{\tau(u)}$, and $P \in C^1(B_{\eps_1}(Q), H^1(\R^2))$.} \smallskip

To prove this claim, let us define a map $\Phi_1: H^1(\R^2) \times H^1(\R^2) \to H^1(\R^2) \times \R^2$ by
$$
\Phi_1(Q_{\tau}, h):=\left(Q_{\tau} + h, \int_{\R^2} v_1(x+\tau) \, \overline{h}\, dx, \int_{\R^2} v_2(x + \tau) \, \overline{h}\, dx \right),
$$
where $v_1, v_2 $ belong to $H^1(\R^2)$ such that $\text{span} \{v_1, v_2\}=T_{Q}.$  It is simple to find that $\Phi_1(Q, 0)=(Q, 0)$, and for any $h, k \in H^1(\R^2)$,
\begin{align} \label{Qdef}
\Phi_1'(Q, 0)(h, k)=\left(h+k, \int_{\R^2} {v}_1  \, \overline{k} dx, \int_{\R^2} {v}_2 \, \overline{k} dx\right).
\end{align}
We shall prove that the linear operator $\Phi_1'(Q, 0)\in L \left(T_{Q} \times H^1(\R^2), H^1(\R^2) \times \R^2\right)$ is invertible. The boundedness of the operator $\Phi_1'(Q, 0)$ is obvious by the definition \eqref{Qdef}. Observe that
\begin{align*}
&\left\{\Phi_1'(Q, 0)(-h, h): h \in T_Q\right\}=\left\{0\right\} \times \R^2, \\
&\left\{\Phi_1'(Q, 0)(0, k): k \in T_{Q}^{\bot}\right\}=T_{Q}^{\bot} \times \left\{0 \right\}, \\
&\left\{\Phi_1'(Q, 0)(h, 0): h \in T_{Q}\right\}=T_{Q} \times \left\{0 \right\}.
\end{align*}
This means that $\Phi_1'(Q, 0)$ is surjective. We next deduce that $\Phi_1'(Q, 0)$ is injective. Assume that there is $(h, k) \in T_Q \times H^1(\R^2)$ such that $\Phi_1'(Q, 0)(h, k)=(0, 0)$, from \eqref{Qdef}, then $h+k=0$ and $k \in T_Q^{\bot}$. Noting that $h \in T_Q$, then we have that $h=k=0$. Thus $\Phi_1'(Q, 0)$ is injective. By using the inverse function theorem, the claim then follows.\smallskip

{\bf Step 2:} {\it Claim that there exist a constant $\tilde{c}_1>0$ and a constant $\eps_2>0$ such that, for any $c \in (0, \tilde{c}_1)$ and for any $\alpha \in (\alpha_0-\eps_2, \alpha_0 + \eps_2)$, there is a solution $w=w(c, \alpha) \in H^1_{rad}(\R^2)$ to the equation
\begin{align*}
-\Delta w + \alpha w &+ c^2\left(A_1^2(w) + A_2^2(w)\right) w +  c^2 A_0(w)w \\
&+ 2c \,i \,\left(A_1(w) \,\partial_1 w + A_2(w)\,\partial_2 w\right) = \lambda |w|^{p-2}w.
\end{align*}
In addition, $w(c, \alpha) \to Q$ in $H^1_{rad}(\R^2)$ as $(c, \alpha) \to (0, \alpha_0)$ in $\R^2$.} \smallskip

To achieve this, we introduce a map $\Phi_2:[0, \infty) \times [0, \infty) \times  H^1_{rad}(\R^2) \to  H^{-1}_{rad}(\R^2)$ as
$$
\Phi_2(c, \alpha, w):=\nabla_w \Gamma^c(w),
$$
where the underlying energy function $\Gamma^c$ is defined by
\begin{align} \label{defgamma}
\begin{split}
\Gamma^{c}(w)&:=\frac 12 \int_{\R^2}|\nabla w|^2\, dx + \frac{\alpha}{2} \int_{\R^2}|w|^2 \, dx + \frac{c^2}{2} \int_{\R^2}\left(A_1^2(w)+A_2^2(w)\right)|w|^2 \, dx \\
& \quad +c \,\mbox{Im} \int_{\R^2}\left(A_1(w)\, \partial_{1} w + A_2(w)\, \partial_{2} w\right)\,\overline{w} \, dx -\frac {\lambda}{p} \int_{\R^2}|w|^p \,dx.
\end{split}
\end{align}
Notice that $\Phi_2(0, \alpha_0, Q)=0$, and for any $h \in H^1_{rad}(\R^2)$,
$$
\Phi_{2, w}'(0, \alpha_0, Q)h=-\Delta h + \alpha_0 h-\lambda (p-1)|Q|^{p-2} h.
$$
We shall deduce that the linear operator $\Phi_{2, u}'(0, \alpha_0, Q) \in L\left(H^1_{rad}(\R^2), H^{-1}_{rad}(\R^2)\right)$ is invertible. Since the embedding $H^1_{rad}(\R^2) \hookrightarrow L^t(\R^2)$ is compact for any $2<t<\infty$, and the function $Q$ decays exponentially as $|x|$ goes to infinity, then it is easy to see that $\Phi_{2, u}'(0, \alpha_0, Q)$ is surjective, namely, for any $f \in H^{-1}_{rad}(\R^2)$, there exists $h \in H^1_{rad}(\R^2)$ solving the linear equation
$$
-\Delta h + \alpha_0 h-\lambda (p-1)|Q|^{p-2} h=f.
$$
Recall that $Q$ is nondegenerate, then $\text{ker}(\Phi_{2, u}'(0, \alpha_0, Q))=0$, which yields that $\Phi_{2, u}'(0, \alpha_0, Q)$ is injective. From the implicit function theorem and the principle of symmetric criticality, see Remark \ref{invar}, the claim then follows. \smallskip

{\bf Step 3:} {\it Claim that there exist a constant $\tilde{c}_2>0$ and two constants $\eps_3, \eps_4>0$ such that, for any $c \in (0, \tilde{c}_2)$ and for any $\alpha \in (\alpha_0 -\eps_3, \alpha_0 + \eps_3)$, there exists a unique $u=u(c, \alpha) \in B_{\eps_4}(Q)$ so that
$$
P(u)=Q, \quad \pi_{T_Q^{\bot}}(\nabla_u \Gamma^c(u))=0,
$$
where the operator $P$ is defined in Step 1, and $\pi: H^1(\R^2) \to T_Q^{\bot}$ stands for the orthogonal projection onto $T_Q^{\bot}$ with respect to $L^2(\R^2)$ scalar product.} \smallskip

To prove this, we define a map $\Phi_3:[0, \infty) \times [0, \infty) \times B_{\eps_1}(Q) \to T_Q^{\bot} \times H^1(\R^2)$ by
$$
\Phi_3(c, \alpha, u):=\left(\pi_{T_Q^{\bot}}\left(\nabla_u \Gamma^c(u)\right), P(u)\right),
$$
where the constant $\eps_1>0$ is given in Step 1, and the energy functional $\Gamma^c$ is defined by \eqref{defgamma}. By the definition of the operator $P$, we know that $P(Q)=Q$ and $P(Q_{\tau})=Q_{\tau}$, then it is immediate to see that $\Phi_3(0, \alpha_0, Q_{\tau})=(0, Q_{\tau})$. Therefore, for any $h \in T_{Q}$,
\begin{align} \label{surj}
\Phi_{3, u}'(0, \alpha_0, Q) h=\left(0, h \right).
\end{align}
We shall assert that $\Phi_{3, u}'(0, \alpha_0, Q)\in L\left(H^1(\R^2), T_Q^{\bot} \times T_Q\right)$ is invertible. By the definition of the map $\Phi_3$, it is not difficult to verify that $\Phi_{3, u}'(0, \alpha_0, Q) \in L\left( T_Q^{\bot}, T_Q^{\bot} \right)$ is surjective. This along with \eqref{surj} indicates that $\Phi_{3, u}'(0, \alpha_0, Q)\in L\left(H^1(\R^2), T_Q^{\bot} \times T_Q\right)$ is surjective. We next assume that there is $h \in H^1(\R^2)$ such that $\Phi_{3, u}'(0, \alpha_0, Q) h=(0, 0)$, which implies that $h \in  T_{Q}$. From \eqref{surj}, it results that $h =0$. Consequently, $\Phi_{3, u}'(0, \alpha_0, Q)$ is invertible. Hence the claim follows by the implicit function theorem. \smallskip

{\bf Step 4:} {\it Prove that, for any $c>0$ small, every minimizer to \eqref{gmin} is radially symmetric up to translations.} \smallskip

Indeed, it is equivalent to prove that, for any $c>0$ small, every minimizer to \eqref{mini} with $\mu=1$ is radially symmetric up to translations. Assume that $u \in S(1)$ is a minimizer to \eqref{mini} with $\mu=1$ , then $u$ enjoys the following equation
\begin{align} \label{equation}
\begin{split}
-\Delta u + \alpha u &+ c^2\left(A_1^2 + A_2^2\right) u +  c^2 A_0 u \\
&+ 2c \, i\,\left(A_1 \,\partial_1 u + A_2\, \partial_2 u\right) = \lambda |u|^{p-2}u,
\end{split}
\end{align}
where $\alpha \in \R$ is the associated Lagrange multiplier. Accordingly, $\nabla_u \Gamma^c(u)=0$. Since $u$ fulfills the equation \eqref{equation}, then it is not hard to prove that there is a constant $\tilde{c}>0$ small with $0<\tilde{c} \leq c_0$ such that,  for any $0<c <\tilde{c}$, $u \in B_{\eps}(Q)$ and $\alpha \in (\alpha_0-\eps, \alpha_0+\eps)$, where the constant $\eps < \eps_k$ for $1 \leq k \leq 4$. Thus it follows from Step 1 that there exist a constant $\tau \in \R^2$ and $R(u) \in T_{Q_{\tau}}^{\bot}$ such that $u=Q_{\tau}+R(u)$, i.e. $u(x-\tau)=Q+R(u(x-\tau))$. Hence
\begin{align}\label{pro}
P(u(x-\tau))=Q, \quad \pi_{T_Q^{\bot}}\left(\nabla_{u(x-\tau)}\Gamma(u)\right)=0,
\end{align}
where the second one is a consequence of the fact that the energy functional $\Gamma^c$ is invariant under any translation in $\R^2$. Furthermore, one can check that the solution $w(c, \alpha) \in H_{rad}^1(\R^2)$ obtained in Step 2 satisfies \eqref{pro} as well. Consequently, by Step 3, we know that $u(x-\tau)=w(c,\alpha)$, which then suggests that $u$ is radially symmetric up to translations. \smallskip

{\bf Step 5:} {\it Prove that, for any $c>0$ small, minimizer to \eqref{gmin} is unique up to translations.} \smallskip

In fact, it is equivalent to show that, for any $c>0$ small, minimizer to \eqref{mini} with $\mu=1$ is unique up to translations. To this end, we assume that there are two minimizers $u_1, u_2 \in S(1)$ to \eqref{mini} with $\mu=1$. Plainly, $u_j$ solves the equation \eqref{equation}, and $\nabla_{u_j} \Gamma^c(u_j)=0$ for $j=1, 2.$ Thus, for any $0<c <\tilde{c}$, one can deduce that there exist constants $\tau_j \in \R^2$ and $R(u_j) \in T_{Q_{\tau_j}}^{\bot}$ such that $u_j=Q_{\tau_j} + R(u_j)$, where the constant $\tilde{c}$ is given in Step 4. This means that $u_j(x-\tau_j)=Q+R(u_j(x-\tau_j))$, then
\begin{align*}
P(u_j(x-\tau_j))=Q, \quad \pi_{T_Q^{\bot}}\left(\nabla_{u_j(x-\tau_j)}\Gamma(u)\right)=0.
\end{align*}
Hence, as a result of Step 3, we have that $u_1(x)=u_2(x-\tau_2 + \tau_1)$, and the proof is completed.
\end{proof}

\section{The mass critical case} \label{masscri}

The aim of this section is to study the existence and nonexistence of solution to \eqref{system1}-\eqref{mass} in the mass critical case. First of all, let us show some elementary observations. By the definitions of $D_j$ for $j=1, 2$, it is straightforward to check that
\begin{align} \label{observe1}
\begin{split}
\int_{\R^2} D_jD_j u\, \overline{u} \, dx &=\int_{\R^2} \left(\partial_1 + i A_j\right)D_j u\, \overline{u} \, dx
=-\int_{\R^2} D_j u \overline{D_j u} \, dx \\
&=-\int_{\R^2} |D_j u|^2 \, dx,
\end{split}
\end{align}
and
\begin{align}\label{observe2}
\begin{split}
\int_{\R^2}\left(D_1 - i \, D_2\right)\left(D_1 + i\, D_2\right) u \, \overline{u} \,dx &=-\int_{\R^2}\left(D_1 + i\, D_2\right) u \, \overline{\left(D_1 - i \, D_2\right) u} \,dx \\
&=-\int_{\R^2}|\left(D_1 + i\, D_2\right) u|^2\,dx.
\end{split}
\end{align}
Moreover, from the system \eqref{defA} satisfied by $A_j$, there holds that
\begin{align}\label{Aid2}
\partial_1 A_2-\partial_2 A_1=-\frac 12 |u|^2.
\end{align}
Notice that
\begin{align*}
\left(D_1D_1 + D_2D_2\right)u =\left(D_1 - iD_2\right)\left(D_1 + iD_2\right) u + \left(\partial_1 A_2-\partial_2 A_1\right) u,
\end{align*}
then, from \eqref{observe1}-\eqref{Aid2},
\begin{align*}
\int_{\R^2} |D_1 u|^2 + |D_2 u|^2 \, dx = \int_{\R^2}|\left(D_1 + i\, D_2\right) u|^2\,dx + \frac 12 \int_{\R^2}|u|^4 \, dx.
\end{align*}
This yields that\begin{align} \label{xx}
\begin{split}
E(u)&=\frac{1}{2} \int_{\R^2}|D_1 u|^2 + |D_2 u|^2 \, dx -\frac {\lambda}{4} \int_{\R^2}|u|^4 \, dx  \\
&=\frac 12 \int_{\R^2}|\left(D_1 + i\, D_2\right) u|^2\,dx + \frac {1 -\lambda} {4} \int_{\R^2}|u|^4 \, dx.
\end{split}
\end{align}

We now turn to the proof of Theorem \ref{nonexistence}.

\begin{proof}[Proof of Theorem \ref{nonexistence}]
(i) Assume that $\lambda<1$. Let us first prove that $m(c)=0$ for any $c>0$. As a result of \eqref{scaling}, one easily concludes that $E(u_t) \to 0$ as $t \to 0^+$, this implies that $m(c) \leq 0$ for any $c>0$. On the other hand, since $\lambda <1$, according to \eqref{xx}, then there holds that $E(u)>0$ for any $u \in S(c)$, this gives that $m(c) \geq 0$ for any $c>0$. Thus $m(c)=0$ for any $c >0$. Observe that $E(u)>0$ for any $u \in S(c)$, from which we derive that that $m(c) =0$ cannot be attained for any $c>0$. Finally, we prove that \eqref{system1}-\eqref{mass} does not admit solution for any $c>0$. Suppose by contradiction that \eqref{system1}-\eqref{mass} has a solution $u \in S(c)$ for some $c>0$. From Lemma \ref{Ph}, we then deduce that $Q(u)=0$, namely $E(u)=0$, which indicates that $m(c)$ is attained, and this is impossible. \smallskip

(ii) Assume that $\lambda=1$. Arguing as the proof of the assertion $(i)$, we first have that $m(c)\leq 0$ for any $c\geq 0$. Using again \eqref{xx}, we know that $E(u) \geq 0$ for any $u \in S(c) $, this shows that $m(c)\geq 0$ for any $c\geq 0$. Hence $m(c)=0$ for any $c >0$. We now derive that $m(c)$ is attained only for $c=8\pi$, and every minimizer $u$ has the explicit expression \eqref{u}. Indeed, assume that there is $u \in S(c)$ such that $E(u)=0$, then $\left(D_1 + i\, D_2\right) u=0 $ by \eqref{xx}, namely,
$$
\left(\partial_{1} + i \partial_2 \right) u =-i \left(A_1 + i A_2\right) u.
$$
Setting $\mathcal{Z}:=\left\{x\in \R^2: u(x)=0\right\}$, then, for any $x \notin \mathcal{Z}$,
$$
\left(\partial_{1} + i \partial_2 \right) \ln u = -i \left(A_1 + i A_2 \right),
$$
hence
\begin{align} \label{massequ1}
\left(\partial_{1} -i \partial_2 \right)\left(\partial_{1} + i \partial_2 \right) \ln u = -i \left(\partial_{1} - i \partial_2 \right) \left(A_1 + i A_2\right).
\end{align}
By the Coulomb gauge condition \eqref{gauge} and \eqref{Aid2}, it then results from \eqref{massequ1} that
$$
-\Delta \ln u= \frac 12 |u|^2 \,\,\, \mbox{for any}\,\,\, x \notin \mathcal{Z}.
$$
This infers that $u$ is real-valued, hence
$$
E(u)=\frac{1}{2} \int_{\R^2}|\nabla u|^2 + \left(A_1^2+A_2^2\right)|u|^2 \, dx -\frac {1}{4} \int_{\R^2}|u|^4 \, dx.
$$
Furthermore, $E(|u|)=0$, because of $E(u)=0$. We now assume without restriction that $u$ is nonnegative, otherwise one can replace $u$ by $|u|$. Recall that $u \in H^1(\R^2)$, by using Lemma \ref{Aineq}, it is easy to find that $A_1^2, A_2^2,$ and $ A_0 \in L_{loc}^1(\R^2)$. In addition, since $E(u)=0$ and $u \neq 0$, then $u$ solves the equation
$$
-\Delta u=a(x) u,
$$
where
$$
a(x):=|u|^{p-2}-\alpha - \left(A_1^2+A_2^2\right)-A_0,
$$
and $\alpha \in \R$ is the associated Lagrange multiplier. Thus the Br\'{e}zi-Kato theorem, see \cite[B.3 Lemma]{St}, reveals that $u \in L_{loc}^t(\R^2)$ for any $1 \leq t<\infty$, which further gives that $u \in W^{2, t}_{loc}(\R^2)$. As a consequence of the maximum principle in \cite{GiTr}, see also \cite{HaLi}, we then have that $u >0$, that is, $\mathcal{Z}=\emptyset$. At this point, by a change of variable $v=2\ln u$, we assert that $v$ solves the Liouville equation
\begin{align} \label{masseq}
-\Delta v=e^v.
\end{align}
Thanks to $\int_{\R^2} e^v \, dx = \int_{\R^2}|u|^2 \, dx < \infty$, then, from \cite[Theorem 1]{ChLi}, any solution $v$ to \eqref{masseq} has the form
$$
v(x)= \ln \frac{32 \mu ^2}{\left(4 + \mu^2 |x-x_0|^2\right)^2}
$$
for $\mu >0, x_0 \in \R^2$. Therefore,
$$
u(x)=\frac{4 \sqrt{2} \mu}{4 + \mu^2|x-x_0|^2},
$$
and it is immediate to see that $\int_{\R^2} |u|^2 \, dx=8 \pi$. Thus we have proved that $m(c)$ is attained only for $c=8\pi$, and every minimizer has the form \eqref{u}.

Clearly, a function $u$ with the form \eqref{u} is a solution to \eqref{system1}-\eqref{mass} for $c=8\pi$. On the other hand, if $u$ is a solution to \eqref{system1}-\eqref{mass}, then $Q(u)=0$ by Lemma \ref{Ph}, i.e. $E(u)=0$, thus $u \in S(8\pi)$ and it takes the form \eqref{u}. \smallskip

(iii) We begin with showing that the infimum to the minimization problem \eqref{mins} is attained. Let $\{v_n\} \subset \mathcal{P}$ be a minimizing sequence to \eqref{mins}, i.e.
$$
\int_{\R^2}|v_n|^2 \, dx =c^* +o_n(1), \quad E(v_n)=0.
$$
Define $u_n(x):=\eps_n v_n(\eps_n x)$, where
$$
\eps_n^2:= \left(\int_{\R^2}|\nabla v_n|^2 \, dx\right)^{-1}.
$$
It is not difficult to check that $ E(u_n)=0$ and $\int_{\R^2}|u_n|^2 \, dx =c^* +o_n(1)$, hence $\{u_n\}$ is also a minimizing sequence to \eqref{mins}, and $\int_{\R^2}|\nabla u_n|^2 \, dx =1$. Thus $\{u_n\}$ is bounded in $H^1(\R^2)$. In addition, observe that $E(u_n)=0$, then
\begin{align} \label{non}
\lim_{n \to \infty}\int_{\R^2}|u_n|^4 \, dx > 0.
\end{align}
If not, since $\{u_n\}$ is bounded in $H^1(\R^2)$, from Lemma \ref{Aineq} and the H\"older inequality, one then gets that
$$
\mbox{Im} \int_{\R^2}\left(A_1 \partial_{1} u + A_2 \partial_{2} u \right)\,\overline{u} \, dx=o_n(1).
$$
By using the fact that $E(u_n)=0$, we then reach a contradiction, because of $\|\nabla u_n\|_2 =1$. In the following, our aim is to discuss the compactness of the minimizing sequence $\{u_n\}$ in $H^1(\R^2)$ up to translations. To do this, we shall employ a slight variant of the classical Lions concentration compactness principle, then we need to rule out vanishing and dichotomy.

Firstly, we claim that vanishing does not occur. Otherwise, by the Lions concentration compactness Lemma \cite[Lemma I.1]{Li2}, one then obtains that $\|u_n\|_4=o_n(1)$, which contradicts \eqref{non}.

Secondly, we claim that dichotomy does not occur, either. To prove this, for any measurable set $\Omega \subset \R^2$, we introduce a sequence of measures $\mu_n$ as
\begin{align} \label{mun}
\mu_n(\Omega):=\int_{\Omega} |u_n|^2 \, dx,
\end{align}
and $\mu_n(\R^2)=\int_{\R^2}|u_n|^2 \, dx = c^* +o_n(1).$ We now argue by contradiction that dichotomy occurs, then there were a constant $c_0>0$ with $0 <c_0 < c^*$, a sequence $\{\rho_n\} \subset \R$ satisfying $\rho_n \to \infty$ as $n \to \infty$ and a sequence $\{\xi_n\} \subset \R^2$, and two nonnegative measures $\mu_{1, n}, \mu_{2, \mu}$ so that
$$
0 \leq \mu_{1, n} + \mu_{2, n} \leq \mu_n, \,\, \text{supp}\,{\mu_{1, n}} \subset B_{\rho_n}(\xi_n), \,\, \text{supp}\,{\mu_{2, n}} \subset B_{2\rho_n}^c(\xi_n),
$$
$$
\mu_{1, n}(\R^2) = c_0 + o_n(1), \, \, \mu_{2, n}(\R^2) = c^*-c_0 +o_n(1).
$$
Let us define a cut-off function $\chi_n \in C^{\infty}_0(\R^2)$ with $0 \leq \chi_n \leq 1$ such that $\chi_n(x)=1$ for any $x \in B_{\rho_n}(\xi_n)$, $\chi_n(x)=0$ for any $x \in B_{2 \rho_n}^c(\xi_n)$, and $|\nabla \chi_n| \leq {2}/{\rho_n}$. Moreover, we set $u_{1, n}:=\chi_n u_n,$ $u_{2, n}:=\left(1-\chi_n\right)u_n,$ and $u_n=u_{1, n}+u_{2, n}$.
Hence
\begin{align} \label{u1}
\int_{B_{\rho_n}(\xi_n)} |u_n|^2 \, dx =\mu_n(B_{\rho_n}(\xi_n)) \geq \mu_{1, n}(B_{\rho_n}(\xi_n))=c_0+o_n(1),
\end{align}
and
\begin{align} \label{u2}
\int_{B_{2\rho_n}^c(\xi_n)}|u_{n}|^2 \, dx  =\mu_{n}(B_{2\rho_n}^c(\xi_n)) \geq \mu_{2, n}(B_{2\rho_n}^c(\xi_n))=c^*-c_0+o_n(1).
\end{align}
Furthermore,
\begin{align} \label{gredient}
\int_{\R^2} |\nabla u_n|^2 \, dx  \geq \int_{\R^2}|\nabla u_{1, n}|^2 \, dx + \int_{\R^2}| \nabla u_{2, n}|^2 \, dx + o_n(1).
\end{align}
Notice that
\begin{align} \label{measure}
\begin{split}
\mu_n\left(B_{2 \rho_n}(\xi_n)\backslash B_{\rho_n}(\xi_n)\right) &= \mu_n(\R^2)-\mu_n(B_{\rho_n}(\xi_n))-\mu_n(B_{2 \rho_n}^c(\xi_n))\\
& \leq \mu_n(\R^2)-\mu_{1,n}(B_{\rho_n}(\xi_n))-\mu_{2, n}(B_{2 \rho_n}^c(\xi_n))\\
&= \mu_n(\R^2)-\mu_{1, n}(\R^2) -\mu_{2, \mu}(\R^2) \\
& =o_n(1).
\end{split}
\end{align}
By the definitions of the measures $\mu_n$, see \eqref{mun}, we then obtain from \eqref{measure} that
\begin{align} \label{mix}
\int_{B_{2 \rho_n}(\xi_n)\backslash B_{\rho_n}(\xi_n)} |u_n|^2 \, dx =o_n(1).
\end{align}
Since $\{u_n\}$ is bounded in $H^1(\R^2)$, it then follows from the H\"older inequality and \eqref{mix} that
$$
\int_{B_{2 \rho_n}(\xi_n)\backslash B_{\rho_n}(\xi_n)} |u_n|^4 \, dx =o_n(1),
$$
from which we infer that
$$
\int_{\R^2}|u_n|^4 \, dx = \int_{\R^2}|u_{1, n}|^4 \, dx + \int_{\R^2}|u_{2, n}|^4 \, dx + o_n(1).
$$
In addition, with the aid of Lemma \ref{Aineq} and \eqref{mix}, we can deduce that
$$
\int_{\R^2} A_j^2(u_n)|u_n|^2\, dx  =\int_{\R^2} A_j^2(u_{1, n})|u_{1, n}|^2 \, dx + \int_{\R^2} A_j^2(u_{2, n}) |u_{2, n}|^2\, dx + o_n(1),
$$
and
\begin{align*}
\mbox{Im} \int_{\R^2}A_j(u_n) \partial_{1}u_n \,\overline{u}_n \, dx
&=\mbox{Im} \int_{\R^2}A_j(u_{1, n}) \partial_{1}u_{1, n} \,\overline{u}_{1, n} \, dx \\
& \quad + \mbox{Im} \int_{\R^2} A_j(u_{2, n}) \partial_{1}u_{2, n} \,\overline{u}_{2, n} \, dx +o_n(1),
\end{align*}
for $j=1,2$. Recalling that \eqref{gredient}, we arrive at
\begin{align} \label{energy}
E(u_n) \geq E(u_{1, n})+ E(u_{2, n}) + o_n(1).
\end{align}
Since $E(u_n)=0$, it then yields from \eqref{energy} that
\begin{align} \label{dichotomy}
\lim_{n \to \infty} E(u_{1, n}) \leq 0 \,\,\, \mbox{or} \,\,\,\lim_{n \to \infty} E(u_{2, n}) \leq 0.
\end{align}
We now apply \eqref{u1}-\eqref{u2} and \eqref{mix} to conclude that
\begin{align*}
c^*+o_n(1)=\int_{\R^2}|u_n|^2 \,dx &=\int_{B_{\rho_n}(\xi_n)} |u_{n}|^2 \,dx + \int_{B_{2\rho_n}^c(\xi_n)}|u_{n}|^2 \, dx+o_n(1) \\
& \geq c^*+o_n(1).
\end{align*}
By using \eqref{u1}-\eqref{u2} again, then
\begin{align} \label{equality}
\int_{B_{\rho_n}(\xi_n)} |u_n|^2 \, dx =c_0+o_n(1), \,\,\, \int_{B_{2\rho_n}^c(\xi_n)}|u_{n}|^2 \, dx =c^*-c_0+o_n(1).
\end{align}
Let us now come back to \eqref{dichotomy}, and assume without restriction that $\lim_{n \to \infty} E(u_{1, n}) \leq 0$. Therefore, it is easy to find that there exist constants $ 0<\theta_n \leq 1$ such that $E(\theta_n u_{1, n}) =0$. Then, by the definition of $c^*$, \eqref{mix} and \eqref{equality}, we obtain that
\begin{align*}
c^* \leq \int_{\R^2}\theta_n^2 |u_{1, n}|^2 \leq \int_{B_{\rho_n}(\xi_n)} |u_n|^2 \, dx +o_n(1)=c_0+o_n(1) <  c^*,
\end{align*}
which is a contradiction. Thus we have proved that dichotomy does not occur.

Consequently, there exists $u \in H^1(\R^2)$ such that $\|u_n -u\|_2=o_n(1)$ up to translations, which suggests that $\|u_n-u\|_4=o_n(1)$, and $E(u) \leq 0$. If $E(u)=0$, then $u$ is a minimizer to \eqref{mins}, and the proof is completed. Otherwise, there is a constant $0<\theta <1$ such that $E(\theta u)=0$, and we can reach a contradiction as above. Thus we conclude that the infimum to the minimization problem \eqref{mins} is attained, this in turn indicates that \eqref{system1}-\eqref{mass} admits a solution for $c=c^*$. Furthermore, from the definition of $c^*$, we know that $m(c)=0$ for $0<c<c^*$, $m(c)$ cannot be attained for any $0<c<c^*$, and \eqref{system1}-\eqref{mass} has no solution for any $0<c<c^*$.
\end{proof}

\begin{rem} \label{massinfty}
When $p=4,$ $\lambda >1$, by using the similar arguments as the proof of Lemma \ref{subadd}, one can deduce that
\begin{align} \label{criticalsub}
m(c) \leq m(c_1) + m(c_2) \quad \text{for} \,\, c=c_1+c_2.
\end{align}
Since $\lambda >1$, then $E(u)<0$, where $u$ is given by \eqref{u} and $u \in S(8\pi)$. Thus it is easy to derive that there is a constant $0<\theta<1$ such that $E(tu)<0$ for any $\theta<t <1$. By the definition of $c^*$, we have that $\theta^2 \geq {c^*}/{8\pi}$. Since $E(v_t)=t^2E(v)$ for any $v \in S(c)$, then $E(v_t) \to -\infty$ as $t \to \infty$ if $E(v) <0$. Hence there is a constant $c_1 \geq c^*$ so that $m(c)=-\infty$ for any $c_1 < c \leq 8\pi$, where $c_1:=8\pi \theta^2$. Finally, by means of \eqref{criticalsub}, we infer that $m(c)=-\infty$ for any $c >c_1$.
\end{rem}

\section{The mass supercritical case} \label{masssup}

In this section, we are concerned with the existence of solutions to \eqref{system1}-\eqref{mass} in the mass supercritical case. We begin with showing some basic results.

\begin{lem} \label{unique}
Assume $p>4$, then, for any $u \in S(c)$, there exists a unique $t_u>0$ such that $u_{t_u} \in \mathcal{M}(c)$ and $\max_{t >0} E(u_t) = E(u_{t_u})$, furthermore, the function $t \mapsto E(u_t)$ is concave on $[t_u, \infty)$, and $t_u =1$ if $Q(u)=0$, $t_u < 1$ if $Q(u)< 0$.
\end{lem}
\begin{proof}
For any $u \in S(c)$, from \eqref{scaling}, we deduce that
\begin{align} \label{DE}
\begin{split}
\frac{d}{dt}E(u_t)&=t\int_{\R^2}|D_1 u|^2 +|D_2 u|^2 \, dx -\frac{\lambda (p-2)}{p}t^{p-3} \int_{\R^2}|u|^p \, dx \\
& =\frac 1t Q(u_t).
\end{split}
\end{align}
Hence it is clear that there is a unique $t_u >0$ with $Q(u_{t_u})=0$, i.e. $u_{t_u} \in \mathcal{M}(c)$ so that
$$
\frac{d }{dt}E(u_t) >0 \,\,\, \mbox{for}\,\,\, 0 <t< t_u, \quad \frac{d }{dt} E(u_t)<0 \,\,\,\mbox{for} \,\,\, t >t_u,
$$
where $t_u$ is given by
\begin{align} \label{tu}
t_u= \left(\frac{p \int_{\R^2}|D_1 u|^2 +|D_2 u|^2 \, dx}{\lambda(p-2) \int_{\R^2}|u|^p \, dx}\right)^{\frac {1}{p-4}}.
\end{align}
This leads to $\max_{t >0} E(u_t)=E(u_{t_u})$. Since $p>4$, it then yields from \eqref{DE} that
$$
\frac{d^2}{dt^2} E(u_t)<0 \,\,\,\mbox{for} \,\,\, t > t_u,
$$
which shows that the function $t \mapsto E(u_t)$ is concave on $[t_u, \infty)$. Finally, by the definition of $t_u$, see \eqref{tu}, we have that $t_u =1$ if $Q(u)=0$, and $t_u < 1$ if $Q(u)< 0$. Thus the proof is completed.
\end{proof}

\begin{lem} \label{coercive}
Assume $p >4$, then, for any $c >0$, $\gamma(c)>0$ and $E$ restricted on $\mathcal{M}(c)$ is coercive.
\end{lem}
\begin{proof}
For any $u \in \mathcal{M}(c)$, since $Q(u)=0$, by using \eqref{MGN} and the assumption that $p>4$, we then obtain that
\begin{align} \label{D}
\int_{\R^2} |D_1 u|^2 +|D_2 u|^2 \, dx \geq C.
\end{align}
In addition, for any $u \in \mathcal{M}(c)$, there holds that
\begin{align} \label{j}
E(u)=E(u)-\frac{1}{p-2} Q(u) &=\frac{p-4}{2\left(p-2\right)}\int_{\R^2}|D_1 u|^2 +|D_2 u|^2 \, dx.
\end{align}
Consequently, $\gamma(c) >0$ for any $c>0$. Next we shall prove that $E$ restricted on $\mathcal{M}(c)$ is coercive. To do this, we claim that if $\{u_n\} \subset \mathcal{M}(c)$ and $\|\nabla u_n\|_2 \to \infty$ as $n \to \infty$, then $\|\nabla |u_n|\|_2 \to \infty$ as $n \to \infty$. Assume the contrary, by the Gagliardo-Nirenberg inequality \eqref{GN}, we then derive that $\|u_n\|_t \leq C$ for any $2 \leq t < \infty$. Since $Q(u_n)=0$, then, from \eqref{bdd} and \eqref{bdd1} with $\eps>0$ small enough, we know that $\|\nabla u_n\|_2 \leq C$, and this contradicts the assumption. Thus the claim follows. Now, applying the claim and \eqref{MGN}, we conclude from \eqref{j} that $E$ is coercive on $\mathcal{M}(c)$, and the proof is completed.
\end{proof}

\begin{lem} \label{codi}
Assume $p >4$, then $\mathcal{M}(c)$ is a $C^1$ manifold of codimension $2$ in $H^1(\R^2)$, and it is a $C^1$ manifold of codimension $1$ in $S(c)$.
\end{lem}
\begin{proof}
For any $u \in \mathcal{M}(c)$, we know that $P(u):= \int_{\R^2} |u|^2 \, dx -c=0$, $Q(u) =0$, and it is obvious that $P, Q$ are of $C^1$ class. Next we shall prove that, for any $u \in \mathcal{M}(c)$,
$$
\left(P'(u), Q'(u)\right): \mathcal{M}(c) \to \R^2 \mbox{\, is surjection.}
$$
We argue by contradiction that $P'(u)$ and $Q'(u)$ are linearly dependent for some $u \in \mathcal{M}(c)$. This means that there is a constant $\mu \in \R$ such that, for any $ \psi \in H^1(\R^2)$,
\begin{align*}
-2 \int_{\R^2} \left(D_1 D_1 u +D_2D_2 u \right)\overline{\psi} + A_0 \overline{\psi} \, dx = 2\mu \int_{\R^2} u \overline{\psi} \, dx+\lambda (p-2) \int_{\R^2}|u|^{p-2} u  \overline{\psi}\, dx,
\end{align*}
that is,
$$
-\left(D_1 D_1 +D_2D_2 \right) u + A_0 u ={\mu} u + \frac{\lambda (p-2)}{2} |u|^{p-2} u.
$$
Thus $u$ satisfies the Pohozaev identity
\begin{align} \label{above}
\int_{\R^2} |D_1 u|^2 + |D_2 u|^2 \, dx =\frac{\lambda (p-2)^2}{2p} \int_{\R^2} |u|^p \, dx.
\end{align}
Recalling that $Q(u) =0$, we then reach a contradiction from \eqref{above}, because of $u \neq 0$ and $p >4$. Therefore, $\mathcal{M}(c)$ is a $C^1$ manifold of codimension $2$ in $H^1(\R^2)$, and it is a $C^1$ manifold of codimension $1$ in $S(c)$.
\end{proof}

\begin{rem}
As we have already seen, when $p=4$, then Lemma \ref{codi} fails. Thus, in this case, the Pohozaev manifold $\mathcal{M}(c)$ is unavailable, which causes more hard to establish the existence of solutions to \eqref{system1}-\eqref{mass}.

\end{rem}

\begin{defi}\label{homotopy}
Let $B$ be a closed subset of a set $Y \subset H^1(\R^2)$. We say that a class $\mathcal{G}$ of compact subsets of $Y$ is a homotopy stable family with the closed boundary $B$ provided that
\begin{enumerate}
\item [(i)] every set in $\mathcal{G}$ contains $B$;
\item [(ii)] for any $A \in \mathcal{G}$ and any function $\eta \in C([0, 1] \times Y, Y)$ satisfying $\eta(t, x)=x$ for all $(t, x) \in (\{0\} \times Y) \cup([0, 1] \times B)$, then $\eta(\{1\} \times A) \in \mathcal{G}$.
\end{enumerate}
\end{defi}

\begin{lem} \label{ps}
Assume $p>4$. Let $\mathcal{G}$ be a homotopy stable family of compact subsets of $S(c)$ with closed boundary $B$, and let
$$
e_{\mathcal{G}}:=\inf_{A \in \mathcal{G}} \max_{u\in A} F(u),
$$
where $F: S(c) \to \R$ is defined by $F(u):=E(u_{t_u})=\max_{t >0} E(u_t)$. Suppose that $B$ is contained in a connected component of $\mathcal{M}(c)$ and
$$
\max\{\sup F(B), 0\}< e_{\mathcal{G}}< +\infty.
$$
Then there exists a Palais-Smale sequence $\{u_n\} \subset \mathcal{M}(c)$ for $E$ restricted on $S(c)$ at the level $e_{\mathcal{G}}$.
\end{lem}

\begin{proof}
We shall follow the ideas from \cite{BaSo} to prove this lemma. By the definition of $e_{\mathcal{G}}$, there exists a minimizing sequence $\{A_n\} \subset \mathcal{G}$ such that
$$
\max_{u \in A_n} F(u) = e_{\mathcal{G}} + o_n(1).
$$
Let a map $\eta: [0, 1] \times S(c) \to S(c)$ be defined by $\eta(t, u):=u_{1-t + t t_u}$. In view of the definition of $t_u$, see \eqref{tu}, we have that $\eta \in C([0, 1] \times S(c), S(c))$. Since $B \subset \mathcal{M}(c)$, it then follows from Lemma \ref{unique} that $t_u=1$ for any $u \in B$, hence $\eta(t, u)=u$ for any $(t, u) \in \left(\{0\} \times S(c)\right) \cup \left([0, 1] \times B\right)$. Recall that $\mathcal{G}$ is a homotopy stable family of compact subsets of $S(c)$ with closed boundary $B$, by Definition \ref{homotopy}, then $D_n:=\eta(\{1\} \times A_n)=\{u_{t_u}: u \in A_n\} \in \mathcal{G}$. It is immediate to see that $D_n \subset \mathcal{M}(c)$, and for any $v \in D_n$, there is $u \in A_n$ such that $v=u_{t_u}$. Thus $F(v)=F(u_{t_u})=E(u_{t_u})=F(u)$, this means that
$$
\max_{v \in D_n}F(v)=\max_{u \in A_n} F(u).
$$
As a consequence, there is another minimizing sequence $\{D_n\} \subset \mathcal{M}(c)$ such that
$$
\max_{v \in D_n} F(v) =e_{\mathcal{G}} +o_n(1).
$$
Using the equivalent minimax principle \cite[Theorem 3.2]{Gh}, we are able to obtain a Palais-Smale sequence $\{\tilde{u}_n\} \subset S(c)$ for $F$ restricted on $S(c)$ at the level $e_{\mathcal{G}}$ such that $\mbox{dist}_{H^1}(\tilde{u}_n, D_n) =o_n(1)$.

We now set $u_n:=({\tilde{u}_n})_{\tilde{t}_n} \in \mathcal{M}(c)$, where $\tilde{t}_n:=t_{\tilde{u}_n}$. At this point, to end the proof, it suffices to deduce that $\{u_n\}$ is a Palais-Smale sequence for $E$ restricted on $S(c)$ at the level $e_{\mathcal{G}}$. To this end, we first claim that there exist two constants $C_1, C_2>0$ such that $C_1 \leq \tilde{t}_n \leq C_2$. Indeed, in view of the definition of the functional $F$, we have that $E({u}_n)=F(\tilde{u}_n)=e_{\mathcal{G}} +o_n(1)$. Noticing that $e_{\mathcal{G}}< \infty$ and $E$ restricted on $\mathcal{M}(c)$ is coercive by Lemma \ref{coercive}, we then know that $\{u_n\}$ is bounded in $H^1(\R^2)$. In particular, $\{\|\nabla u_n\|_2\}$ is bounded from above. Moreover, thanks to $Q(u_n) =0$, we know that $\{\|\nabla u_n\|_2\}$ is bounded from below by a positive constant. Otherwise, by using Lemma \ref{Aineq}, then
$$
\int_{\R^2} |D_1 u_n|^2 +|D_2 u_n|^2 \, dx=o_n(1),
$$
and this contradicts \eqref{D}. Hence $\{\|\nabla u_n\|_2\}$ is bounded from upper and below by a positive constant. On the other hand, recalling that $\{D_n\} \subset \mathcal{M}(c)$ is a minimizing sequence for $F$ at the level $e_{\mathcal{G}}$, using again Lemma \ref{coercive} and the fact that $\mbox{dist}_{H^1}(\tilde{u}_n, D_n) =o_n(1)$, we can derive that $\{\|\nabla \tilde{u}_n\|_2\}$ is bounded from upper and below by a positive constant. Notice that
\begin{align*}
\tilde{t}_n^2= \frac{\int_{\R^2} |\nabla {u}_n|^2 \, dx }{\int_{\R^2} |\nabla \tilde{u}_n|^2 \, dx},
\end{align*}
the claim then follows.

Since
$$
\int_{\R^2} u\psi \, dx =\int_{\R^2}u_{t_u} \psi_{t_u} \, dx,
$$
then $\psi_{t_u} \in T_{u_{t_u}} S(c)$ if and only if $ \psi \in T_uS(c)$. Thus it is not difficult to find that a map $T_uS(c) \to T_{u_{t_u}} S(c)$ defined by $\psi \mapsto \psi_{t_u}$ is an isomorphism with the inverse $\psi \mapsto \psi_{t^{-1}_u}$. Furthermore, by using the definition that $F(u)=\max_{t >0} E(u_t)=E(u_{t_u})$, one can check that $F'(u) \psi=E'(u_{t_u})\psi_{t_u}$ for any $u \in S(c)$ and $\psi \in T_u S(c)$.
Hence
\begin{align*}
\|E_{\mid_{S(c)}}'(u_n)\| &= \sup_{\psi \in T_{u_n}S(c), \|\psi\| \leq 1} |E'(u_n) \psi| = \sup_{\psi \in T_{u_n}S(c), \|\psi\| \leq 1} |E'(u_n) (\psi_{{t^{-1}_n}})_{t_n}| \\
&=\sup_{\psi \in T_{u_n} S(c), \|\psi\| \leq 1} |F'(\tilde{u}_n)(\psi_{{t^{-1}_n}})|.
\end{align*}
Observe that $\psi_{{t^{-1}_n}} \in T_{\tilde{u}_n}S(c)$ if and only if $\psi \in T_{u_n} S(c)$. In addition, it follows from the claim that $\|\psi_{t^{-1}_n}\| \leq C \|\psi\|$. Thus
\begin{align} \label{pssequence}
\|E_{\mid_{S(c)}}'(u_n)\| \leq \sup_{\psi_{{t^{-1}_n}} \in T_{\tilde{u}_n} S(c), \|\psi_{{t^{-1}_n}}\| \leq C} |F'(\tilde{u}_n)\psi_{{t^{-1}_n}}|
= C\sup_{\phi \in T_{\tilde{u}_n} S(c), \|\phi \| \leq 1} |F'(\tilde{u}_n)\phi|.
\end{align}
Recall that $\{\tilde{u}_n\} \subset S(c)$ is a Palais-Smale sequence for $F$ restricted on $S(c)$ at the level $e_{\mathcal{G}}$, then, from \eqref{pssequence}, we obtain that $\{u_n\} \subset \mathcal{M}(c)$ is a Palais-Smale sequence for $E$ restricted on $S(c)$ at the level $e_{\mathcal{G}}$, and the proof is completed.
\end{proof}

\begin{lem}
Assume $p >4$, then there exists a Palais-Smale sequence $\{u_n\} \subset \mathcal{M}(c)$ for $E$ restricted on $S(c)$ at the level $\gamma(c)$.
\end{lem}
\begin{proof}
Let $\mathcal{G}$ be all singletons in $S(c)$, $B=\emptyset$, and it is easy to check that $\mathcal{G}$ is a homotopy stable family of compact subsets of $S(c)$ without boundary. Moreover,
$$
e_{\mathcal{G}}=\inf_{A \in \mathcal{G}} \max_{u \in A} F(u)=\inf_{u \in S(c)} \max_{t >0} E(u_t).
$$
We next prove that $e_{\mathcal{G}} =\gamma(c)$. By Lemma \ref{unique}, on one hand, we know that, for any $u \in S(c)$, there exists a unique $t_u >0$ such that $u_{t_u} \in \mathcal{M}$ and $\max_{t >0}E(u_t)=E(u_{t_u})$, then
$$
\inf_{u \in S(c)} \max_{t>0} E(u_t) \geq \inf_{u \in \mathcal{M}(c)} E(u).
$$
On the other hand, for any $u \in \mathcal{M}(c)$, we have that
$$
\inf_{u \in \mathcal{M}(c)} E(u)\geq \inf_{u \in S(c)} \max_{t >0} E(u_t).
$$
Thus $e_{\mathcal{G}} =\gamma(c)$. It then follows from Lemma \ref{ps} that the lemma necessarily holds.
\end{proof}

\begin{lem} \label{pre}
Assume $p>4$. Let $\{u_n\} \subset \mathcal{M}(c)$ be a Palais-Smale sequence for $E$ restricted on $S(c)$ at the level $\gamma(c)$, then there exist a nontrivial $u_c \in  H^1(\R^2),$ a sequence $\{\alpha_n\} \subset \R$ and a constant $\alpha_c \in \R$ such that, up to translations,
\begin{enumerate}
\item [(i)] $u_n \rightharpoonup u_c$ in $H^1(\R^2)$ as $n \to \infty$;
\item [(ii)] $\alpha_n \to \alpha_c$ in $\R$ as $n \to \infty$;
\item [(iii)] $-\left(D_1D_1 + D_2D_2\right) u_n+ A_0(u_n) u_n + \alpha_n u_n- \lambda|u_n|^{p-2}u_n=o_n(1)$;
\item [(iv)] $-\left(D_1D_1 + D_2D_2\right) u_c+ A_0(u_c) u_c + \alpha_c u_c- \lambda|u_c|^{p-2}u_c=0$.
\end{enumerate}
In addition, if $\|u_n -u_c\|_p =o_n(1)$ and $\alpha_c >0$, then $\|u_n- u_c\|=o_n(1)$.
\end{lem}
\begin{proof}
By Lemma \ref{coercive}, we know that $\gamma(c)>0$ and $\{u_n\}$ is bounded in $H^1(\R^2)$, it then follows from the Lions concentration compactness Lemma \cite[Lemma I.1]{Li2} that there exists a nontrivial $u_c \in H^1(\R^2)$ such that $u_n\rightharpoonup u_c$ in $H^1(\R^2)$ as $n \to \infty$, up to translations. Otherwise, one has that $\|u_n\|_p=o_n(1)$, this then leads to $\gamma(c)=0$ by the fact that $Q(u_n)=0$, which is impossible. Thus the assertion $(i)$ follows. Recalling that $\{u_n\}$ is bounded in $H^1(\R^2)$ and $\|E_{|_{S(c)}}'(u_n)\|=o_n(1)$, we then obtain that $\|E'(u_n)-\left(E'(u_n) u_n\right) u_n\|=o_n(1)$, this means that, for any $\psi \in H^1(\R^2)$,
\begin{align*}
-\int_{\R^2} \left(D_1D_1 u_n +D_2D_2 u_n\right)  \overline{\psi} \, dx & + \int_{\R^2} A_0(u_n)u_n \overline{\psi} \,dx + \alpha_n \int_{\R^2}u_n \overline{\psi} \, dx \\
&- \lambda \int_{\R^2} |u_n|^{p-2} u_n \overline{\psi}=o_n(1),
\end{align*}
where
\begin{align} \label{an}
\begin{split}
\alpha_n:= -\frac {1}{c} \left(\int_{\R^2} \left( |D_1 u_n|^2 + |D_2 u_n|^2 \right) + A_0(u_n)|u_n|^2 \, dx -\lambda \int_{\R^2} |u_n|^p \, dx \right).
\end{split}
\end{align}
Thus the assertion $(iii)$ follows. From \eqref{an} and the boundedness of $\{u_n\}$ in $H^1(\R^2)$, we obtain that the sequence $\{\alpha_n\} \subset \R$ is bounded in $\R$, hence the assertion $(ii)$ follows. Since $u_n \rightharpoonup u_c$ in $H^1(\R^2)$ as $n \to \infty$, by using Lemmas \ref{BL}-\ref{BL1}, we then know from the assertion $(iii)$ that the assertion $(iv)$ follows.

If $\|u_n- u\|_p=o_n(1)$, it then yields from the assertions $(ii)$-$(iv)$ that
\begin{align} \label{com}
&\int_{\R^2} \left( |D_1 u_n|^2 + |D_2 u_n|^2 \right) + A_0(u_n)|u_n|^2 \, dx + \alpha_n \int_{\R^2} |u_n|^2 \, dx \\
&= \int_{\R^2} \left( |D_1 u_c|^2 + |D_2 u_c|^2 \right) + A_0(u_c)|u_c|^2 \, dx + \alpha_c \int_{\R^2}|u_c|^2 \, dx + o_n(1).
\end{align}
Since $\alpha_c >0$ and $u_n \rightharpoonup u_c$ in $H^1(\R^2)$ as $n \to \infty$, then \eqref{com} indicates that $\|u_n -u_c\|=o_n(1)$. Thus the proof is completed.
\end{proof}

\begin{lem} \label{lpconv}
Assume $p>4$. Let $\{u_n\} \subset \mathcal{M}(c)$ be a Palais-Smale sequence for $E$ restricted on $S(c)$ at the level $\gamma(c)$ such that $u_n\rightharpoonup u_c \neq 0$ in $H^1(\R^2)$ as $n \to \infty$. If
\begin{align} \label{rc}
\gamma(c) \leq \gamma(\bar{c}) \, \, \, \mbox{for any} \, \, \, \bar{c} \in (0, c],
\end{align}
then $\|u_n -u_c\|_p=o_n(1)$ and $E(u_c)=\gamma(c)$.
\end{lem}
\begin{proof}
In view of Lemma \ref{pre}, we know that there exists a constant $\alpha_c \in \R$ such that $u_c$ satisfies the equation
$$
- \left(D_1D_1+D_2D_2\right)u_c + A_0(u) u_c + \alpha_c u_c=\lambda |u_c|^{p-2} u_c,
$$
and $Q(u_c)=0$ by Lemma \ref{Ph}. Since $u_n \rightharpoonup u_c$ in $H^1(\R^2)$ as $n \to \infty$, then
\begin{align*}
\int_{\R^2}|\nabla u_n-\nabla u_c|^2 \, dx + \int_{\R^2}|\nabla u_c|^2 \, dx &= \int_{\R^2}|\nabla u_n|^2 \, dx +o_n(1), \\
\int_{\R^2} |u_n- u_c|^p \, dx + \int_{\R^2} |u_c|^p \, dx &= \int_{\R^2}|u_n|^p \, dx +o_n(1).
\end{align*}
It then follows from Lemmas \ref{BL}-\ref{BL1} that
\begin{align} \label{ju}
E(u_n-u_c) +E(u_c)=E(u_n) +o_n(1)=\gamma (c)+ o_n(1).
\end{align}
We now assume that $\bar{c}:=\|u_c\|_2^2 \leq c$. Since $Q(u_c)=0$, then $u_c \in \mathcal{M}(\bar{c})$, this infers that $\gamma(\bar{c}) \leq E(u_c)$. By using the assumption \eqref{rc}, it then results from \eqref{ju} that $E(u_c-u_c) \leq o_n(1)$. On the other hand, observe that
\begin{align} \label{jq}
\begin{split}
E(u_n-u_c)&=E(u_n-u_c)- \frac 12 Q(u_n-u_c) \\
&=\frac{\lambda(p-4)}{2p}\int_{\R^2} |u_n-u_c|^p \, dx \geq 0.
\end{split}
\end{align}
Consequently, $E(u_n-u_c)=o_n(1)$. Thus, from \eqref{jq}, we have that $\|u_n-u_c\|_p=o_n(1)$. Finally, by applying \eqref{ju}, we get that $E(u_c)=\gamma(c)$. Therefore, we finish the proof.
\end{proof}

\begin{lem} \label{nonincreasing}
Assume $p>4$, then the function $c \mapsto \gamma(c)$ is nonincreasing on $(0, +\infty)$.
\end{lem}
\begin{proof}
To prove this, it is equivalent to show that $\gamma(c_2) \leq \gamma(c_1)$ if $0 <c_1 <c_2$. From the definition of $\gamma(c)$ and Lemma \ref{unique}, we know that, for any $\eps>0$, there exists $u_1 \in \mathcal{M}(c_1)$ such that
$$
E(u_1) \leq \gamma(c_1) + \frac{\eps}{2}, \quad \max_{t >0} E((u_1)_t)=E(u_1).
$$
By the density of $C^{\infty}_0(\R^2)$ in $H^1(\R^2)$, then there is $u_1^{\delta} \in C^{\infty}_0(\R^2)$ with $\text{supp}\, u_1^{\delta} \subset B_{\frac{1}{\delta}}(0)$ such that $\|u_1 -u_1^{\delta}\|=o(\delta)$. Hence
\begin{align} \label{nonincr}
\begin{split}
\int_{\R^2} |D_1 u_1^{\delta}|^2 + |D_2 u_1^{\delta}|^2\,dx&=\int_{\R^2} |D_1 u_1|^2 + |D_2 u_1|^2\,dx+o(\delta),\\
\int_{\R^2}|u_1^{\delta}|^p \, dx &=  \int_{\R^2}|u_1|^p \, dx +o(\delta).
\end{split}
\end{align}
Let $v^{\delta} \in C^{\infty}_0(\R^2)$ be such that $\text{supp}\,v^{\delta} \subset B_{1+\frac{2}{\delta}}(0)\backslash B_{\frac{2}{\delta}}(0)$, and define
$$
v_0^{\delta}:=\left(c_2-\|u_1^{\delta}\|_2^2\right)^{\frac 12} \frac{v^{\delta}}{\|v^{\delta}\|_2}.
$$
For any $t \in (0, 1)$, we now set $w_t^{\delta}:=u_1^{\delta} + (v_0^{\delta})_t$. Note that
$$
\text{dist}(\text{supp} \, u_1^{\delta}, \text{supp} \,(v_0^{\delta})_t) \geq \frac{1}{\delta} \left(\frac{2}{{t}}-1\right)>0,
$$
then $w_{t}^{\delta} \in S(c_2)$. By applying \eqref{nonincr}, we infer that
\begin{align*}
\int_{\R^2} |D_1 w_t^{\delta}|^2 + |D_2 w_t^{\delta}|^2\,dx &\to \int_{\R^2} |D_1 u_1|^2 + |D_2 u_1|^2\,dx, \\
\int_{\R^2}|w_t^{\delta}|^p \, dx &\to \int_{\R^2} |u_1|^p \, dx
\end{align*}
as $\delta, t \to 0$. Thus, for any $\delta, t >0$ small enough,
$$
\gamma(c_2) \leq \max_{s > 0}E((w_t^{\delta})_s) \leq \max_{s >0} E((u_1)_s) + \frac{\eps} {2} \leq \gamma(c_1) +\eps,
$$
and the proof is completed.
\end{proof}

\begin{lem} \label{c}
Assume $p>4$. Let $u_c \in S(c)$ satisfy the equation
\begin{align} \label{equ}
-\left(D_1D_1 + D_2D_2\right) u_c +  A_0 u_c +\alpha_c u_c = \lambda |u_c|^{p-2} u_c,
\end{align}
then there exists a constant $\hat{c} >0$ such that, for any $0<c<\hat{c}$, $\alpha_c >0$.
\end{lem}
\begin{proof}
Since $u_c$ is a solution to the equation \eqref{equ}, then $Q(u_c)=0$ by Lemma \ref{Ph}, namely
\begin{align} \label{a1}
\int_{\R^2}|D_1 u_c|^2 +|D_2 u_c|^2 \, dx = \frac{\lambda (p-2)}{p} \int_{\R^2}|u|^p \, dx.
\end{align}
On the other hand, multiplying \eqref{equ} by $u_c$ and integrating on $\R^2$, then
\begin{align} \label{a2}
\int_{\R^2}|D_1 u_c|^2 +|D_2 u_c|^2 \, dx +\int_{\R^2} A_0(u_c) |u_c|^2 \, dx + \alpha_c \int_{\R^2} |u_c|^2 \, dx =\lambda \int_{\R^2} |u_c|^p \, dx.
\end{align}
By combining \eqref{a1} and \eqref{a2}, thus
\begin{align} \label{alpha}
\begin{split}
\alpha_c \int_{\R^2} |u_c|^2 \, dx &= \frac{2}{p-2} \int_{\R^2}|D_1 u_c|^2 +|D_2 u_c|^2 \, dx -\int_{\R^2} A_0(u_c) |u_c|^2 \, dx  \\
&=\frac{2}{p-2} \int_{\R^2}|\nabla u_c|^2 \, dx + \frac{6- 2p}{p-2} \int_{\R^2} \left(A_1^2(u_c) + A_2^2(u_c)\right) |u_c|^2 \, dx \\
& \quad +\frac{8- 2p}{p-2} \, \mbox{Im} \int_{\R^2} \left(A_1(u_c)\, \partial_1 u_c + A_2(u_c) \,\partial_2 u_c\right) \overline{u}_c \, dx,
\end{split}
\end{align}
where we used the identity \eqref{Aid}. Additionally, according to \eqref{bdd} and the Gagliardo-Nirenberg inequality \eqref{GN}, we get that
\begin{align} \label{csmall}
\int_{\R^2}\left(A_1^2(u_c) + A_2^2(u_c)\right) |u_c|^2 \, dx  \leq C c^{2} \int_{\R^2} |\nabla u_c|^2 \, dx.
\end{align}
Applying \eqref{csmall} and  \eqref{bdd1} with $\eps>0$ small enough, we then infer from \eqref{alpha} that $\alpha_c >0$ for $c>0$ small enough.
\end{proof}

Based upon the previous results, we now are ready to prove the existence of ground state to \eqref{system1}-\eqref{mass}.

\begin{proof}[Proof of Theorem \ref{existence}]
By Lemma \ref{ps}, there exists a Palais-Smale sequence $\{u_n\} \subset \mathcal{M}(c)$ for $E$ restricted on $S(c)$ at the level $\gamma(c)$. From Lemma \ref{pre}, there are a constant $\alpha_c \in \R$ and a nontrivial $u_c \in H^1(\R^2)$ as the weak limit of the sequence $\{u_n\}$ in $H^1(\R^2)$ satisfying the equation \eqref{equ}. Moreover, it follows from Lemma \ref{c} that $\alpha_c>0$ for any $0<c<\hat{c}$. As a consequence of Lemmas \ref{lpconv}-\ref{nonincreasing}, we obtain that $\|u_n-u_c\|_p=o_n(1)$ and $E(u_c)=\gamma(c)$. Thus we conclude from Lemma \ref{pre} that $\|u_n-u_c\| =o_n(1)$, which in turn suggests that $u_c \in S(c)$ is a ground state to \eqref{system1}-\eqref{mass}, and the proof is completed.
\end{proof}

\begin{prop} \label{nonnegative}
Assume $p>4$. If $u\in S(c)$ with $E(u) = \gamma(c)$ satisfies the equation
\begin{align} \label{uequa}
-\left(D_1D_1 + D_2D_2\right) u +  A_0 u +\alpha u = \lambda |u|^{p-2} u,
\end{align}
then $\alpha \geq 0$. In addition, if $\alpha >0$, then the function $c \mapsto \gamma(c)$ is strictly decreasing on a right neighborhood of $c$.
\end{prop}
\begin{proof}
To begin with, we claim that if $\alpha>0$ and $\alpha<0$, then the function $c  \mapsto \gamma(c)$ is strictly decreasing
and strictly increasing on a right neighborhood of $c$, respectively. To prove this, for any $t,\lambda > 0$, let us introduce $u_{t,\tau}(x):=\tau tu(tx)$, and define
$\beta_E (t,\tau):=E(u_{t,\tau})$, $\beta_Q (t,\tau):= Q(u_{t,\tau})$. Since $u$ is a solution to \eqref{uequa}, then $Q(u)=0$ by Lemma \ref{Ph}. Thus it is easy to compute that
$$
\dfrac{\partial \beta_E}{\partial t}(1,1)=0, \quad  \dfrac{\partial^2 \beta_E}{\partial t^2}(1,1)<0, \quad \dfrac{\partial \beta_E}{\partial \tau}(1,1)=-\alpha\, c.
$$
As a consequence, for any $|\delta_t|>0$ small enough and $\delta_\tau>0$,
\begin{align} \label{lessthan}
\beta_E (1+\delta_t, 1+\delta_{\tau})<\beta_E (1,1) \,\,\,\mbox{ if } \alpha >0,
\end{align}
and
\begin{align} \label{lessthan+}
\beta_E (1+\delta_t, 1-\delta_{\tau})<\beta_E (1,1) \,\,\,\mbox{ if } \alpha <0.
\end{align}
Observe that
$$
\beta_{Q}(1, 1)=0, \quad \dfrac{\partial \beta_Q}{\partial t}(1,1)\neq 0,
$$
it then follows from the implicit function theorem that there exist a constant $\varepsilon>0$ and a continuous function $g :[1-\varepsilon ,1+\varepsilon] \mapsto \R$ satisfying $g(1)=1$ such that $\beta_Q (g(\tau), \tau)=0$ for any $\tau \in [1-\varepsilon, 1+\varepsilon]$.
Therefore, we derive from \eqref{lessthan} and \eqref{lessthan+} that
$$
\gamma((1+\varepsilon)c)=\inf_{u\in \mathcal{M}((1+\varepsilon )c)}E(u)\leq E (u_{g(1+\varepsilon), 1+\varepsilon , })< E(u)=\gamma(c) \,\,\,\mbox{ if } \alpha >0,
$$
and
$$
\gamma((1-\varepsilon)c)=\inf_{u\in \mathcal{M}((1-\varepsilon )c)}E(u)\leq E (u_{g(1-\varepsilon), 1-\varepsilon ,})< E(u)=\gamma(c) \,\,\,\mbox{ if } \alpha <0,
$$
respectively, thus the claim follows. With the help of the claim and Lemma \ref{nonincreasing}, we then deduce that $\alpha \geq 0$, and the proof is completed.
\end{proof}

We now are in a position to discuss the existence of infinitely many radially symmetric solutions to \eqref{system1}-\eqref{mass}. To do this, let us first define a transformation $\sigma: H^1_{rad}(\R^2) \to H^1_{rad}(\R^2)$ by $\sigma(u)=-u$, and we say that a set $A \subset H^1_{rad}(\R^2)$ is $\sigma$-invariant if $\sigma (A)=A$. We next introduce the definition of genus of a set, due to M.A. Krasnosel'skii.

\begin{defi}\label{genus}
For any closed $\sigma$-invariant set $A \subset H^1_{rad}(\R^2)$, the genus of $A$ is defined by
$$
\mathcal{\gamma}(A):= \min \{n \in \N^+ : \exists \ \psi : A \rightarrow \R^n\backslash \{0\}, \psi \,\, \text{is continuous and odd}\}.
$$
When there is no $\psi$ as described above, we set $\mathcal{\gamma}(A):= \infty.$

\end{defi}
Let $\mathcal{A}$ be a family of compact and $\sigma$-invariant sets contained in $\mathcal{M}_{rad}(c)$. For any $k \in \N^+$, set
$$
\mathcal{A}_{k}:=\{A \in \mathcal{A}: \mathcal{\gamma}(A) \geq k\},
$$
and
\begin{align*}
\beta_k:=\inf_{A \in \mathcal{A}_{k}}\sup_{u \in A}E(u).
\end{align*}

First of all, we justify that, for any $k \in \N^+ $, $\beta_k$ is well-defined.

\begin{lem} \label{nonempty}
Assume $p>4$, then, for any $k \in \N^+ $, $\mathcal{A}_k \neq \emptyset.$
\end{lem}
\begin{proof}
For any fixed $k \in \N^+$, let $V \subset H^1_{rad}(\R^2)$ be such that $\dim(V)=k$, and set $SV(c):=V \cap S(c)$. By the basic property of genus, see \cite[Theorem 10.5]{AmMa}, we then have that $\gamma (SV(c))= \dim V =k$. From Lemma \ref{unique}, for any $u \in SV(c)$, there exists a unique $t_u >0$ such that $ u_{t_u} \in \mathcal{M}_{rad}(c)$. Thus we are able to define a map $\eta: SV(c)\rightarrow \mathcal{M}_{rad}(c)$ by $\eta(u)=u_{t_u}$, and it is easy to see that $\eta$ is continuous and odd. By using \cite[Lemma 10.4]{AmMa}, hence we obtain that $\gamma(\eta(SV(c))) \geq \gamma(SV(c))=k$, and this indicates that $\mathcal{A}_k \neq \emptyset$.
\end{proof}

The following Definition \ref{defi1} and Lemma \ref{ps1} can be regarded as a counterpart of Definition \ref{homotopy} and Lemma \ref{ps}, respectively.

\begin{defi} \label{defi1}
Let $B$ be a closed $\sigma$-invariant subset of $Y \subset H^1_{rad}(\R^2)$. We say that a class $\mathcal{F}$ of compact subsets of $Y$ is a $\sigma$-homotopy stable family with closed boundary $B$ if
\begin{enumerate}
\item [(i)] every set in $\mathcal{F}$ is $\sigma$-invariant;
\item [(ii)]every set in $\mathcal{F}$ contains $B$;
\item [(iii)] for any $A\in \mathcal{F}$ and for any $\eta \in C([0,1]\times Y, Y)$ satisfying $\eta (t,u)=\eta (t,\sigma (u))$ for all $t\in [0,1]$, and $\eta (t,x)=x$ for all $(t,x)\in (\{0\}\times Y)\cup ([0,1]\times B)$, then $\eta (\{1\} \times A)\in
\mathcal{F}$.
\end{enumerate}
\end{defi}

\begin{lem} \label{ps1}
Let $\mathcal{F}$ be a $\sigma$-homotopy stable family of compact subsets of $\mathcal{M}_{rad}(c)$ with a close boundary $B$. Let
$$
c_{\mathcal{F}}:= \inf_{A\in \mathcal{F}}\max_{u\in A}E(u).
$$
Suppose that $B$ is contained in a connected component of $\mathcal{M}_{rad}(c)$ and
$$
\max \{\sup E(B),0\}<c_{\mathcal{F}}<\infty.
$$
Then there exists a Palais-Smale sequence $\{u_n\} \subset \mathcal{M}_{rad}(c)$ for $E$ restricted to $S_{rad}(c)$ at the level $c_{\mathcal{F}}$.
\end{lem}
\begin{proof}
By applying \cite[Theorem 7.2]{Gh} instead of \cite[Theorem 3.2]{Gh}, the proof can be completed by an almost identical manner as the one of Lemma \ref{ps}, hence we omit the proof.
\end{proof}

Benefiting from the arguments above, we shall prove Theorem \ref{infsolutions}.

\begin{proof}[Proof of Theorem \ref{infsolutions}]
For any fixed $k \in \N^+$, from Lemma \ref{nonempty}, one has that $\beta_k <\infty$. In view of Lemma \ref{ps1}, there is a Palais-Smale sequence $\{u_n\} \subset \mathcal{M}_{rad}(c)$ for $E$ restricted on $S_{rad}(c)$ at the level $\beta_k$. From Lemma \ref{pre}, there are a constant $\alpha_c \in \R$ and a nontrivial $u_c \in H^1(\R^2)$ as the weak limit of $\{u_n\}$ in $H^1_{rad}(\R^2)$ satisfying the equation \eqref{equ}. In addition, $\alpha_c>0$ for $0<c<\hat{c}$, see Lemma \ref{c}. Since the embedding $H^1_{rad}(\R^2)\hookrightarrow L^p(\R^2)$ is compact, we then apply Lemma \ref{pre} to conclude that $\|u_n-u_c\|=o_n(1)$ and $u_c \in S(c)$ is a solution to \eqref{system1}-\eqref{mass} at the level $\beta_k$. Thus we know that \eqref{system1}-\eqref{mass} has infinitely many radial symmetric solutions.
\end{proof}

\section{The dynamic behaviors} \label{dynamical}

We devote this section to the study of the dynamical behaviors of ground states to \eqref{system1}-\eqref{mass}. First of all, let us show the local well-posedness of solutions to the Cauchy problem of the system \eqref{sys1} in $H^1(\R^2)$.

\begin{lem} \label{wellposed}
Assume $p>2$, then, for any $\varphi_0 \in H^1(\R^2)$, there exist a constant $T>0$ and a unique solution $\varphi \in C([0, T); H^1(\R^2))$ to the Cauchy problem of the system \eqref{sys1} with $\varphi(0)=\varphi_0$, satisfying the conservations of mass and energy, namely, for any $t \in [0, T)$,
$$
\|\varphi(t)\|_2=\|\varphi_0\|_2,\quad E(\varphi(t)) =E(\varphi_0),
$$
and the solution map $\varphi_0 \mapsto \varphi$ is continuous from $H^1(\R^2)$ to $C([0, T); H^1(\R^2))$. In addition, either $T= \infty$ or $\lim_{t \to T^-} \left(\|D_1 \varphi(t)\|_2 + \|D_2 \varphi(t)\|_2\right) =\infty$.
\end{lem}
\begin{proof}
The local well-posedness of solutions to the problem can be achieved by following the ideas developed in \cite{huh2, Lim}, where the problem was considered in the mass critical case. We next deduce that the conservation laws hold true. Multiplying the first equation in the system \eqref{sys1} by $\overline{\varphi}$, integrating on $\R^2$ and taking the real part, we then get the conservation of mass. Multiplying the first equation in the system \eqref{sys1} by $\partial_t \overline{\varphi}$, integrating on $\R^2$ and taking the imaginary part, we then obtain the conservation of energy. Since the constant $T>0$ is the maximal existence time of the solution, then the blowup alternative necessarily follows, and the proof is completed.
\end{proof}

Based upon the local well-posedness of solutions to the Cauchy problem of the system \eqref{sys1} in $H^1(\R^2)$, we are able to establish the global well-posedness of the problem.

\begin{proof} [Proof of Theorem \ref{globalexis}]
By Theorem \ref{wellposed}, we may suppose that $\varphi \in C([0, T);H^1(\R^2))$ is the solution to the Cauchy problem of the system \eqref{sys1} with $\varphi(0)=\varphi_0$. If $2<p<4$, or $p=4$ and $\|\varphi_0\|_2$ is small enough, by using the conservation laws and \eqref{MGN}, we then deduce that $\|D_1 \varphi(t)\|_2 + \|D_2 \varphi(t)\|_2 \leq C$ for any $t\in [0, T)$, where the constant $
C>0$ is independent of $t$. Thus the blowup alternative in Lemma \ref{wellposed} shows that $\varphi$ exists globally in time. Next we consider the case that $p>4$ and $\varphi_0 \in \mathcal{O}_c$. In this case, we argue by contradiction that $T < \infty$. By Lemma \ref{wellposed}, then
\begin{align} \label{blowup}
\lim_{t \to T^-} \|D_1 \varphi(t)\|_2 + \|D_2 \varphi(t)\|_2 =\infty.
\end{align}
Note that
$$
E(\varphi(t))-\frac{1}{p-2} Q(\varphi(t))= \frac{p-4}{2(p-2)} \left(\|D_1 \varphi(t)\|_2^2 + \|D_2 \varphi(t)\|_2^2\right),
$$
and the conservation of energy $E(\varphi(t))=E(\varphi_0)$ for any $t \in [0, T)$, by virtue of \eqref{blowup}, we then derive that $\lim_{t \to T^-} Q(\varphi(t))=- \infty$. Recall that $Q(\varphi_0) >0$, thus there exists a constant $t_0 \in (0, T)$ such that $Q(\varphi(t_0)) =0$. This suggests that $\varphi(t_0) \in \mathcal{M}(c)$, hence $E(\varphi(t_0) \geq \gamma(c)$. However, $E(\varphi(t_0))=E(\varphi_0) < \gamma(c)$, we then reach a contradiction. Therefore, the solution $\varphi$ exists globally in time, and we have completed the proof.
\end{proof}

In view of the global well-posedness of solutions to the Cauchy problem of the system \eqref{sys1} in the mass subcritical case, we are now able to prove the orbital stability of the set of minimizers to \eqref{gmin}.

\begin{proof} [Proof of Theorem \ref{stable}]
We argue by contradiction that $G(c)$ were orbitally unstable for some $0<c<c_0$, then there would exist a constant $\eps_0>0$, a sequence $\{t_n\} \subset \R^+$ and a sequence $\{\varphi_{0, n}\} \subset H^1(\R^2)$ with
\begin{align} \label{distmass}
\inf_{u\in G (c)} \|\varphi_{0, n} - u\|\leq \frac 1n,
\end{align}
so that
\begin{align} \label{contra}
\inf_{u\in G(c)}\|\varphi_n(t_n)-u\| \geq \eps_0,
\end{align}
where $\varphi_n \in C([0, \infty); H^1(\R^2))$ is the solution to the Cauchy problem of the system \eqref{sys1} with $ \varphi_n(0)=\varphi_{0, n}$.
We now define
$$
\tilde{\varphi}_n:=\frac{\varphi_n(t_n)}{\|\varphi_n(t_n)\|_2} c^{\frac 12},
$$
and $\{\tilde{\varphi}_n\} \subset S(c)$. By the conservation laws, we have that $\|\varphi_n(t_n)\|_2=\|\varphi_{0, n}\|_2$ and $E(\varphi_n(t_n))=E(\varphi_{0, n})$. Thus we apply \eqref{distmass} to derive that $E(\tilde{\varphi}_n)=E(\varphi_n(t_n))+o_n(1)=E(\varphi_{0, n})+o_n(1)=m(c)+o_n(1)$. Consequently, $\{\tilde{\varphi}_n\} \subset S(c)$ is a minimizing sequence to \eqref{gmin}. As a result of Theorem \ref{compactness}, we know that $\{\tilde{\varphi}_n\}$ is compact in $H^1(\R^2)$ up to translations, so is $\{\varphi_n(t_n)\}$, which then contradicts \eqref{contra}, and the proof is completed.
\end{proof}

We are now in a position to discuss the instability of ground states to \eqref{system1}-\eqref{mass} in the mass supercritical case. To do this, we first establish the following crucial result, whose proof will be postponed until in Appendix.

\begin{lem} \label{virial}
Assume $p>2$. Let $\xi \in C^4(\R^2, \R)$ be a radially symmetric function and let $\varphi \in C([0, T); H^1(\R^2))$ be a solution to the Cauchy problem of the system \eqref{sys1}. Define
\begin{align} \label{vxi}
V_{\xi}[\varphi(t)]:=\mbox{Im} \int_{\R^2} \bar{\varphi} \left(D_1 \varphi \, \partial_1 \xi + D_2 \varphi \, \partial_2 \xi \right) \, dx,
\end{align}
then, for any $t \in [0, T)$,
\begin{align*}
\frac{d}{dt}V_{\xi}[\varphi(t)] &= 2\int_{\R^2}\left(|D_1 \varphi|^2\, \partial_1^2 \xi
+2 \mbox{Re}\,\overline{D_1 \varphi} \, D_2 \varphi \, \partial_2\partial_1 \xi + |D_2 \varphi|^2\, \partial_2^2 \xi \right)\,dx \\
& \quad -\frac{\lambda(p-2)}{p}\int_{\R^2}|\varphi|^p \left(\partial_1^2 \xi + \partial_2^2 \xi \right)\, dx \\
& \quad -\frac 12 \int_{\R^2} |\varphi|^2 \left(\partial_1^4 \xi +  2 \partial_1^2 \partial_2^2 \xi + \partial_2^4 \xi \right)\, dx.
\end{align*}
\end{lem}

As an immediate consequence of Lemma \ref{virial}, we have the following.

\begin{lem} \label{virial1}
Assume $p>2$. Let $\varphi \in C([0, T); H^1(\R^2))$ be a solution to the Cauchy problem of the system \eqref{sys1} satisfying $\varphi(t) \in \Sigma$ for any $t \in [0, T)$. Define
$$
I(t):=\int_{\R^2}|x|^2|\varphi|^2 \, dx,
$$
then, for any $t\in [0, T)$,
\begin{align*}
\frac{d}{dt}I(t)&=4 \,\mbox{Im} \int_{\R^2} \overline{\varphi} \left(D_1 \varphi \, x_1 + D_2 \varphi \, x_2 \right) \, dx,\\
\frac{d^2}{dt^2}I(t)&=8Q(\varphi),
\end{align*}
where $Q$ is defined by \eqref{defq}.
\end{lem}
\begin{proof}
Since $A_0$ is real-valued and $\varphi$ is a solution to the Cauchy problem of the system \eqref{sys1}, then
\begin{align*}
\frac{d}{dt}I(t)&=2 \, \mbox{Re} \int_{\R^2} |x|^2 \partial_t \varphi \, \overline{\varphi} \, dx
=2\, \mbox{Re} \int_{\R^2} |x|^2 D_t \varphi \, \overline{\varphi} \, dx \\
&=2 \,\mbox{Re} \int_{\R^2} i\, |x|^2 \left( \lambda |\varphi|^{p-2} \varphi + D_1D_1 \varphi + D_2D_2 \varphi\right) \overline{\varphi} \, dx \\
&=-2 \,\mbox{Im} \int_{\R^2} |x|^2 \left( \lambda |\varphi|^{p-2} \varphi + D_1D_1 \varphi + D_2D_2 \varphi\right) \overline{\varphi} \, dx \\
&=-2 \,\mbox{Im} \int_{\R^2} |x|^2 \left( D_1D_1 \varphi + D_2D_2 \varphi\right) \overline{\varphi} \, dx \\
&=-2 \,\mbox{Im} \int_{\R^2} \left( \left(\partial_1 + i \, A_1 \right) D_1 \varphi + \left(\partial_2 + i \, A_2 \right)D_2 \varphi\right) |x|^2 \overline{\varphi} \, dx \\
&= 2 \,\mbox{Im} \int_{\R^2} D_1 \varphi \,\overline{D_1 \left( |x|^2 \varphi \right)}+ D_2 \varphi \,\overline{D_2 \left( |x|^2 \varphi \right)} \, dx \\
&= 4 \,\mbox{Im} \int_{\R^2} \overline{\varphi} \left(D_1 \varphi \, x_1 + D_2 \varphi \, x_2 \right) \, dx.
\end{align*}
We next apply Lemma \ref{virial} by taking $\xi=2|x|^2$ to conclude that
$$
\frac{d^2}{dt^2}I(t)=8Q(\varphi).
$$
Thus we have completed the proof.
\end{proof}

\begin{rem}
By applying the method proposed in \cite{BBS}, one can derive that if $\varphi_0 \in \Sigma$, then the solution $\varphi \in C([0, T); H^1(\R^2)$ to the Cauchy problem of the system \eqref{sys1} with $\varphi(0)=\varphi_0$ satisfies that $\varphi(t) \in \Sigma$ for any $t \in [0, T)$.
\end{rem}

We now aim to discuss the evolution of a localized virial quantity with respect to a radially symmetric solution to the Cauchy problem of the system \eqref{sys1}. Let $\chi : \R^2 \to \R$ be a radially symmetric function with regularity property $\nabla^k \chi \in L^{\infty}(\R^2)$ for $1 \leq k \leq 4$ such that
\begin{equation*}
\chi(r):=\left\{
\begin{array}{lr}
\frac{r^2}{2}   \,\,\,&\text{for} \,\, r \leq 1,\\
\text{const.}  \,\,\,&\text{for} \, \, r \geq 10,
\end{array}
\quad \mbox{and} \,\,\, \varphi''(r) \leq 1 \,\, \mbox{for any} \,\,r \geq 0.
\right.
\end{equation*}
For $R>0$ given, we define a radially symmetric function $\chi_R: \R^2 \to \R$ by
\begin{align} \label{defchi}
\chi_R(r):=R^2 \chi\left(\frac{r}{R}\right).
\end{align}
It is simple to check that
\begin{align} \label{vpro}
1- \chi_R''(r) \geq 0, \,\,\, 1 -\frac{\chi_R'(r)}{r} \geq 0, \,\,\, 2- \Delta \chi_R(r) \geq 0 \,\, \, \mbox{for any} \,\, \, r \geq 0.
\end{align}
For a radially symmetric solution $\varphi \in C([0, T); H^1(\R^2))$ to the Cauchy problem of the system \eqref{sys1}, we now introduce a localized virial quantity
\begin{align} \label{vloc}
V_{\chi_R}[\varphi(t)]:=\mbox{Im} \int_{\R^2} \overline{\varphi} \left(D_1 \varphi \, \partial_1 \chi_R + D_2 \varphi \, \partial_2 \chi_R \right) \, dx.
\end{align}
As an easy consequence of the H\"older inequality, then
\begin{align} \label{vloc1}
\left|V_{\chi_R}[\varphi(t)]\right| \leq C \|\varphi(t)\|_2 \left(\|D_1 \varphi(t)\|_2+\|D_2\varphi(t)\|_2\right),
\end{align}
which reveals that $V_{\chi_R}[\varphi(t)]$ is well-defined for any $t \in [0, T)$. \medskip

In this direction, we have the following crucial lemma.

\begin{lem}\label{vlem}
Assume $p>2$. Let $\varphi \in C([0, T); H^1(\R^2))$ be a radially symmetric solution to the Cauchy problem of the system \eqref{sys1}, then, for any $t \in [0, T)$,
\begin{align*}
\frac{d}{dt}V_{\chi_R}[\varphi(t)]&\leq 2Q(\varphi) + C R^{-\frac{p-2}{2}} \left(\|D_1 \varphi\|^{\frac{p-2}{2}}_2 + \|D_2 \varphi\|^{\frac{p-2}{2}}_2\right) + C R^{-2} \\
& = 2(p-2)E(\varphi) -\left(p-4\right) \left(\|D_1 \varphi\|^2_2+\|D_2 \varphi\|_2^2\right) \\
& \quad + C R^{-\frac{p-2}{2}} \left(\|D_1 \varphi\|^{\frac{p-2}{2}}_2 + \|D_2 \varphi\|^{\frac{p-2}{2}}_2\right) + C R^{-2}.
\end{align*}
\end{lem}
\begin{proof}
Applying Lemma \ref{virial} by taking $\xi=\chi_R$, we have that
\begin{align} \nonumber
\frac{d}{dt}V_{\chi_R}[\varphi(t)] &= 2\int_{\R^2}|D_1 \varphi|^2\, \partial_1^2 \chi_R
+2 \mbox{Re}\,\overline{D_1 \varphi} \, D_2 \varphi \, \partial_2\partial_1 \chi_R + |D_2 \varphi|^2\, \partial_2^2 \chi_R \,dx \\ \label{v0}
& \quad -\frac{\lambda(p-2)}{p}\int_{\R^2}|\varphi|^p \left(\partial_1^2 \chi_R + \partial_2^2 \chi_R \right)\, dx \\ \nonumber
& \quad -\frac 12 \int_{\R^2} |\varphi|^2 \left(\partial_1^4 \chi_R +  2 \partial_1^2 \partial_2^2 \chi_R + \partial_2^4 \chi_R \right)\, dx.
\end{align}
We now estimate three terms in the right hand side of \eqref{v0}. We begin with treating the first term. Recall that $\chi_R$ is a radially symmetric function, then it is straightforward to check that
$$
\partial_{j} \partial_l \chi_R=\left(\delta_{jl}-\frac{x_jx_l}{r^2}\right)\frac{\chi'_R}{r} + \frac{x_jx_l}{r^2} \chi''_R,
$$
where $\delta_{jl}=1$ if $j=l$, $\delta_{jl}=0$ if $j\neq l$, and $j,l=1, 2.$ Thus we can deduce that
\begin{align} \label{destimate}
\begin{split}
&|D_1 \varphi|^2\, \partial_1^2 \chi_R
+2 \mbox{Re}\,\overline{D_1 \varphi} \, D_2 \varphi \, \partial_2\partial_1 \chi_R + |D_2 \varphi|^2\, \partial_2^2 \chi_R\\
&=\chi''_R |D^r \varphi|^2 + \frac{\chi'_R}{r}|D^{\tau} \varphi|^2,
\end{split}
\end{align}
where we define $D \varphi:=\left(D_1 \varphi, D_2 \varphi \right)$, and $D^{\tau}$ is the projection of the operator $D$ on the tangent plane to the sphere such that
$$
|D^r \varphi|^2+|D^{\tau} \varphi|^2=|D \varphi|^2=|D_1 \varphi|^2+|D_2 \varphi|^2,
$$
and
$$
D^r \varphi:=\left(D \varphi \cdot \frac{x}{|x|}\right)\frac{x}{|x|}, \quad D^{\tau}\varphi \cdot D^r \varphi=0.
$$
In light of \eqref{vpro} and \eqref{destimate}, we then derive that
\begin{align} \label{vd}
\begin{split}
& \int_{\R^2}|D_1 \varphi|^2\, \partial_1^2 \chi_R
+2 \mbox{Re}\,\overline{D_1 \varphi} \, D_2 \varphi \, \partial_2\partial_1 \chi_R + |D_2 \varphi|^2\, \partial_2^2 \chi_R \, dx\\
&=\int_{\R^2} |D_1 \varphi|^2+|D_2 \varphi|^2 + \left(\chi''_R-1\right) |D^r \varphi|^2 + \left(\frac{\chi'_R}{r}-1\right)|D^{\tau} \varphi|^2 \, dx \\
&\leq \int_{\R^2} |D_1 \varphi|^2+|D_2 \varphi|^2 \, dx.
\end{split}
\end{align}
We next deal with the second term. Since $\varphi$ is radially symmetric, by using the Strauss inequality, then
\begin{align}\label{radineq}
|\varphi| \leq C |x|^{-\frac 12} \| \varphi \|_2^{\frac 12} \|\nabla |\varphi| \|_2^{\frac 12} \quad \mbox{for any} \,\,\,|x| \geq R.
\end{align}
According to Lemma \ref{diaineq}, we know that
\begin{align}\label{magineq}
|\nabla |\varphi|| \leq |\left(\partial_1 + i A_1, \partial_2 + iA_2\right) \varphi| = |D \varphi| \leq  |D_1 \varphi| + |D_2 \varphi|.
\end{align}
Combining \eqref{radineq} and \eqref{magineq}, we then obtain that
\begin{align*}
\int_{\{x: |x| \geq R\}} |\varphi|^p \, dx &\leq  C R^{-\frac{p-2}{2}} \|\varphi\|_2^{\frac{p+2}{2}}\|\nabla |\varphi|\|^{\frac{p-2}{2}}_2 \\
&\leq  CR^{-\frac{p-2}{2}} \|\varphi\|_2^{\frac{p+2}{2}} \left(\|D_1 \varphi\|^{\frac{p-2}{2}}_2 + \|D_2 \varphi\|^{\frac{p-2}{2}}_2\right).
\end{align*}
This asserts that
\begin{align} \label{error}
\begin{split}
\int_{\R^2} |\varphi|^p \Delta \chi_R \, dx &= 2 \int_{\R^2} |\varphi|^p \, dx + \int_{\R^2} |\varphi|^p\left(\Delta \chi_R-2\right) \, dx \\
&=2 \int_{\R^2} |\varphi|^p \, dx + \int_{\{x: |x| \geq R\}}  |\varphi|^p\left(\Delta \chi_R-2\right) \, dx \\
& \leq 2 \int_{\R^2} |\varphi|^p \, dx + C \int_{\{x: |x| \geq R\}}  |\varphi|^p \, dx \\
&\leq  2 \int_{\R^2} |\varphi|^p \, dx + CR^{-\frac{p-2}{2}} \left(\|D_1 \varphi\|^{\frac{p-2}{2}}_2 + \|D_2 \varphi\|^{\frac{p-2}{2}}_2\right),
\end{split}
\end{align}
where the second identity is insured by the property that $\Delta \chi_R - 2=0$ for any $|x| \leq R$.
We now estimate the last term. By the definition of $\chi_R$, see \eqref{defchi}, there holds that
\begin{align*}
\|\partial_1^4 \chi_R +  2 \partial_1^2 \partial_2^2 \chi_R + \partial_2^4 \chi_R\|_{\infty} \leq C R^{-2}.
\end{align*}
Accordingly, from the conservation of mass, we obtain the following estimate to the last term,
\begin{align}\label{last}
\int_{\R^2} |\varphi|^2 \left(\partial_1^4 \chi_R +  2 \partial_1^2 \partial_2^2 \chi_R + \partial_2^4 \chi_R \right)\, dx \leq C R^{-2}.
\end{align}
By using \eqref{vd}, \eqref{error} and \eqref{last}, it then follows from \eqref{v0} that
\begin{align*}
\frac{d}{dt}V_{\chi_R}[\varphi(t)]& \leq  2 \left(\|D_1 \varphi \|_2^2 + \|D_2 \varphi \|_2^2\right) -\frac{2\lambda (p-2)}{p} \int_{\R^2}|\varphi|^p\,dx \\
& \quad + C R^{-\frac{p-2}{2}} \left(\|D_1 \varphi\|^{\frac{p-2}{2}}_2 + \|D_2 \varphi\|^{\frac{p-2}{2}}_2\right) + C R^{-2} \\
&=2Q(\varphi) + C R^{-\frac{p-2}{2}} \left(\|D_1 \varphi\|^{\frac{p-2}{2}}_2 + \|D_2 \varphi\|^{\frac{p-2}{2}}_2\right) + C R^{-2} \\
&= 2(p-2)E(\varphi) -\left(p-4\right) \left(\| D_1 \varphi\|^2_2+\|D_2 \varphi\|_2^2\right) \\
& \quad + C R^{-\frac{p-2}{2}} \left(\|D_1 \varphi\|^{\frac{p-2}{2}}_2 + \|D_2 \varphi\|^{\frac{p-2}{2}}_2\right) + C R^{-2},
\end{align*}
and the proof is completed.
\end{proof}

Relying on the previous Lemmas \ref{virial1}-\ref{vlem}, we are now able to prove the instability of ground states to \eqref{system1}-\eqref{mass} in the mass supercritical case.

\begin{proof}[Proof of Theorem \ref{unstable}]
Let us first define $\varphi_{0, \tau}:=\tau u_c(\tau x)$ and
$$
\Theta:=\left\{v \in H^1(\R^2) : E(v) <E(u_c), \,\,\, \|v\|_2=\|u_c\|_2, \,\,\,Q(v) <0\right\}.
$$
By Lemma \ref{unique}, one can derive that $\varphi_{0, \tau} \in \Theta$ for any $\tau$ close enough to $1$ from above. In addition, there holds that $\|\varphi_{0, \tau}-u_c\| \to 0$ as $\tau \to 1^+$. For any $\eps>0$. we now fix a constant $\tau >1$ close enough to $1$ such that $\|\varphi_{0, \tau}-u_c\| \leq \eps$, and let $\varphi \in C([0, T); H^1(\R^2))$ be the solution to the Cauchy problem of the system \eqref{sys1} with $\varphi(0)= \varphi_{0, \tau}$, then $\varphi(t) \in \Theta$ for any $t \in [0, T)$. Indeed, recall that $\varphi_{0, \tau} \in \Theta$, by the conservation laws, hence, for any $t \in [0, T)$,
\begin{align} \label{laws}
\|\varphi(t)\|_2=\|u_c\|_2, \quad E(\varphi(t)) < E(u_c).
\end{align}
It remains to show that $Q(\varphi(t))<0$ for any $t \in [0, T)$. If there were a constant $t_1\in (0, T)$ such that $Q(\varphi(t_1)) \geq 0$, it then implies that there is a constant $t_2 \in (0, t_1]$ such that $Q(\varphi(t_2))=0$, thus $E(\varphi(t_2)) \geq \gamma(c)=E(u_c)$, this contradicts \eqref{laws}. Thus the assertion follows.

For simplicity of notation, we shall write $\varphi:=\varphi(t)$. Since $Q(\varphi)<0$ by the previous discussion, from Lemma \ref{unique}, then there is a constant $0<\tau^* <1$ such that $Q(\varphi_{\tau^*})=0$. Furthermore, the function $\tau \mapsto E(\varphi_{\tau})$ is concave on $[\tau^*, 1]$. Thus
\begin{align} \label{QE}
E(\varphi_{\tau^*}) -E(\varphi) \leq \left( \tau^*-1\right) \frac{\partial E(\varphi_{\tau})}{\partial \tau}{\mid_{\tau=1}} = \left( \tau^*-1\right) Q(\varphi).
\end{align}
Note that $Q(\varphi) <0$, $E(\varphi)=E(\varphi_{0, \tau})$ and $\varphi_{\tau^*} \in \mathcal{M}(c)$, it then yields from \eqref{QE} that
$$
Q(\varphi) \leq \left(1-\tau^*\right) Q(\varphi) \leq E(\varphi)-E(\varphi_{\tau^*})  \leq  E(\varphi_{0,\tau})-E(u_c),
$$
this in turn gives that
\begin{align} \label{vQ}
Q(\varphi) \leq -\beta,
\end{align}
where $\beta:= E(u_c)-E(\varphi_{0, \tau}) >0$.

At this point, in order to prove that $u_c$ is strongly unstable, it suffices to derive that the solution $\varphi$ blows up in finite time. Firstly, we assume that $u_c \in \Sigma$. Then $\varphi_{0, \tau} \in \Sigma$, and we make use of Lemma \ref{virial1} and \eqref{vQ} to obtain that
$$
\frac{d^2}{dt^2}\int_{\R^2}|x|^2|\varphi|^2 \, dx =8Q(\varphi)<-8\beta,
$$
from which we conclude that $\varphi$ has to blow up in finite time. Secondly, we assume that $u_c$ is radially symmetric and $p \leq 6$. In this situation, suppose by contradiction that the solution $\varphi$ exists globally in time. We now claim that there exist a constant $\delta>0$ such that
\begin{align} \label{evo}
\frac{d}{dt}V_{\chi_R}[\varphi(t)] \leq -\delta \left(\|D_1 \varphi(t)\|_2^2 + \|D_2 \varphi(t)\|_2^2 \right)
\,\,\, \mbox{for any} \,\,\, t \in [0, \infty),
\end{align}
and a constant $t_0>0$ such that
\begin{align} \label{evo1}
V_{\chi_R}[\varphi(t)] <0 \quad \mbox{for any} \,\,\, t \geq t_0.
\end{align}
To prove this, we consider the following two distinguishing cases. To begin with, let us note the fact that $E(\varphi_{0,\tau})>0$ for $\tau>1$ close enough to $1$.\\
{\bf Case 1:} Let
$$
T_1:=\left\{t\in [0, \infty): \|D_1 \varphi(t)\|_2^2 + \|D_1 \varphi(t)\|_2^2  \leq \frac{4(p-2)}{(p-4)} E(\varphi_{0, \tau})\right\}.
$$
In view of Lemma \ref{vlem} and \eqref{vQ}, by taking $R>0$ large enough, we have that, for any $t \in T_1$,
$$
\frac{d}{dt}V_{\chi_R}[\varphi(t)] \leq Q(\varphi(t)) \leq -\beta \leq - \delta \left(\|D_1 \varphi(t)\|_2^2 + \|D_1 \varphi(t)\|_2^2 \right)
$$
for some $\delta >0$ small enough. \\
{\bf Case 2:} Let
$$
T_2:=[0, \infty)\backslash T_1=\left\{t\in [0, \infty): \|D_1 \varphi(t)\|_2^2 + \|D_1 \varphi(t)\|_2^2 > \frac{4(p-2)}{(p-4)} E(\varphi_{0, \tau})\right\}.
$$
By the conservation of energy and Lemma \ref{vlem}, we obtain that, for any $t \in T_2$,
\begin{align*}
\frac{d}{dt}V_{\chi_R}[\varphi(t)] &\leq -\frac{p-4}{2} \left(\|D_1 \varphi(t)\|^2_2+\|D_2 \varphi(t)\|_2^2\right) \\
& \quad + C R^{-\frac{p-2}{2}} \left(\|D_1 \varphi(t)\|^{\frac{p-2}{2}}_2 + \|D_2 \varphi(t)\|^{\frac{p-2}{2}}_2\right) + C R^{-2}.
\end{align*}
Since $p \leq 6$, by taking $R>0$ large enough, then, for any $t \in T_2$,
$$
\frac{d}{dt}V_{\chi_R}[\varphi(t)] \leq -\frac{p-4}{4} \left(\|D_1 \varphi(t)\|^2_2+\|D_2 \varphi(t)\|_2^2\right).
$$
From the arguments above, the claim then follows.

We now integrate \eqref{evo} on $[t_0, t]$ and use \eqref{evo1} to conclude that
\begin{align} \label{VV1}
V_{\chi_R}[\varphi(t)] \leq  - \delta \int_{t_0}^t \left(\|D_1 \varphi(t)\|_2^2 + \|D_2 \varphi(t)\|_2^2 \right) \, dx.
\end{align}
In addition, from \eqref{vloc1} and the conservation of mass, we know that
\begin{align} \label{VV2}
|V_{\chi_R}[\varphi(t)]| \leq C \left(\|D_1 \varphi(t)\|_2 + \|D_2 \varphi(t)\|_2\right).
\end{align}
Consequently, \eqref{VV1} and \eqref{VV2} leads to
\begin{align}\label{vineq}
V_{\chi_R}[\varphi(t)] \leq  - \mu \int_{t_0}^t |V_{\chi_R}[\varphi(s)]|^2 \, ds
\end{align}
for some constant $\mu>0$. Setting
$$
z(t):=\int_{t_0}^t |V_{\chi_R}[\varphi(s)]|^2 \, ds,
$$
we then get from \eqref{vineq} that $z'(t) \geq \mu^2 z(t)^2$. By integrating this differential inequality, we have that $V_{\chi_R}[\varphi(t)] \to -\infty$ as $t$ goes to some constant $t^*>0$. This then contradicts our assumption that $\varphi$ exists globally in time. Therefore, we conclude that the solution $\varphi$ blows up in finite time, and the proof is completed.
\end{proof}

\section{Appendix}

This last section is devoted to the proof of Lemma \ref{virial}.

\begin{proof} [Proof of Lemma \ref{virial}]
Noting first the definition of the virial type quantity $V_{\xi}[\varphi(t)]$, see \eqref{vxi}, we deduce that
\begin{align*}
\frac{d}{dt}V_{\xi}[\varphi(t)]&= \mbox{Im} \int_{\R^2} \partial_t \overline{\varphi} \left(D_1 \varphi \, \partial_1 \xi + D_2 \varphi \, \partial_2 \xi \right) + \overline{\varphi} \left(\partial_t D_1 \varphi \, \partial_1 \xi + \partial_t D_2 \varphi \, \partial_2 \xi \right) \, dx \\
&= \mbox{Im} \int_{\R^2} \left(\partial_t \overline{\varphi} \, D_1 \varphi + \overline{\varphi} \, \partial_t D_1 \varphi \right) \partial_1 \xi  + \left(\partial_t \overline{\varphi} \, D_2 \varphi + \overline{\varphi} \, \partial_t D_2 \varphi \right) \partial_2 \xi \, dx.
\end{align*}
Since $A_0$ is real-valued, then
\begin{align}\label{v1}
\partial_t \overline{\varphi} \, D_j \varphi + \overline{\varphi} \, \partial_0 D_j \varphi=\overline{D_t \varphi} \, D_j \varphi + \overline{\varphi} \, D_t D_j \varphi \,\,\, \text{for} \, \, j=1,2.
\end{align}
Moreover, by the definitions of $D_{t}$, $D_j$ and the fact that $\varphi$ satisfies the system \eqref{sys1}, there holds that
\begin{align} \label{v2}
\begin{split}
D_tD_1 \varphi&=D_1D_t \varphi-i \,\mbox{Im} \left(\overline{\varphi} \,D_2 \varphi \right) \varphi,\\
D_tD_2 \varphi&=D_2D_t \varphi+i \, \mbox{Im} \left(\overline{\varphi} \,D_1 \varphi \right) \varphi.
\end{split}
\end{align}
Thus it follows from \eqref{v1} and \eqref{v2} that
\begin{align} \label{dvt}
\begin{split}
\frac{d}{dt}V_{\chi_R}[\varphi(t)]&= \mbox{Im} \int_{\R^2} \left(\overline{D_t\varphi} \, D_1 \varphi + \overline{\varphi} \, D_1D_t \varphi \right) \partial_1 \xi + \left(\overline{D_t\varphi} \, D_2 \varphi + \overline{\varphi} \,D_2D_t \varphi \right) \partial_2 \xi \, dx \\
& \quad + \mbox{Im} \int_{\R^2} \overline{\varphi} \left(\,D_1 \varphi  \,\partial_2 \xi  - D_2 \varphi \, \partial_1 \xi \right) |\varphi|^2\, dx \\
&:=I_1+I_2+\mbox{Im} \int_{\R^2} \overline{\varphi} \left(\,D_1 \varphi  \,\partial_2 \xi  - D_2 \varphi \, \partial_1 \xi \right) |\varphi|^2\, dx,
\end{split}
\end{align}
where we defined that
\begin{align*}
I_1:=\mbox{Im} \int_{\R^2} \left(\overline{D_ t\varphi} \, D_1 \varphi + \overline{\varphi} \, D_1D_t \varphi \right) \partial_1 \xi \,dx, \quad
I_2:=\mbox{Im} \int_{\R^2} \left(\overline{D_t\varphi} \, D_2 \varphi + \overline{\varphi} \,D_2D_t \varphi \right) \partial_2 \xi \, dx.
\end{align*}

In the following, we shall compute $I_1$, and $I_2$ can handled by a similar way. To do this, we decompose $I_1$ into two parts $I_{1,1}$ and $I_{1, 2}$, where
$$
I_{1,1}:=\mbox{Im} \int_{\R^2} \overline{D_t\varphi} \, D_1 \varphi \, \partial_1 \xi \, dx,
\quad I_{1,2}:=\mbox{Im} \int_{\R^2}\overline{\varphi} \, D_1D_t \varphi \, \partial_1 \xi \, dx.
$$
Firstly, let us deal with $I_{1,1}$. Recall that $\varphi$ satisfies the system \eqref{sys1}, then
\begin{align*}
I_{1, 1}&=-\mbox{Im} \int_{\R^2} i \left( \lambda |\varphi|^{p-2} \overline{\varphi} + \overline{D_1D_1 \varphi} + \overline{D_2D_2 \varphi} \right) D_1 \varphi \,\partial_1 \xi \, dx  \\
&=-\mbox{Re} \int_{\R^2} \left( \lambda |\varphi|^{p-2} \overline{\varphi} + \overline{D_1D_1 \varphi} + \overline{D_2D_2 \varphi} \right) D_1 \varphi \,\partial_1 \xi \, dx.
\end{align*}
By the definitions of $D_j$ and the fact that $A_j$ are real-valued for $j=1, 2$, we have that
\begin{align} \label{i111}
\begin{split}
&-\mbox{Re} \int_{\R^2} \lambda |\varphi|^{p-2} \overline{\varphi}\, D_1 \varphi \, \partial_1 \xi \, dx
=-\mbox{Re} \int_{\R^2} \lambda |\varphi|^{p-2} \overline{\varphi} \, \left(\partial_1+iA_1\right) \varphi \, \partial_1 \xi \, dx\\
&=-\mbox{Re} \int_{\R^2} \lambda |\varphi|^{p-2} \overline{\varphi} \,\partial_1 \varphi \, \partial_1 \xi \, dx
=-\frac {\lambda}{2} \int_{\R^2} |\varphi|^{p-2} \,  \partial_1 \left(|\varphi|^2\right) \partial_1 \xi \, dx\\
&=-\frac{\lambda}{p}  \int_{\R^2} \partial_1 \left(|\varphi|^2\right)^{\frac p2} \partial_1 \xi \, dx
=\frac{\lambda}{p}  \int_{\R^2} |\varphi|^p \,\partial_1^2 \xi \, dx,
\end{split}
\end{align}
as well as
\begin{align} \label{i112}
\begin{split}
-\mbox{Re} \int_{\R^2} \overline{D_1D_1 \varphi} \,D_1 \varphi \,\partial_1 \xi \, dx
&= -\mbox{Re} \int_{\R^2} \left(\partial_1 -iA_1\right)\overline{D_1 \varphi} \, D_1 \varphi \, \partial_1 \xi \, dx \\
&=-\frac 12 \int_{\R^2} \partial_1\left(|D_1 \varphi|^2\right) \partial_1 \xi \, dx \\
&=\frac 12 \int_{\R^2}|D_1 \varphi|^2 \partial_1^2 \xi\,dx.
\end{split}
\end{align}
Furthermore,
\begin{align} \label{i113}
\begin{split}
&-\mbox{Re} \int_{\R^2} \overline{D_2D_2 \varphi} \, D_1 \varphi \, \partial_1 \xi \, dx
=-\mbox{Re} \int_{\R^2}\left(\partial_2-iA_2\right) \overline{D_2 \varphi} \, D_1 \varphi \, \partial_1 \xi \, dx\\
&=-\mbox{Re} \int_{\R^2}  \partial_2\overline{D_2 \varphi} \, D_1 \varphi \, \partial_1 \xi \, dx + \mbox{Re} \int_{\R^2}i A_2\overline{D_2 \varphi} \, D_1 \varphi \, \partial_1 \xi \, dx\\
&= \mbox{Re}\int_{\R^2} \overline{D_2 \varphi} \, D_2D_1 \varphi \, \partial_1 \xi \, dx + \mbox{Re} \int_{\R^2}\overline{D_2 \varphi} \, D_1 \varphi \, \partial_2\partial_1 \xi \, dx.
\end{split}
\end{align}
Observe that
$$
D_2D_1 \varphi=D_1D_2 \varphi + i \left(\partial_2 A_1-\partial_1 A_2\right) \varphi = D_1D_2 \varphi + \frac i2 |\varphi|^2 \varphi.
$$
Thus
\begin{align} \label{identity1}
\begin{split}
&\mbox{Re}\int_{\R^2} \overline{D_2 \varphi} \, D_2D_1 \varphi \, \partial_1 \xi \, dx
=\mbox{Re}\int_{\R^2} \overline{D_2 \varphi} \, D_1D_2 \varphi \, \partial_1 \xi \, dx - \frac 12\,\mbox{Im} \int_{\R^2} \overline{D_2 \varphi} \, |\varphi|^2 \varphi \, \partial_1 \xi \, dx \\
&=\mbox{Re}\int_{\R^2} \overline{D_2 \varphi} \, \left(\partial_1 + i A_1\right) D_2 \varphi \, \partial_1 \xi \, dx - \frac 12\,\mbox{Im} \int_{\R^2} \overline{D_2 \varphi} \, |\varphi|^2 \varphi \, \partial_1 \xi \, dx \\
&=\mbox{Re}\int_{\R^2} \overline{D_2 \varphi} \, \partial_1 D_2 \varphi \, \partial_1 \xi \, dx - \frac 12\,\mbox{Im} \int_{\R^2} \overline{D_2 \varphi} \, |\varphi|^2 \varphi \, \partial_1 \xi \, dx \\
&=-\frac 12 \int_{\R^2} |D_2 \varphi|^2 \, \partial_1^2 \xi \, dx - \frac 12\,\mbox{Im} \int_{\R^2} \overline{D_2 \varphi} \, |\varphi|^2 \varphi \, \partial_1 \xi \, dx.
\end{split}
\end{align}
By inserting the identity \eqref{identity1} into \eqref{i113}, then
\begin{align}  \nonumber
-\mbox{Re} \int_{\R^2} \overline{D_2D_2 \varphi} \, D_1 \varphi \, \partial_1 \xi \, dx
&=-\frac 12 \int_{\R^2} |D_2 \varphi|^2 \, \partial_1^2 \xi \, dx + \mbox{Re} \int_{\R^2}\overline{D_2 \varphi} \, D_1 \varphi \, \partial_2\partial_1 \xi \, dx \\ \label{i114}
&\quad - \frac 12\, \mbox{Im} \int_{\R^2} \overline{D_2 \varphi} \, |\varphi|^2 \varphi \, \partial_1 \xi \, dx.
\end{align}
Therefore, from \eqref{i111}, \eqref{i112} and \eqref{i114},
\begin{align} \label{i11}
\begin{split}
I_{1,1}&=\frac 12 \int_{\R^2} \left(|D_1 \varphi|^2-|D_2 \varphi|^2 \right) \, \partial_1^2 \xi \, dx + \mbox{Re} \int_{\R^2}\overline{D_2 \varphi} \, D_1 \varphi \, \partial_2\partial_1 \xi \, dx \\
&\quad +\frac{\lambda}{p}  \int_{\R^2} |u|^p \,\partial_1^2 \xi \, dx - \frac 12\, \mbox{Im} \int_{\R^2} \overline{D_2 \varphi} \, |\varphi|^2 \varphi \, \partial_1 \xi \, dx.
\end{split}
\end{align}
Secondly, we shall treat $I_{1,2}$. Notice that $\varphi$ satisfies the system \eqref{sys1}, then
\begin{align*}
I_{1,2}&=\mbox{Im} \int_{\R^2} i\, \overline{\varphi} \,D_1 \left(\lambda |\varphi|^{p-2}\varphi + D_1D_1 \varphi + D_2D_2\varphi\right) \partial_1 \xi \, dx\\
&=\mbox{Re}\int_{\R^2} \overline{\varphi} \,D_1\left(\lambda |\varphi|^{p-2}\varphi + D_1D_1 \varphi + D_2D_2\varphi\right) \partial_1 \xi \, dx.
\end{align*}
Applying again the definitions of $D_j$ and the fact that $A_j$ are real-valued for $j=1, 2$, we get that
\begin{align}  \label{i121}
\begin{split}
&\mbox{Re}\int_{\R^2} \lambda \,\overline{\varphi} \,D_1 \left(|\varphi|^{p-2}\varphi\right) \partial_1 \xi \, dx
= \mbox{Re} \int_{\R^2} \lambda \,\overline{\varphi} \, \partial_1 \left(|\varphi|^{p-2}\varphi \right) \partial_1 \xi \, dx \\
&= -\mbox{Re} \int_{\R^2} \lambda \,\partial_1 \overline{\varphi} \,|\varphi|^{p-2} \varphi \, \partial_1 \xi \, dx
-\lambda\int_{\R^2} |\varphi|^p \partial_1^2 \xi \, dx \\
&= -\frac {\lambda}{2} \int_{\R^2} \partial_1\left(|\varphi|^2\right) |\varphi|^{p-2} \, \partial_1 \xi\, dx
-\lambda\int_{\R^2} |\varphi|^p \,\partial_1^2 \xi \, dx \\
&=\lambda\left(\frac 1p -1 \right) \int_{\R^2}|\varphi|^p \, \partial_1^2 \xi\, dx,
\end{split}
\end{align}
and
\begin{align} \label{identity2}
\begin{split}
&\mbox{Re}\int_{\R^2} \overline{\varphi} \,D_1D_1D_1 \varphi \, \partial_1 \xi \, dx
=\mbox{Re}\int_{\R^2} \overline{\varphi} \, \partial_1 D_1D_1 \varphi \, \partial_1 \xi \, dx
+\mbox{Re}\int_{\R^2} \overline{\varphi} \, iA_1 D_1D_1 \varphi \, \partial_1 \xi \, dx  \\
&=-\mbox{Re}\int_{\R^2} \partial_1 \overline{\varphi} \, D_1D_1 \varphi \, \partial_1 \xi \, dx
+\mbox{Re}\int_{\R^2} iA_1 \overline{\varphi} \,D_1D_1 \varphi \, \partial_1 \xi \, dx
-\mbox{Re}\int_{\R^2}\overline{\varphi} \, D_1D_1 \varphi \, \partial_1^2 \xi \, dx \\
&=-\mbox{Re}\int_{\R^2} \overline{D_1\varphi} \, D_1D_1 \varphi \, \partial_1 \xi \, dx
-\mbox{Re}\int_{\R^2}\overline{\varphi} \, D_1D_1 \varphi \, \partial_1^2 \xi \, dx \\
&=-\mbox{Re}\int_{\R^2} D_1\varphi \, \overline{D_1D_1 \varphi} \, \partial_1 \xi \, dx
-\mbox{Re}\int_{\R^2}\overline{\varphi} \, D_1D_1 \varphi \, \partial_1^2 \xi \, dx \\
&=-\mbox{Re}\int_{\R^2} D_1 {\varphi} \, \overline{D_1D_1 \varphi} \, \partial_1 \xi \, dx
-\mbox{Re}\int_{\R^2}\overline{\varphi} \, \partial_1 D_1 \varphi \, \partial_1^2 \xi \, dx -\mbox{Re}\int_{\R^2}iA_1\,\overline{\varphi} \, D_1 \varphi \, \partial_1^2 \xi \, dx \\
&=-\mbox{Re}\int_{\R^2} D_1 {\varphi} \, \overline{D_1D_1 \varphi} \, \partial_1 \xi \, dx
+ \int_{\R^2}|D_1 \varphi|^2 \partial_1^2 \xi \, dx
+ \mbox{Re}\int_{\R^2} \overline{\varphi} \, D_1 \varphi \, \partial_1^3 \xi \, dx \\
&=-\mbox{Re}\int_{\R^2} D_1 {\varphi} \, \overline{D_1D_1 \varphi} \, \partial_1 \xi \, dx
+ \int_{\R^2}|D_1 \varphi|^2 \partial_1^2 \xi \, dx
-\frac 12 \int_{\R^2} |\varphi|^2 \partial_1^4 \xi \, dx.
\end{split}
\end{align}
Plugging \eqref{i112} into the identity \eqref{identity2}, we then obtain that
\begin{align} \label{i122}
\mbox{Re}\int_{\R^2} \overline{\varphi} \,D_1D_1D_1 \varphi \, \partial_1 \xi \, dx
=\frac 32 \int_{\R^2}|D_1 \varphi|^2 \partial_1^2 \xi \, dx
-\frac 12 \int_{\R^2} |\varphi|^2 \partial_1^4 \xi \, dx.
\end{align}
In addition, observe that
\begin{align} \label{identity3}
\begin{split}
&\mbox{Re}\int_{\R^2} \overline{\varphi} D_1 D_2D_2 \varphi \, \partial_1 \xi \, dx
=\mbox{Re}\int_{\R^2} \overline{\varphi} \, \partial_1 D_2D_2 \varphi \, \partial_1 \xi \, dx
+\mbox{Re}\int_{\R^2}  iA_1 \,\overline{\varphi} \,D_2D_2 \varphi \, \partial_1 \xi \, dx  \\
&=- \mbox{Re}\int_{\R^2} \overline{ D_1 \varphi} \, D_2D_2 \varphi \, \partial_1 \xi \, dx
-\mbox{Re}\int_{\R^2} \overline{\varphi} \, D_2D_2 \varphi \, \partial_1^2 \xi \, dx  \\
&=- \mbox{Re}\int_{\R^2} D_1 \varphi \, \overline{D_2D_2 \varphi} \, \partial_1 \xi \, dx
-\mbox{Re}\int_{\R^2} \overline{\varphi} \, D_2D_2 \varphi \, \partial_1^2 \xi \, dx  \\
&=- \mbox{Re}\int_{\R^2} D_1 \varphi \, \overline{ D_2D_2 \varphi} \, \partial_1 \xi \, dx
+\int_{\R^2} |D_2 \varphi|^2 \, \partial_1^2 \xi \, dx
+ \mbox{Re} \int_{\R^2} \overline{\varphi}\,D_2 \varphi \, \partial_2 \partial_1^2 \xi \, dx \\
&=- \mbox{Re}\int_{\R^2} D_1 \varphi \, \overline{D_2D_2 \varphi} \, \partial_1 \xi \, dx
+\int_{\R^2} |D_2 \varphi|^2 \, \partial_1^2 \xi \, dx
-\frac 12 \int_{\R^2}|\varphi|^2 \, \partial_2^2 \partial_1^2 \xi\, dx.
\end{split}
\end{align}
By inserting \eqref{i114} into the identity \eqref{identity3}, then
\begin{align*}
\mbox{Re}\int_{\R^2} \overline{\varphi} D_1 D_2D_2 \varphi \, \partial_1 \xi\, dx
&=\frac 12 \int_{\R^2} |D_2 \varphi|^2 \, \partial_1^2 \xi \, dx + \mbox{Re} \int_{\R^2}\overline{D_2 \varphi} \, D_1 \varphi \, \partial_2\partial_1 \xi \, dx \\
&\quad -\frac 12 \int_{\R^2}|\varphi|^2 \, \partial_2^2 \partial_1^2 \xi \, dx - \frac 12\,\mbox{Im} \int_{\R^2} \overline{D_2 \varphi} \, |\varphi|^2 \varphi \, \partial_1 \xi \, dx.
\end{align*}
This along with \eqref{i121} and \eqref{i122} implies that
\begin{align} \label{i12}
\begin{split}
I_{1,2}&=\frac 12 \int_{\R^2}\left(3|D_1 \varphi|^2 + |D_2 \varphi|^2 \right) \, \partial_1^2 \xi \, dx
+\mbox{Re} \int_{\R^2}\overline{D_2 \varphi} \, D_1 \varphi \, \partial_2\partial_1 \xi \, dx \\
&\quad +\lambda\left(\frac 1p -1 \right) \int_{\R^2}|\varphi|^p \, \partial_1^2 \xi\, dx-\frac 12 \int_{\R^2} |\varphi|^2 \left(\partial_1^4 \xi +  \partial_2^2 \partial_1^2 \xi \right)\, dx\\
&\quad - \frac 12\,\mbox{Im} \int_{\R^2} \overline{D_2 \varphi} \, |\varphi|^2 \varphi \, \partial_1 \xi \, dx.
\end{split}
\end{align}
Recalling that $I_1=I_{1,1}+I_{1,2}$, by means of \eqref{i11} and \eqref{i12}, we then arrive at
\begin{align*}
I_1&=2\int_{\R^2}|D_1 \varphi|^2\, \partial_1^2 \xi \, dx
+2 \,\mbox{Re} \int_{\R^2}\overline{D_2 \varphi} \, D_1 \varphi \, \partial_2\partial_1 \xi \, dx-\frac{(p-2)\lambda}{p}\int_{\R^2}|\varphi|^p \, \partial_1^2 \xi \, dx \\
&\quad -\frac 12 \int_{\R^2} |\varphi|^2 \left(\partial_1^4 \xi +  \partial_2^2 \partial_1^2 \xi \right)\, dx-\mbox{Im} \int_{\R^2} \overline{D_2 \varphi} \, |\varphi|^2 \varphi \, \partial_1 \xi \, dx.
\end{align*}
Similarly, we can deduce that
\begin{align*}
I_2&=2\int_{\R^2}|D_2 \varphi|^2\, \partial_2^2 \xi \, dx
+2\,\mbox{Re} \int_{\R^2}\overline{D_1 \varphi} \, D_2 \varphi \, \partial_2\partial_1 \xi \, dx -\frac{(p-2)\lambda}{p}\int_{\R^2}|\varphi|^p \, \partial_2^2 \xi\, dx\\
&\quad -\frac 12 \int_{\R^2} |\varphi|^2 \left(\partial_2^4 \xi +  \partial_1^2 \partial_2^2 \xi \right)\, dx +\mbox{Im} \int_{\R^2} \overline{D_1 \varphi} \, |\varphi|^2 \varphi \, \partial_2 \xi \, dx.
\end{align*}
Consequently, it follows from \eqref{dvt} that
\begin{align*}
\frac{d}{dt}V_{\xi}[\varphi(t)] &= 2\int_{\R^2}\left(|D_1 \varphi|^2\, \partial_1^2 \xi
+2 \, \mbox{Re}\,\overline{D_1 \varphi} \, D_2 \varphi \, \partial_2\partial_1 \xi + |D_2 \varphi|^2\, \partial_2^2 \xi \right)\,dx \\
&\quad -\frac{\lambda(p-2)}{p}\int_{\R^2}|\varphi|^p \left(\partial_1^2 \xi + \partial_2^2 \xi\right)\, dx \\
& \quad -\frac 12 \int_{\R^2} |\varphi|^2 \left(\partial_1^4 \xi +  2 \partial_1^2 \partial_2^2 \xi + \partial_2^4 \xi\right)\, dx,
\end{align*}
and the proof is completed.
\end{proof}

\begin{ack}
The work was supported by the National Science Foundation of China (11771428) and the Postdoctoral Science Foundation of China (Y890125G21). The authors would like to thank Prof. Louis Jeanjean and Dr. Quanguo Zhang for the valuable suggestions and comments which help to improve the manuscript.
\end{ack}

\end{document}